\documentclass[11pt]{article}

\usepackage{graphicx}
\usepackage{amsmath}
\usepackage{amsthm}
\usepackage{amsxtra}
\usepackage{amsfonts,amssymb}
\usepackage{srcltx}
\usepackage{latexsym}
\usepackage{color}

\newtheorem{Th}{Theorem}
\newtheorem{Prop}[Th]{Proposition}
\newtheorem{Lm}[Th]{Lemma}
\newtheorem{Co}[Th]{Corollary}

\theoremstyle{definition}
\newtheorem{Def}[Th]{Definition}

\newtheorem{Rem}{Remark}

\begin{title}
 {\bf Classification of indecomposable states on the infinite symmetric inverse semigroup invariant under the infinite symmetric group. Semifinite case.}
\end{title}
\author{ {\bf Artem Dudko} and {\bf  Nikolay I. Nessonov} \\
                    IMPAN, Warsaw, Poland, and \\ B. Verkin ILTPE NAS, Kharkiv, Ukraine }

\date{}

\begin{document}

\maketitle

\begin{abstract}
Let  $\mathbb{N}$ be a set of the natural numbers.
 Symmetric inverse semigroup $R_\infty$ is the semigroup of all infinite 0-1 matrices $\left[ g_{ij}\right]_{i,\j\in \mathbb{N}}$ with  at most one 1 in each  row and each column such that $g_{ii}=1$ on the complement of a finite set. The binary operation in  $R_\infty$ is the ordinary matrix multiplication.  It is clear that infinite symmetric group $\mathfrak{S}_\infty$ is a subgroup of  $R_\infty$. The map $\star:\left[ g_{ij}\right]\mapsto\left[ g_{ji}\right]$ is an involution on $R_\infty$. We call a function $f$ on  $R_\infty$ positive definite if for all $r_1, r_2, \ldots, r_n\in R_\infty$  the matrix $\left[ f\left( r_ir_j^\star\right)\right]$ is Hermitian and positive semi-definite. A function $f$ said to be  indecomposable if the corresponding $\ast$-representation $\pi_f$ is  a factor-representation. A class of the $\mathfrak{S}_\infty$-invariant functions is defined by the  condition  $f(rs)=f(sr)$ for all $r\in R_\infty$ and $s\in\mathfrak{S}_\infty$. In this paper we classify all  semifinite  factor-representations of $R_\infty$ that correspond to the $\mathfrak{S}_\infty$-invariant positive definite functions.
\end{abstract}


\tableofcontents

\section{Introduction}
In the field of group representations, one of the most important questions is a description of "elementary" representations of the group, constructing "natural" representations, and decomposing them into elementary ones. Usually, for infinite-dimensional groups the role of elementary representations is played by either factor representations or irreducible representations, satisfying some natural restrictions.

One of the most interesting classification results in this area is a description of finite type factor representations of the infinite symmetric group $\mathfrak{S}_\infty$ ({\rm i.e.} the group of finitary permutation of the set $\mathbb N$ of positive integers). It is known that finite type representations of a group are in a correspondence with characters on the group (positive-definite functions constant on conjugacy classes). In 1964, Thoma \cite{Thoma} using analytic tools showed that indecomposable characters (extreme points in the set of characters) on $\mathfrak{S}_\infty$ are parameterized by two sequences $\alpha_1\geqslant\alpha_2\geqslant\ldots\geqslant 0$, $\beta_1\geqslant\beta_2\geqslant\ldots\geqslant 0$ with $\sum\alpha_i+\sum\beta_j\leqslant 1$. Later, many remarkable interpretations of these parameters and important relations of factor representations of $\mathfrak S_\infty$ with other areas were found.

Notably, Vershik and Kerov $\cite{VK1}$ developed asymptotic character theory for infinite-dimensional groups using ergodic approach and applied it to study characters of $\mathfrak S_\infty$ as limits of characters of finite symmetric groups $\mathfrak S_n$. In particular, they showed that parameters $\{\alpha_i\}$ and $\{\beta_j\}$ are limits of normalized rows and columns of Young diagrams associated to irreducible representations of $\mathfrak{S}_n$.

Later, Okounkov and Olshanski \cite{Ok2}, \cite{Olsh_tamesym} developed a so-called semigroup approach to study representations of $\mathfrak S_\infty$ and related groups. It was discovered that parameters $\{\alpha_i\}$ and $\{\beta_j\}$ are spectral values of weak limits of operators of factor-representations associated to transpositions (see eg. \cite{Ok2}). Namely, given a finite type representation $\pi$ of $\mathfrak{S}_\infty$ for every $k\in\mathbb N$ there exists weak limit $\mathcal O_k=\lim\limits_{l\to\infty}\pi((k,l))$, where $(k,l)\in\mathfrak{S}_\infty$ is the transposition of $k$ and $l$. If $\pi$ is a factor representation, then the spectrum of $\mathcal O_k$ is equal to $\{\alpha_i\}\cup\{\beta_j\}$ with, possibly, zero added.

In addition, many remarkable relations were discovered between characters and representations of $S(\infty)$ and various branches of math, including Combinatorics, Ergodic Theory, Probability Theory, Random Matrices, and Operator Algebras (see, for instance, \cite{BorodinOkounkovOlschanskii00}, \cite{Ok00}, \cite{V12}).

One of the most natural semigroup extensions of $\mathfrak{S}_\infty$ is infinite symmetric inverse semigroup $R_\infty$. It can be viewed as the group of infinite matrices $\{g_{ij}\}_{i,j\in\mathbb N}$ with $0,1$ entries such that $g_{ij}\neq \delta_{ij}$ only for finitely many pairs $i,j\in\mathbb N$. It appears naturally in the representation theory of $\mathfrak{S}_\infty$ as follows. Notice that operators $\mathcal O_k$ defined above need not to be invertible, are pairwise commuting and are naturally permuted by the action of $\mathfrak S_\infty$. Fix an $\mathfrak S_\infty$-equivariant collection $E_k$ of spectral projections of $\mathcal O_k$. Let $\epsilon_k\in R_\infty$ be the matrix obtained from the identity matrix by replacing the entry at $i$th row and $i$th column with zero. Then the assignment $\pi(\epsilon_k)=E_k$ uniquely extends the representation $\pi$ to a representation of $R_\infty$.

The description of indecomposable characters on $R_\infty$ was given by Vershik and Nikitin in \cite{VN} and, independently, by the second author of the present paper \cite{N_INV}. It turned out that any indecomposable character on $R_\infty$ is determined by the Thoma parameters of its restriction to $\mathfrak S_\infty$ and an additional parameter $\rho$ equal to either zero or one of the Thoma parameters $\alpha_i$.

In this paper, we consider a natural generatlization of a notion of a character on $R_\infty$. Namely, we study normalized positive-definite functions $f$ on $\mathfrak S_\infty$ which are preserved under the adjoint action only by elements of $\mathfrak S_\infty$: $f(sgs^{-1})=f(g)$ for all $g\in R_\infty,s\in\mathfrak S_\infty$. Such functions lead to a larger class of factor representations of $R_\infty$ compared to previously studied. A principle difference from the case of characters on $R_\infty$ is that for the GNS representation $\pi_f$ associated to a function $f$ as above the Okounkov operator $\mathcal O_k=\lim\limits_{l\to\infty}\pi((k,l))$ (weak operator limit), $k\in\mathbb N$, does not need to commute with the operator $\pi_f(\epsilon_k)$.

Our main result is a complete description of indecomposable positive-definite $\mathfrak S_\infty$-central functions $f$ producing semi-finite factor-representations. In particular, these include examples of type ${\rm I}_\infty$ and ${\rm II}_\infty$ representations of $R_\infty$. In the latter case obtained representations are parameterized by Thoma parameters $\{\alpha_i\},\{\beta_i\}$, the choice  of one element from the collection $\{\alpha_i\}$ (appearing also in ${\rm II}_1$ case), and an additional new parameter $t\in (0,1)$.

We notice that closely related to $R_\infty$ are the semidirect products $\mathfrak{S}_\infty\ltimes \Gamma^\infty$, where $\Gamma$ is  arbitrary group. Using a technique developed by A. Okounkov (\cite{Ok1}, \cite{Ok2}) and G. Olshansky \cite{O2} for studying representations of $\mathfrak{S}_\infty$ , the authors of the present paper obtained a series of classification results for $\mathfrak{S}_\infty\ltimes \Gamma^\infty$ (\cite{Dudko_Nes_2009}, \cite{Dudko_Nes_2008_Sb}, \cite{{Dudko_Nes}}). In \cite{Dudko_Nes_2008_Sb}, \cite{{Dudko_Nes}} they considered ${\rm II}_1$ factor representations. In \cite{Dudko_Nes_2009} they obtained new examples of type  ${\rm II}_\infty$ and type ${\rm III}$ factor representations of $S(\infty)\ltimes\Gamma^\infty$. In the present paper, we use the techniques from above mentioned papers as well as new ideas and methods.

The paper is organized as follows. In Section \ref{section:preliminaries}, we introduce the main concepts appearing in the paper and give necessary preliminaries on inverse semigroups, factor-representations, and positive-definite functions. In Section \ref{section:main}, we describe in detail the main results of the paper. Section \ref{section:admissible} is dedicated to an important class of $\mathfrak {S}_\infty $-admissible representation of $R_\infty$. In Section \ref{section:properties} we study properties of indecomposable $\mathfrak S_\infty$ central positive-definite functions on $R_\infty$ and obtain their description. In Section \ref{section:realizations}, we construct examples of associated factor representations and complete the proof of the main result. For the readers convenience, in Section \ref{section:notations}, we provide a list of main notations used in the paper.

\subsection{Acknowledgements.}  The authors acknowledge the funding by the “Long-term program of support of the Ukrainian research teams at the Polish Academy of Sciences carried out in collaboration with the U.S. National Academy of Sciences with the financial support of external partners”.

\section{Preliminaries}\label{section:preliminaries}
In this section we introduce in details the notions used in the present paper and formulate some auxiliary statements.
\subsection{Key notions and their properties}
The symmetric inverse semigroup on the $n$-element set $X_n = \{1, . . . , n\}$ ($X_\infty$ is the set of all
positive integers), which we denote by $R_n$ in what follows, consists of all partial bijections on $X_n$,
i.e. bijections between subsets of $X_n$. Following \cite{Munn}, we denote the domain and the range of a
bijection $r\in R_n$ by $\mathcal{D}(r)$ and $\mathcal{I}(r)$, respectively, so that $r(\mathcal{D}(r)) = \mathcal{I}(r)$. In the case $n =\infty$, we
assume that the complement of the set $\{x \in X_\infty : rx = x\}$ is finite for all $r \in R_\infty$.  Note that there exists an involution on $R_n$, which is denoted by $\star$. Namely, the domain of each partial bijection $r^\star$ is $\mathcal{I}(r)$
and if $r(d) = i\in \mathcal{I}(r)$, then $r^\star(i) = d \in \mathcal{D}(r)$. It is convenient to represent the elements of the semigroup $R_n$ in the form
of $n \times n$ matrices with entries $0$ and $1$: $r=\left[ r_{lk} \right]_{l,k=1}^n$, where  $r_{lk}=\left\{\begin{array}{rl}
 1,&\text{ if }r(k)=l\\
0,&\text{ otherwise. } \end{array}\right.$
In particular, if $ k\notin \mathcal{D}(r)$, then $r_{lk} = 0$ for all $l \in X_n$.
Multiplication in the semigroup corresponds to
usual matrix multiplication, and involution corresponds to matrix transposition. In what follows, it
is convenient to identify $R_n$ with this realization. The semigroup $R_n$ contains the symmetric group
$\mathfrak{S}_n$, which consists of all bijections of $X_n$ for $n <\infty$. A bijection $s: \mathbb{N}\to \mathbb{N}$ is called
{\it finite} if the set $\left\{ i\in\mathbb{N}\big|\,s(i) \neq i\right\}$
is finite. In the case $n =\infty$, the subgroup $\mathfrak{S}_\infty$ consists
of all finite bijections of $X_\infty$. In particular, $s = [s_{lk}] \in R_n$ belongs to $\mathfrak{S}_n$ if and only if the matrix
$\left[s_{lk} \right]$ is invertible. The semigroup $R_n$ contains the Abelian semigroup of identity partial bijections,
which is denoted by ${\rm Diag}_n$ and defined by the conditions
 \begin{eqnarray*}
 r\in {\rm Diag}_n\Leftrightarrow \mathcal{D}(r)=\mathcal{I}(r) \text{  end } r(x)=x \text{ for all } x\in \mathcal{D}(r).
 \end{eqnarray*}
 The bijection with empty domain corresponds to the zero matrix and is the zero of the semigroup.
In particular, the semigroup $R_\infty$ has no zero. We denote the identity
element of the semigroup $R_\infty$ by $e$. Set $R_{n\infty}=\left\{r\in R_\infty:r(x)=x \text{ for all } x=1,2,\ldots,n \right\}$.\label{R_n_infty} In particular, this means that
$\{1,2,\ldots,n\}\subset \mathcal{D}(r)$ for all $r\in R_{n\infty}$. Take subset $\mathbb{A}\subset X_n$ and denote by $\epsilon_\mathbb{A}$ \label{epsilon_mathbb_A} element of ${\rm Diag}_n$ such that $\mathcal{D}(\epsilon_\mathbb{A})=X_n\setminus\mathbb{A}$.

 Let $\mathcal{B}(\mathcal{H})$ be the algebra of all bounded operators on a Hilbert space $\mathcal{H}$ and $M$ be $W^*$-subalgebra in $\mathcal{B}(\mathcal{H})$. Denote by $M_*$ the Banach space of all weakly continuous linear functionals on $M$. From now on, $M_*^+$ stands  for the cone of all positive functionals from $M_*$.

  Let $\mathcal{S}$ be the subset in $\mathcal{B}(\mathcal{H})$. Set $\mathcal{S}^\prime=\left\{A\in\mathcal{B}(\mathcal{H}): A\,m\,=\,m\,A\;\text{ for all }\right.$ $\left.m\in \mathcal{S} \right\}$ and $\mathcal{S}^{\prime\prime}=\left( \mathcal{S}^\prime \right)^\prime$.

  We denote the identity
element of the semigroup $R_\infty$ by $e$ and the identity operator in $\mathcal{B}(\mathcal{H})$ by $I$. A $*$-representation
of the semigroup $R_\infty$ is a homomorphism $R_\infty\stackrel{\pi}{\mapsto} \mathcal{B}(\mathcal{H})$
 such that $\pi(r^{\star})=\left( \pi(r) \right)^*$,  where $\left( \pi(r) \right)^*$  is
the operator conjugate to $ \pi(r) $, and $ \pi(e)  = I$. Therefore, for $s\in \mathfrak{S}_\infty$, the operator $\pi(s)$ is unitary,
and for $d\in {\rm Diag}_\infty$, the operator $\pi(d)$ is a self-adjoint projection.

Let us define the semigroup algebra $\mathbb{C}[R_\infty]$ as the set  of all linear combinations of finitely many elements of $R_\infty$ with coefficients in $\mathbb{C}$, hence of all elements of the form
$\sum\limits_{r\in R_\infty} a(r)r$, 	 where $a(r)\in\mathbb{C}$ and $\#\left\{r:a(r)\neq0 \right\}<\infty$. Involution $\star:\left[ g_{ij}\right]\mapsto\left[ g_{ji}\right]$
on $R_\infty$ define a structure of $*$-algebra on  $\mathbb{C}[R_\infty]$ by $\left(\sum\limits_{r\in R_\infty} a(r)r\right)^\star=\sum\limits_{r\in R_\infty} a(r)r^\star$.

Recall that a complex valued function $f$ on $R_\infty$ is called positive definite if matrix $\left[f\left(  r_j^\star r_i \right) \right]$ is positive semidefinite for any finite collection $\left\{r_i \right\}\subset R_\infty$.  We will identify the function $f$ with the functional on $\mathbb{C}[R_\infty]$, which defined     as follows: $\sum\limits_{r\in R_\infty} a(r)r\mapsto \sum\limits_{r\in R_\infty} a(r)f(r)$. It is clear that
  \begin{eqnarray*}
  \sum\limits_{q,r\in R_\infty} \overline{a(q)}\,a(r)f(q^\star \,r)\geq 0.
  \end{eqnarray*}
  We call $f$ a state, if $f(e)=1$. If $H$ is a subset in $R_\infty$, then $H$-central or $H$-invariant state $f$ is defined by the condition
\begin{eqnarray*}
f(hr)=f(rh)\;\text{ for all }\;h\in H \;\text{ and for all}\;r\in R_\infty.
\end{eqnarray*}
Note that $R_\infty$-central states on $R_\infty$ are usual characters \cite{VN}, \cite{N_INV}.
\subsection{On Gelfand-Najmark-Segal representations of $R_\infty$}\label{Subsection: GNS}
Let $\left( \pi_f,\mathcal{H}_f,\xi_f \right)$ be  the GNS-representation of the semigroup $R_\infty$ (i. e., the representation obtained by applying the Gelfand-Naimark-Sigal construction); it acts on a Hilbert space $\mathcal{H}_f$ with cyclic vector $\xi_f$ such that $f(r)=\left(\pi_f(r)\xi_f,\xi_f\right)$ for all $r\in R_\infty$.

\begin{Def} We say that $f$ is a factor-state, or indecomposable, if the corresponding GNS-representation $\pi_f$ is a factor-representation. Denote by $\mathfrak{F}_\mathfrak{S}$ the set of all $\mathfrak{S}_\infty$-central indecomposable states on $R_\infty$\label{F_S}.
\end{Def}
\noindent We notice that a description of a different class of semifinite functions on $R_\infty$ can be obtained from the results of Nikitin-Safonkin \cite{NikSaf}. The principal difference of $\mathfrak{F}_{\mathfrak{S}}$ from the class of functions appearing in \cite{NikSaf} is that their semifinite functions take  finite nonzero values on projectors from semigroup algebras of finite semigroups, while for $\mathfrak{F}_{\mathfrak{S}}$ the corresponding values are either zero or infinity.

Recall that representations $\pi_1$ and $\pi_2$ of the semigroup $R_\infty$ are said to be quasi-equivalent if there exists an isomorphism $\theta$: $\pi_1(R_\infty)^{\prime\prime}\mapsto \pi_2(R_\infty)^{\prime\prime}$ such that $\theta\left( \pi_1(r)\right)=\pi_2(r)$ for all $r\in R_\infty$. The following assertion is valid.
\begin{Prop}\label{quasi_unique}
Suppose that GNS-representations $\left( \pi_\varphi,\mathcal{H}_\varphi,\xi_\varphi \right)$ and  $\left( \pi_\psi,\mathcal{H}_\psi,\xi_\psi \right)$ are quasi-equivalent. If  $\varphi,\psi\in\mathfrak{F}_\mathfrak{S}$ then $\varphi=\psi$.
\end{Prop}
\begin{proof}
Take an arbitrary $r\in R_\infty$. Set $\mathcal{C}(r)=\left\{srs^{-1} \right\}_{s\in\mathfrak{S}_\infty}$. Clearly,
\begin{eqnarray}\label{S_inv}
\varphi(r)=\varphi(r_n) \text{ and }
\psi(r)=\psi(r_n) \text{ for all } r_n\in R_{n\infty}\cap\mathcal{C}(r).
\end{eqnarray}
 Since   $\left( R_{n\infty}\cap\mathcal{C}(r) \right)\neq\emptyset$ for each $n$, there exists sequence $\left\{r_n \right\}_{n\in\mathbb{N}}$ such that
 $r_n\in R_{n\infty}\cap\mathcal{C}(r)$.

 Take any $g,h\in R_n\subset R_\infty$.
 If $n<N$ then  $r_N=s_{nN}\;r_n\;s_{nN}^{-1}$ for some $s_{nN}\in R_{n\infty}\cap\mathfrak{S}_\infty$. It follows from this that
 \begin{eqnarray*}
 \left( \pi_\varphi(r_{n_1})\pi_\varphi(g)\xi_\varphi,\pi_\varphi(h)\xi_\varphi \right)=\left( \pi_\varphi(r_{n_2})\pi_\varphi(g)\xi_\varphi,\pi_\varphi(h)\xi_\varphi \right)
\text{ for all } n_1,n_2>n.
 \end{eqnarray*}
Therefore, $\pi_\varphi(r_n)$ is weakly convergent. Let $$w-\lim\limits_{n\to\infty}\pi_\varphi(r_n)=A(r)\in\pi_\varphi(R_\infty)^{\prime\prime}$$
be the corresponding limit in the weak operator topology. Clearly, $$A(r)\in\bigcap\limits_{n=1}^\infty\pi_\varphi(R_n)^\prime=\pi_\varphi(R_\infty)^{\prime}.$$
Hence we obtain that $A(r)$ is a scalar operator; i.e. $A(r)=c(r)I$, where $c(r)\in \mathbb{C}$. Now, applying (\ref{S_inv}), we have
\begin{eqnarray}\label{state_on_scalar_operator}
\varphi(r)=\left(\pi_\varphi(r)\xi_\varphi,\xi_\varphi\right)=c(r).
\end{eqnarray}
Since isomorphism $ \pi_\varphi(R_\infty)^{\prime\prime}\ni\pi_\varphi(r)\stackrel{\theta}{\mapsto}\pi_\psi(r)
\in\pi_\psi(R_\infty)^{\prime\prime}$ is continuous in the weak operator topology, we claim that
\begin{eqnarray*}
w-\lim\limits_{n\to\infty}\pi_\psi\left( r_n \right)=w-\lim\limits_{n\to\infty}\theta(\pi_\varphi(r_n))  =\theta(A(r))=c(r)I.
\end{eqnarray*}
Hence, using the equality $\psi(r)\stackrel{(\ref{S_inv})}{=}\left(\pi_\psi(r_n)\xi_\psi,\xi_\psi\right)$, we conclude from (\ref{state_on_scalar_operator}) that $c(r)=\psi(r)=\varphi(r)$.
\end{proof}

 If $\left( \pi_\varphi,\mathcal{H}_\varphi,\xi_\varphi \right)$ is GNS-representation, corresponding to positive definite function $\varphi$ on $R_\infty$ then $\varphi(r) = \left(\pi_\varphi(r)\xi_\varphi,\xi_\varphi\right)$ for all $r\in R_\infty$.     Given vector $\eta$ in Hilbert space $\mathcal{H}_\varphi$, let $\omega_\eta$ be the vector functional on $\pi_\varphi (R_\infty)''$ defined by $\omega_\eta(A)=\left(A\eta,\eta\right)$, $A\in \pi_\varphi (R_\infty)''$. For abbreviation, we use the same letter $\varphi$ for $\omega_{\xi_{\varphi}}$.
\begin{Prop}\label{1 or 2}
Let $\varphi\in\mathfrak{F}_\mathfrak{S}$. Suppose that the factor $\pi_\varphi(R_\infty)''$ has type ${\rm I}_\infty$ or ${\rm II}_\infty$. If ${\rm Tr}$ is a normal trace on $\pi_\varphi(R_\infty)''$ then there exist a positive number $\kappa$ and an orthogonal projection $F_\varphi \in \pi_\varphi(R_\infty)''$ such that $\varphi(A)=\kappa {\rm Tr}(F_\varphi\,A)$ for all $A\in\pi_\varphi(R_\infty)''$.
\end{Prop}
\begin{proof}
For the sake of convenience, we will denote $\pi_\varphi(R_\infty)''$ by $M$.
Without loss of generality we may assume that $\pi_\varphi$ acts in Hilbert space $\mathcal{H}_\varphi=L^2(M,{\rm Tr})$ by left multiplication. Let $M_+$ be the cone of the nonnegative operators in factor $M$, and let  $\overline{M}_+$ be a closure   of $$\mathfrak{J}=M_+\cap L^2(M,\rm Tr)$$ in $\mathcal{H}_\varphi$. Then there exists a unique $\xi\in \overline{M}_+$ such that $\varphi(A)=(A\xi,\xi)_{\mathcal{H}_\varphi}$ for all $A\in M$ and $\|\xi\|_{\mathcal{H}_\varphi}=1$ (see \cite{TAKES}, Chapter {\rm IX}, Theorem 1.2). Now we  take  $a\in M_+$ such that
\begin{eqnarray}\label{epsilon_estimate}
\|\xi-a\|_{\mathcal{H}_\varphi}<1/2.
\end{eqnarray}
Set $a_j=(\# \mathfrak{S}_j)^{-1}\,\sum\limits_{s\in\mathfrak{S}_j}\,\pi_\varphi(s)\,a\,\pi_\varphi(s)^*$. Since the ball \- $B_{\|a\|}=\left\{b\in M: \|b\|\leq\|a\| \right\}$ is weakly compact in $M$ and sequence $\left\{ a_j\right\}\subset B_{\|a\|}$, there exists $\lim\limits_{k\to\infty}a_{j_k}=\tilde a\in B_{\|a\|}$ in the week operator topology for some subsequence $j_k$ and
 \begin{eqnarray}\label{inv}
 \pi_\varphi(s)\,\tilde a=\tilde a \pi_\varphi(s) ~\text{ for all }~ s\in\mathfrak{S}_\infty.
  \end{eqnarray}
  On the other hand the mapping, acting by $m\ni M\stackrel{P_j}{\mapsto} (\# \mathfrak{S}_j)^{-1}\,\sum\limits_{s\in\mathfrak{S}_j}\,\pi_\varphi(s)\,m\,\pi_\varphi(s)^*$, extends to an orthogonal projection in $\mathcal{H}_\varphi$. Since, by definition, $P_{n+1}\leq P_n$, the sequence $\{P_n\}$ converges in the strong operator topology to some orthogonal projection $P$.
Let us prove that $Pa=\tilde a$.

Since $\lim\limits_{j\to\infty}\|a_j-Pa\|_{L^2(M,{\rm Tr})}=0$,  we have
\begin{eqnarray}\label{L^2 equality}
\lim\limits_{j\to\infty}\left(a_j,h^2\right)_{L^2(M,{\rm Tr})}=\lim\limits_{j\to\infty}{\rm Tr}(h^2a_j)=\left( Pa,h^2\right)_{L^2(M,{\rm Tr})} ~\text{for all }~ h\in M_+\cap\mathfrak{I}.
\end{eqnarray}
On the other hand
\begin{eqnarray*}
\left(\tilde a h,h\right)_{L^2(M,{\rm Tr})}={\rm Tr}\left( h^2\tilde a\right)=\lim\limits_{p\to\infty}\left(a_{j_p}h,h\right)=\lim\limits_{p\to\infty}{\rm Tr}\left( h^2 a_{j_p}\right)\stackrel{\eqref{L^2 equality}}{=}\left( Pa,h^2\right)_{L^2(M,{\rm Tr})}.
\end{eqnarray*}
Thus
\begin{eqnarray*}
{\rm Tr}\left( h^2\tilde a\right)=\left( Pa,h^2\right)_{L^2(M,{\rm Tr})}.
\end{eqnarray*}
Hence, replacing in turn $h^2$ by $(x+y)^*(x+y)$, $(x-y)^*(x-y)$, $(x+iy)^*(x+iy)$, $(x-iy)^*(x-iy)$, where $x,y\in\mathfrak{I}$, and using the {\it polarization} identity, we obtain
\begin{eqnarray*}
{\rm Tr}(x^*y\tilde a)=\left(y\tilde a,x \right)_{L^2(M,{\rm Tr})}=\left(Pa,y^*x \right)_{L^2(M,{\rm Tr})} ~\text{ for all }~ x,y\in\mathfrak{I}.
\end{eqnarray*}
Thus $\left(y\tilde a,x \right)_{L^2(M,{\rm Tr})}=\left(y(Pa),x \right)_{L^2(M,{\rm Tr})}$ for all $x,y\in\mathfrak{I}$. It follows that $y\tilde a=y(Pa)$ for all $y\in\mathfrak{I}$, which implies that $\tilde a=Pa$.

 Further, using similar calculations as above one can show that $P\xi=\xi$. We obtain from \eqref{epsilon_estimate} that
\begin{eqnarray}
\left\|\xi-\tilde a \right\|_{L^2(M,{\rm Tr})}<\frac{1}{2}\Rightarrow \left\|\tilde a \right\|_{L^2(M,{\rm Tr})}>\frac{1}{2}.
\end{eqnarray} In particular $\tilde a\neq 0$. Since $\tilde{a}$ belongs to the commutant of $\pi_\varphi(S_\infty)$, the state $\psi(A)=\tfrac{1}{{\rm Tr}(\tilde a)}{\rm Tr}(\tilde{a}A)$ on $M=\pi_\varphi(R_\infty)''$ is $\mathfrak S_\infty$-invariant. Thus, $\psi\in\mathfrak{F}_\mathfrak{S}$. Observe that by construction the GNS-representation associated to $\psi$ is quasiequivalent to $\pi_\varphi$. By Proposition \ref{quasi_unique}, $\varphi=\psi$.

It remains to show that $b=\tfrac{\tilde a}{{\rm Tr}(\tilde a)}$ is an orthogonal projection. Define the positive functionals $\omega_1$, $\omega_2$ on $M$ by
\begin{eqnarray*}
\omega_1(m)={\rm Tr}(b\,m), \omega_2(m)={\rm Tr}(b^2\,m), m\in M.
\end{eqnarray*}
It follows from \eqref{inv} that $\omega_1(m\,\pi_\varphi(s))=\omega_1(\pi_\varphi(s)\,m)$ and $\omega_2(m\,\pi_\varphi(s))$ $=\omega_2(\pi_\varphi(s)\,m)$ for all $m\in M$, $s\in\mathfrak{S}_\infty$. By Proposition \ref{quasi_unique}, there exists a positive number $\vartheta$ such that $b^2=\vartheta b$. Since, by definition, $\|b\|=\|b^2\|=1$, we obtain that $b=b^2$.
\end{proof}

\subsection{Case of the factor of type ${\rm I}_\infty$}\label{Subsection: I_infty}
Let us  consider the case when $M=\pi_\varphi(R_\infty)''$ is ${\rm I}_\infty$-factor.

We recall that a permutation $s\in\mathfrak{S}_\infty$ is even if it is a product of an even number of transpositions and $s$ is odd if it is a product of an odd number of transpositions. Define the function  ${\rm sign}:$  $\mathfrak{S}_\infty \rightarrow \{-1,1\}$
  by ${\rm sign}(s)=\left\{\begin{array}{rl}
 -1,&\text{ if } s ~\text{ is odd};\\
1,&\text{ if  } s ~\text{ is even}.\end{array}\right.$

Denote $\pi_\varphi(R_\infty)''$ by $M$ and suppose that $\pi_\varphi$ acts in Hilbert space $\mathcal{H}_\varphi=L^2\left(M, {\rm Tr} \right)$ by the operators of left multiplication. Since $M$ is ${\rm I}_\infty$-factor, by Propositions \ref{quasi_unique} and \ref{1 or 2}, we have $\pi_\varphi(s)F_\varphi= F_\varphi\pi_\varphi(s)$ for all $s\in\mathfrak{S}_\infty$. It follows from the equality $1=\varphi(I)=\kappa {\rm Tr}(F_\varphi)$
that $w^*$-algebra $F_\varphi M F_\varphi$ is ${\rm I}_n$-factor, where $n<\infty$.  Thus  $w^*$-algebra $F_\varphi \pi_\varphi(\mathfrak{S}_\infty)''F_\varphi\subset F_\varphi M F_\varphi$ is also finite dimensional. Therefore, there exists orthogonal projection $E_0\leq F_\varphi$ from the  $w^*$-algebra $F_\varphi \pi_\varphi(\mathfrak{S}_\infty)''F_\varphi$  such that
\begin{eqnarray*}
F_\varphi\pi_\varphi(s)=E_0+({\rm sign}\,s)(E_1), ~\text{ where }~ E_1=F_\varphi-E_0, ~\text{ for all }~ s\in\mathfrak{S}_\infty\subset R_\infty.
\end{eqnarray*}
Hence, using the fact that $F_\varphi\in M$, we obtain that $E_0, E_1\in M$.

Denote by $\left[\mathfrak{A} \right]$ the minimal closed subspace in $\mathcal{H}_\varphi$ containing subset $\mathfrak{A}\subset \mathcal{H}_\varphi$. Let $\widetilde{E}_0$ be the orthogonal projection onto subspace $\left[\pi_\varphi \left({\rm Diag}_\infty \right)\,E_0\,\mathcal{H}_\varphi\right]$, and let $\widetilde{E}_1$ be the orthogonal projection onto  $\left[\pi_\varphi \left({\rm Diag}_\infty \right)\,E_1\,\mathcal{H}_\varphi\right]$. Then
\begin{eqnarray*}
\pi_\varphi(r)\widetilde{E}_0\,\mathcal{H}_\varphi\subset\widetilde{E}_0\,\mathcal{H}_\varphi~\text{ and}~
\pi_\varphi(r)\widetilde{E}_1\,\mathcal{H}_\varphi\subset\widetilde{E}_1\,\mathcal{H}_\varphi ~\text{ for all}~ r\in R_\infty.
\end{eqnarray*}
It follows from this that $\widetilde{E}_0$ and $\widetilde{E}_1$ belong to the center of $M$. Since $\widetilde{E}_0$ and $\widetilde{E}_1$ are mutually orthogonal and $M$ is factor, we obtain that one the projections  $\widetilde{E}_0$, $\widetilde{E}_1$ is  equal to zero.

$1)$ Assume first that $\widetilde E_1=0$. Then $F_\varphi=E_0$. By definitions of $E_0$  and $F_\varphi$ we have for every $s,t\in\mathfrak S_\infty$ and $r\in R_\infty$:
$$\varphi(srt)=\kappa {\rm Tr}(F_\varphi\pi_\varphi(srt))=\kappa {\rm Tr}(E_0\pi_\varphi(srt))=\kappa {\rm Tr}(E_0\pi_\varphi(r))=\varphi(r).$$

$2)$ Assume now that $\widetilde{E}_0=0$. Then $F_\varphi=E_1$. Therefore, $\widetilde{E}_1=I$. Let us prove that $\pi_\varphi\left( \epsilon_\mathbb{A}\right)E_1=0$ for all $\mathbb{A}\subset X_\infty$ such that $\#\mathbb{A}\geq 2$. Take the elements $p,q\in\mathbb{A}$ and denote by $(p\;\;q)$ the corresponding transposition. Then we have $\epsilon_\mathbb{A}\,(p\;\;q)=\epsilon_\mathbb{A}$. It follows from this that
\begin{eqnarray*}
\pi_\varphi(\epsilon_\mathbb{A})\,E_1=\pi_\varphi(\epsilon_\mathbb{A})\,\pi_\varphi((p\;\;q))\,E_1=-\pi_\varphi(\epsilon_\mathbb{A})\,E_1.
\end{eqnarray*}
Therefore,  $\pi_\varphi(\epsilon_\mathbb{A})\,E_1=0$. Hence we obtain that the subspaces $\left[\pi_{\epsilon_{\{j\}}}E_1\,\mathcal{H}_\varphi \right]$ are orthogonal for different $j$.  Since subspace $ \bigoplus\limits_{j=1}^\infty\left[\pi_{\epsilon_{\{j\}}}E_1\,\mathcal{H}_\varphi \right]$ is $\pi_\varphi(R_\infty)$-invariant and $E_1\in M$, only two cases are possible:
\begin{itemize}
  \item[\rm i)] $\bigoplus\limits_{j=1}^\infty\left[\pi_\varphi({\epsilon_{\{j\}}})E_1\,\mathcal{H}_\varphi \right]=\mathcal{H}_\varphi$;
  \item[\rm ii)] $ \bigoplus\limits_{j=1}^\infty\left[\pi_\varphi({\epsilon_{\{j\}}})\,E_1\,\mathcal{H}_\varphi \right]=0$.
\end{itemize}
Take any nonzero $\eta\in E_1\,\mathcal{H}_\varphi$. In the case {\rm i)} we have $\eta=\sum\limits_{j=1}^\infty\pi_\varphi({\epsilon_{\{j\}}})\eta$. Since $\left\|\pi_\varphi({\epsilon_{\{j\}}})\eta \right\|$ does not depend on $j$, we obtain that $\|\eta\|=\infty$. Thus $\eta=0$. Hence, $\pi_\varphi(\epsilon_{\{j\}})=0$ for all $j$.
\begin{Co}
If $\pi_\varphi$ is type ${\rm I}$ factor-representation of $R_\infty$ then only two cases are possible:
\begin{itemize}
  \item[\rm 1)] $\varphi$ is two-sided $\mathfrak{S}_\infty$-invariant; i. e. $\varphi(srt)=\varphi(r)$ for all $s,t\in\mathfrak{S}_\infty$, $r\in R_\infty$;
  \item[\rm 2)] if $r=s\,\epsilon_{\mathbb{A}}$ then $\varphi(r)=\left\{\begin{array}{rl}
 {\rm sign}\,s,&\text{ if } \mathbb{A}=\emptyset;\\
0,&\text{ if  } \mathbb{A}\neq\emptyset.\end{array}\right.$
\end{itemize}
\end{Co}

\section{The main results}\label{section:main}
The finite characters or $R_\infty$-central indecomposable states on $R_\infty$ were first classified by A. Vershik and P. Nikitin in \cite{VN}. Their approach  is based on the Vershik–Kerov ergodic method \cite{VK1}. Namely, they studied the sequences $\left\{\chi_n \right\}$ of the irreducible characters of the finite semigroups $\left\{R_n \right\}$, and  found a characterization of those sequences that weakly converge, as $n\to\infty$, to an indecomposable characters on $R_\infty$. The new proof of this result  suggested by N. Nessonov \cite{N_INV} was based on the multiplicativity property  of the indecomposable characters.
\subsection{An Analogue of Cycle Decomposition in $R_\infty$.}\label{subsec: cycle decomposition}
In $R_\infty$ there exists an analogue of the decomposition of elements of the group $\mathfrak{S}_\infty$ into products of disjoint cycles \cite{Munn} (see Theorem 1). It is constructed as follows.

For $r\in R_\infty$ and $a\in \mathcal{D}(r)\setminus\mathcal{I}(r)$, we set $$\mathfrak{m}(r,a)=\min\left\{ k\geq 1: r^k(a)\in  \mathcal{I}(r)\setminus\mathcal{D}(r)\right\}$$
 and $\mathfrak{O}_a=\left\{a,r(a),\ldots,r^{\mathfrak{m}(r,a)}(a)\right\}$. If $a,b\in  \mathcal{D}(r)\setminus\mathcal{I}(r)$ and $a\neq b$, then $\mathfrak{O}_a\cap\mathfrak{O}_b=\emptyset$. Consider the partial bijection $\;^a\!q$ defined by
  \begin{eqnarray*}
\,^a\!q(x) =
\left\{\begin{array}{rl}
 r(x),&\text{ if } x\in \mathfrak{O}_a\setminus r^{\mathfrak{m}(r,a)}(a)\\
x,&\text{ if  } x\in X_\infty\setminus  \mathfrak{O}_a.\end{array}\right.
 \end{eqnarray*}
Clearly, $\mathcal{D}\left( \,^a\!q \right)=X_\infty\setminus r^{\mathfrak{m}(r,a)}(a)$ and $\mathcal{I}\left( \,^a\!q \right)=X_\infty\setminus a$. If $\,^a\!c=\left( a\rightarrow r(a)\rightarrow\cdots\rightarrow r^{\mathfrak{m}(r,a)}(a)\rightarrow a\right)$ is a usual cycle and $\epsilon_{\{j\}}$ is an element of ${\rm Diag}_\infty$ with domain $X_\infty\setminus \{j\}$, then
\begin{eqnarray}\label{cycle_diag}
 \,^a\!q =\,^a\!c\;\epsilon_{\{r^{\mathfrak{m}(r,a)}(a)\}}.
\end{eqnarray}
If $a,b\in\mathcal{D}(r)\setminus\mathcal{I}(r)$, then the cycles $\,^a\!c$ and $\,^b\!c$ according to (\ref{cycle_diag}) are disjoint and have length at least $2$. In this case, we say that quasi-cycles $\,^a\!q$ and $\,^b\!q$ are disjoint as well.

 In what follows, we refer to partial bijections of the form $\,^a\!q $ as nontrivial {\it quasi-cycles}. It is convenient to regard the elements of ${\rm Diag}_\infty$, whose domains are obtained by removing one point from $X_\infty$, as (trivial) quasi-cycles as well.

 We define the support of an element $r\in R_\infty$ as
\begin{eqnarray}\label{supp(r)}
{\rm supp}\,r=\left\{x\in \mathcal{D}(r):r(x)\neq x\right\}\cup\left( X_\infty\setminus \mathcal{D}(r) \right).
\end{eqnarray}

Thus, we arrive at the following important assertion of \cite{Munn}.
\begin{Th}\label{decomposition_into_product_cycles}
Each $r\in R_\infty$ has a unique decomposition into a product of pairwise disjoint  quasi-cycles and usual  cycles of the form
\begin{eqnarray}\label{decomposition_of_r}
r=\left[ \prod\limits_{a\in\mathcal{D}(r)\setminus\mathcal{I}(r)} \,^a\!q \right]\left[ \prod\limits_{i}c_i^{(r)} \right] d^{(r)},
\end{eqnarray}
where $\;c_i^{(r)}$ are the usual nontrivial cycles and $d^{(r)}$ is the product of trivial quasi-cycles; i.e. $d^{(r)}=\prod\limits_{a\in X_\infty\setminus(\mathcal{D}(r)\cup \mathcal{I}(r))}\;\epsilon_{\{a\}}\in {\rm Diag}_\infty$.
\end{Th}
\begin{Rem}
By (\ref{decomposition_of_r}), we have
$${\rm supp}\,r=\left(\bigcup\limits_{a\in\mathcal{D}(r)\setminus\mathcal{I}(r)}\mathfrak{O_a}\right)\cup\left( \bigcup\limits_i{\rm supp}\,c_i^{(r)} \right)\cup\left(X_\infty\setminus(\mathcal{D}(r)\cup \mathcal{I}(r))  \right).$$
\end{Rem}
For each $r\in R_\infty$,  we define the unordered partitions $\,^q\!\lambda(r)=\left(\,^q\!\lambda_1(r)\geq\,^q\!\lambda_2(r)\geq\ldots\geq\,^q\!\lambda_j(r)  \right)$,
$\,^c\!\lambda(r)=\left(\,^c\!\lambda_1(r)\geq\,^c\!\lambda_2(r)\geq\ldots\geq\,^c\!\lambda_k(r)  \right)$, where $j=\#\left( \mathcal{D}(r)\setminus\mathcal{I}(r) \right)$, $k$ is the number of nontrivial usual cycles in decomposition (\ref{decomposition_of_r}), and nonnegative integer $m(r)$ as follows
\begin{eqnarray*}
^q\!\lambda_i(r)=\#\mathfrak{F}_{a_i},\;a_i\in\mathcal{D}(r)\setminus\mathcal{I}(r), \\
\,^c\!\lambda_i(r)=\#\left( {\rm supp}\,c_i^{(r)} \right),\;\;m(r)=\#\left(X_\infty\setminus(\mathcal{D}(r)\cup \mathcal{I}(r))\right).
\end{eqnarray*}
Under the above assumptions,  $\,^q\!\lambda_i(r)\geq2$ and $^c\!\lambda_i(r)\geq2$ for all $i$.
\begin{Prop}\label{mult7}
For given $r_1, r_2\in R_\infty$, suppose that $\left(\,^q\!\lambda(r_1),\,^c\!\lambda(r_1),m(r_1)\right)= \left(\,^q\!\lambda(r_2),\,^c\!\lambda(r_2),m(r_2)\right)$. Then there exists $s\in\mathfrak{S}_\infty$ such that $r_1=sr_2s^{-1}$.
\end{Prop}
\begin{proof}[{\it The proof} {\rm is straightforward}]\end{proof}

The next statement is a specification for the case of the semigroup algebra $\mathbb{C}\left[R_\infty\right]$ of the general asymptotic factorization property of factor states on AF-algebras, which was  noticed by  Powers in \cite{Powers}. For the characters of the infinite-dimensional unitary group, a similar factorization was obtained by Voiculescu in \cite{Voic}. Powers' idea was applied  to prove the multiplicativity of spherical functions on $GL(\infty,\mathbb{C})$ in \cite{NN_GL}.  To prove  the following Proposition we can now proceed analogously to the proof of Proposition 7 \cite{Dudko_Nes}.
\begin{Prop}\label{mult}
Let $f$ belong to $\mathfrak{F}_S$. Take any $r_1, r_2$ such that $({\rm supp}\,r_1)\cap({\rm supp}\,r_2)=\emptyset$. Then $f(r_1\,r_2)=f(r_1)\,f(r_2)$.
\end{Prop}
\subsection{Factor-representation of $R_\infty$ of the finite type.}
For the completeness, we recall here the description of the indecomposable characters on $R_\infty$ \cite{N_INV}.
Let $f$ be $R_\infty$-central factor-state (indecomposable character) on $R_\infty$.
By Proposition \ref{mult}, the restriction $f$ to the subgroup $\mathfrak{S}_\infty\subset R_\infty$ is Thoma character $\chi_{\alpha\beta}$ \cite{Thoma}, where $\alpha=\left\{\alpha_1\geq\alpha_2\geq\ldots>0 \right\}$ and $\beta=\left\{\beta_1\geq\beta_2\geq\ldots>0\right\}$ are the corresponding Thoma parameters. The following statement has been proved in \cite{N_INV}.
\begin{Th}[\cite{N_INV}]\label{finite_factor_repr}
Let $c=(1\;2\;\ldots\;n)$ be an usual cycle, and let $q=c\,\epsilon_{\{a\}}$, where $a\in \{1,2,\ldots,n\}$, be a quasi-cycle. Then there exists $\alpha_i$ such that  $f(q)=\rho^n$, where $\rho\in\{0,\alpha_i\}$, and $f(c)=\chi_{\alpha\beta}(c)=\sum\limits_i\alpha_i^n+(-1)^{n-1}\sum\limits_i\beta_i^k$.
\end{Th}
It follows from Theorem \ref{finite_factor_repr} and Proposition \ref{mult} that $f$ is defined by its values  on  quasi-cycles and cycles.

\subsection{The description of $\mathfrak{S}_\infty$-central indecomposable states on $R_\infty$, corresponding to a semifinite representations.}
Take any state $f\in\mathfrak{F}_\mathfrak{S}$. Applying  Theorem \ref{decomposition_into_product_cycles}
 and Proposition \ref{mult}, we obtain that $f$ is defined by its values on the quasi-cycles and usual cycles.
 The restriction $f$ to subgroup $\mathfrak{S}_\infty\in R_\infty$ is a character  $\chi$ on $\mathfrak{S}_\infty$. By Proposition \ref{mult},  $\chi$ is indecomposable character. Therefore, $\chi=\chi_{\alpha\beta}$, where $\chi_{\alpha\beta}$ is Thoma character, defined by the parameters  $\alpha=\left\{\alpha_1\geq\alpha_2\geq\ldots>0 \right\}$ and $\beta=\left\{\beta_1\geq\beta_2\geq\ldots>0\right\}$.

The main result of this paper is the following generalisation of Theorem \ref{finite_factor_repr}.
\begin{Th}\label{semifinite_repr}
Let $\left(\pi_f,\mathcal{H}_f,\xi_f\right)$ be GNS-representation, corresponding to $f\in\mathfrak{F}_\mathfrak{S}$. Suppose that  $\left(\pi_f,\mathcal{H}_f,\xi_f\right)$ is  a semifinite representation; i. e. factor $ \pi_f(R_\infty)^{\prime\prime}$ has type ${\rm II}$ or  ${\rm I}$.
Let $c=(1\;2\;\ldots\;n)$ be an usual cycle, and let $q=c\,\epsilon_{\{a\}}$, where $a\in \{1,2,\ldots,n\}$, be a quasi-cycle. Then $f(c)=\chi_{\alpha\beta}(c)$  and there exist a positive parameter Thoma $\alpha_i$ and $t\in[0,1]$ such that $f(q)=t\alpha_i^n$.
\end{Th}

\begin{Th}
Let $(t,\alpha,\beta,i)$ be the quadruple, corresponding to $f$ according to Theorem \ref{semifinite_repr}. Then, under the conditions of Theorem \ref{semifinite_repr}, we have the following:
\begin{itemize}
  \item[$({\bf i})$] if $\alpha_1=1$ and $t\in(0,1)$, then $\left(\pi_f,\mathcal{H}_f,\xi_f\right)$ is the representation of the type ${\rm I}_\infty$ and $\pi_f(s)\xi_f=\xi_f$ for all $s\in\mathfrak{S}_\infty$;
  \item[$({\bf ii})$] if $\alpha=\emptyset$, then $\pi_f(\epsilon_{\{a\}})=0$ for all $a\in X_\infty$; i. e. $\pi_f\left( R_\infty \right)^{\prime\prime}$ is a ${\rm II}_1$-factor or $\mathbb{C}$;
  \item[$({\bf iii})$] if $\alpha_i\in (0,1)$ and $t\in(0,1)$, then factor $\pi_f\left( R_\infty \right)^{\prime\prime}$ has a type ${\rm II}_\infty$;
  \item[$({\bf iv})$] if $\alpha_i\in (0,1)$ and $t\in\{0,1\}$, then $\pi_f\left( R_\infty \right)^{\prime\prime}$ is ${\rm II}_1$-factor.
\end{itemize}
\end{Th}
 \section{{$\mathfrak{S}_\infty$-admissible representations of $R_\infty$.}}\label{section:admissible}
 Let $\mathbb{N}$ be a set of the natural numbers. For simplicity of notation we will write  $\mathbb{N}$ instead $X_\infty$.
For the reader's convenience, let us recall some notations introduced earlier. A bijection $s: \mathbb{N}\to \mathbb{N}$ is called finite if the set $\left\{ i\in\mathbb{N}\big|\,s(i) \neq i\right\}$
is finite. Then subgroup $\mathfrak{S}_\infty\subset R_\infty$ is the set
 of all finite bijections $\mathbb{N}\to\mathbb{N}$. We write
$\mathfrak{S}_n=\left\{s\in\mathfrak{S}_\infty|\; s(i)=i\;
 \text{ for all }\;
i>n\right\}$.

 Let $\delta_{ij}$ be the Kronecker  delta. We identify $s\in \mathfrak{S}_\infty$ with a matrix $\left[ g_{ij}\right]_{i,j\in\mathbb{N}}$,  where  $ g_{ij}=\delta_{i\,s(j)}$.  Denote by $\mathbf{1}_\mathbb{A}$ the indicator function of a set  $\mathbb{A}\subset\mathbb{N}$.   Let  $\epsilon_\mathbb{A}=\left[ \left( 1-\mathbf{1}_\mathbb{A}(i)\right)\delta_{ij}\right]$
 and $\mathfrak{S}_\mathbb{A}=\left\{ s\in\mathfrak{S}_\infty:s(j)=j \text{ for all } j\in\mathbb{N}\setminus\mathbb{A}\right\}$. It is clear that $r\cdot\epsilon_\mathbb{A}=\epsilon_\mathbb{A}\cdot r=\epsilon_\mathbb{A}$ for all $r\in\mathfrak{S}_\mathbb{A}$.  Each element $r\in R_\infty$ is defined by $s_r \in \mathfrak{S}_\infty$ and a finite subset $\mathbb{A}_r\subset\mathbb{N}$: $r=s_r\cdot\epsilon_{\mathbb{A}_r}$.
 In the matrix realisation  $R_n=\left\{ r=[r_{ij}]\in R_\infty:r_{ij}=\delta_{ij}\text{ for all } i >n\right\}$ and $R_{n\infty}$ \ $=\left\{ r=[r_{ij}]\in R_\infty:r_{ij}=\delta_{ij}\text{ for all } i \leq n\right\}$. It is clear that $\mathfrak{S}_n\subset R_n$. Let $e$ be the unit element of $R_\infty$. For a semigroup $S$, denote by $\mathbb{C}[S]$ the associated semigroup algebra. The reader can construct a simple proof of next statement.
  \begin{Lm}\label{Gelfand_pair}
 Let $\mathfrak{p}_n=\sum\limits_{s\in\mathfrak{S}_n}s$. Then algebra $\mathfrak{p}_n\mathbb{C}[R_n]\mathfrak{p}_n$ is abelian.
 \end{Lm}

Denote by $\overline{\mathfrak{S}}_\infty$ the group of all bijections $\mathbb{N}\mapsto\mathbb{N}$. For $n\in\mathbb{N}$ define
\begin{eqnarray*}\label{def_of_S_n_infty}
\overline{\mathfrak{S}}_{n\infty}=\left\{ s\in\overline{\mathfrak{S}}_\infty\left( \mathfrak{S}_\infty\right):s(i)=i \text{ for all }i=1,2,\ldots,n\right\},\;\mathfrak{S}_{n\infty}=\overline{\mathfrak{S}}_{n\infty}\cap \mathfrak S_\infty.
\end{eqnarray*}
The  subgroups $\mathfrak{S}_{n\infty}$ form a fundamental system  of neighborhoods of the unit element of
 the {\it weak} topology on $\mathfrak{S}_\infty$.  The weak closure of $\mathfrak{S}_\infty$ coincides with $\overline{\mathfrak{S}}_\infty$.

 Recall the notion of a tame representation \cite{Olshsurvey}.
 \begin{Def}\label{Def_tame}
 Let unitary representation $\pi$ of the group  $\mathfrak{S}_\infty$ act in a Hilbert space $H_\pi$.
   $\pi$ is called {\it tame}, if it is continuous with respect to the {\it weak} topology on  $\overline{\mathfrak{S}}_\infty$ and the strong operator topology on $\mathcal{B}(H_\pi)$.
 \end{Def}
 Put $H_\pi(n)=\left\{ \eta\in H_\pi:\pi(s)\eta=\eta \text{ for all } s\in\mathfrak{S}_{n\infty}\right\}$. By definition, $H_\pi(1)\subset H_\pi(2) \subset\ldots$.\label{def_yP_n}
 The next statement provides a natural  intrinsic characterization of tame representations.
 \begin{Prop}[\cite{Olsh_tamesym}]\label{prop_tame_fix}
 A representation $\pi$ is tame if and only if the closure of $\bigcup\limits_{n=1}^\infty H_\pi(n)$ coincides
 with $H_\pi$.
 \end{Prop}
 \noindent In \cite{Olsh_tamesym}, Olshanski also introduced and studied so called \emph{admissible representations} of several pairs of groups related to $\mathfrak S_\infty$. Here we study analogous notion for the pair $(R_\infty,\mathfrak{S}_\infty)$.
 \begin{Def}\label{def_yP_n}
  A unitary representation $\pi$ of the semigroup $R_\infty$ is called  $\mathfrak{S}_\infty$-{\it admissible}, \label{admissible} if its restrictions to $\mathfrak{S}_\infty$ is tame.
 \end{Def}
\subsection{Representations of $R_\infty$, living in the tame representations of $\mathfrak{S}_\infty$.}
The classification of tame representations of the group $\mathfrak{S}_\infty$ was first obtained by Lieberman \cite{LIEB}. Later G.I. Olshanski gave a new proof of Lieberman's result using the semigroup method \cite{Olsh_tamesym}, \cite{Olshsurvey}.

Fix a tame representation $\pi$ of $\mathfrak{S}_\infty$. Denote by $(i\;n)$  the transposition exchanging   $i$ and $n$.   We begin with a simple property of tame representations.
\begin{Lm}\label{as_projection}
The limit of the sequence of operators $\left\{ \pi\left( (i\;n)\right)\right\}_{n\in\mathbb{N}}$ exists in the weak operator topology. If $P_i=\lim\limits_{n\to\infty}\pi\left( (i\;n)\right)$, then $P_i^2=P_i$.
\end{Lm}
\begin{proof}
Fix any positive number $\epsilon$. Using definition \ref{Def_tame}, for the vectors $\zeta, \eta$ $\in H_\pi$, we find
natural $N$ such that
\begin{eqnarray*}
\left\| \pi\left((k\;n)\right)\zeta-\zeta \right\|<\epsilon \text{ and }
\left\| \pi\left((k\;n)\right)\eta-\eta \right\| <\epsilon \text{ for all } k,n>N.
\end{eqnarray*}
It follows that for all $k,n>N$
\begin{eqnarray*}
\left|\left(\pi\left( (i\;n)\right)\eta,\zeta \right)- \left(\pi\left( (i\;k)\right)\eta,\zeta \right)  \right|<\epsilon (\|\eta\|+\|\zeta\|).
\end{eqnarray*}
Therefore, the sequence $\left\{ \left( \pi((i\;n))\eta,\zeta \right)\right\}_{n\in\mathbb{N}}$ is fundamental.  This proves the first statement of our lemma.

Now we claim that $P_i^2=P_i$. Indeed, for any number $\epsilon>0$ and the vectors $\zeta, \eta$ $\in H_\pi$ there exists $N$, which satisfies  the conditions
\begin{eqnarray*}
&\left\| \pi((k\;n))\eta-\eta \right\|<\epsilon ~\text{ and }~ \left\| \pi((k\;n))\zeta-\zeta \right\|<\epsilon \text{ for all } k,n>N;\\
&\left| \left(P_i^2\eta,\zeta \right) -\left(\pi\left( (i\;k)(i\;n)\right)\eta,\zeta \right)\right| <\epsilon  \text{ for fixed } k>N \text{ and for all } n > N;\\
&\left| \left( P_i\eta,\zeta \right)-\left((i\;n)\eta,\zeta \right)\right|<\epsilon \text{ for all } n>N.
\end{eqnarray*}
Hence, applying the simple estimates and using the relation $(i\;k)(i\;n)$$=(i\;n)(k\;n)$, we obtain
$\left|\left(P_i^2\eta,\zeta \right)-\left(P_i\eta,\zeta \right)  \right|<\epsilon (\| \zeta\|+1)$.
\end{proof}
\begin{Lm}\label{as_relations}
Let $ P_i $ be the projection that was defined in lemma \ref{as_projection}. Then
\begin{itemize}
  \item [{\rm i)}] $\pi(s)P_i\pi\left( s^{-1}\right)=P_{s(i)}$ and $ P_iP_k=P_kP_i$ for all $i,k\in\mathbb{N}$, $s\in\mathfrak{S}_\infty$;
  \item [{\rm ii)}] if $s\in\mathfrak{S}_\mathbb{A}$, then $\pi(s)\,\prod\limits_{i\in\mathbb{A}}P_i=\prod\limits_{i\in\mathbb{A}}P_i$.
\end{itemize}
\end{Lm}
\begin{proof}
The relations from {\rm i)} are the direct consequence of the definition $P_i$.

To prove statement {\rm ii)} it suffices to consider the case $\mathbb{A}=\{k,l\}$.

Fix $\eta,\zeta\in H_\pi$.
For any positive $\epsilon$ find $N\in\mathbb{N}$ such that
\begin{eqnarray*}
&\left|\left( \pi((k\;l))P_kP_l\eta,\zeta\right)-\left(\pi((k\;l))\pi((k\;n))\pi((l\;m))\eta,\zeta \right)\right|<\epsilon \text{ and }\\
&\left|\left( P_kP_l\eta,\zeta\right)-\left(\pi((k\;n))\pi((l\;m))\eta,\zeta \right)\right|<\epsilon \text{ for fixed } m>N \text{ and all } n>N;\\
&\left\| \pi((m\;n)) \eta-\eta\right\|<\epsilon \text{ for all } m,n>N.
\end{eqnarray*}
Combining these inequalities with the relations $(k\;l)(k\;n)(l\;m)$ $=(k\;n)(l\;m)(n\;m)$, we get
$\left| \left( \pi((k\;l))P_kP_l\eta,\zeta\right)-\left( P_kP_l\eta,\zeta\right) \right|<\epsilon (\| \zeta\|+2)$.
\end{proof}

Associated with the tame representation $\pi$ of $\mathfrak{S}_\infty$ is the map $\widehat{\pi}$ that is defined by the formula
\begin{eqnarray*}
\widehat{\pi}(s)=\pi(s),\; s\in\mathfrak{S}_\infty;\\
\widehat{\pi}\left( \epsilon_\mathbb{A} \right)=\prod\limits_{i\in\mathbb{A}}P_i, \text{ where } \mathbb{A} \subset\mathbb{N} \text{ and } \#\mathbb{A}<\infty.
\end{eqnarray*}
The following assertion follows from lemmas \ref{as_projection} and \ref{as_relations}.
\begin{Prop}
The map $\widehat{\pi}$ extends by multiplicativity to the representation of the semigroup $R_\infty$.
\end{Prop}

\subsection{$\mathfrak{S}_\infty$-spherical representations of $R_\infty$.}\label{sph_repr}
Let $\pi$ be a factor-representation of the semigroup $R_\infty$ in a Hilbert space $H_\pi$. Suppose  that there exists a cyclic vector $\xi$ for $\pi$, which satisfies the condition
\begin{eqnarray}\label{fixed}
\pi(s)\xi =\xi  \text{ for all } s\in \mathfrak{S}_\infty.
\end{eqnarray}
Such representations are called a $\mathfrak{S}_\infty$-{\it spherical}.  Observe that for a spherical representation $\pi$ for any $n\in\mathbb N$ and $r\in R_n\subset R_\infty$ the vector $\pi(r)\xi$ is fixed under the action of $\mathfrak S_{n\infty}$. It follows that the restriction of $\pi$ onto $\mathfrak S_\infty$ is tame, and thus $\pi$ is $\mathfrak S_\infty$-admissible.

Recall that for $r\in R_\infty$ the set ${\rm supp}\,r$ was defined in \eqref{supp(r)}
\begin{Rem}\label{support_of_el_semigroup}
Given an element $r=[r_{mn}]\in R_\infty$, the set ${\rm supp}\,r$ coincides with the complement of the set $\left\{n\in \mathbb{N}:r_{nn}=1\right\}$.
\end{Rem}
 \noindent By definition of $R_\infty$, ${\rm supp}\,r$ is a finite set. If ${\rm supp}\,r_1\cap{\rm supp}\,r_2=\emptyset$, then $r_1$ and $r_2$ will be called independent.

 Let $c=\left( n_1\;n_2,\;\cdots\;n_k\right)$ be a cycle from $\mathfrak{S}_\infty$. If $a\in{\rm supp}\,c$, then the quasi-cycle $q=c\cdot \epsilon_{a}$ (see Subsection \ref{subsec: cycle decomposition}) we call a $k$-{\it quasicycle}. Notice, that $\epsilon_{\{k\}}$ is a $1$-quasicycle. Two quasicycles $q_1$ and $q_2$ are called {\it independent}, if $({\rm supp}\,q_1)\cap({\rm supp}\,q_2)=\emptyset$.    By Theorem \ref{decomposition_into_product_cycles}, each $r\in R_\infty$ can be decomposed  into  the product of the quasicycles\label{quasicycle}:
\begin{eqnarray}\label{decomposition_into_product}
r=q_1\cdot q_2\cdots q_l, \text{ where } {\rm supp}\, q_i\,\cap \,{\rm supp}\, q_j=\emptyset \text{ for all } i\neq j.
\end{eqnarray}
\begin{Prop}\label{Prop_spherical}
Let $H_\pi^{\rm sph}=\left\{ \eta\in H_\pi: \pi(s)\eta=\eta\text{ for all } s\in\mathfrak{S}_\infty\right\}$. Then the following holds:
\begin{itemize}
  \item [{\rm (a)}] if ${\rm supp}\,q\cap{\rm supp}\,r=\emptyset$, then $\left( \pi(q)\pi(r)\xi,\xi\right)=(\pi(q)\xi,\xi)(\pi(r)\xi,\xi)$;

  \item [{\rm (b)}] orthogonal projection $E$ onto $H_\pi^{\rm sph}$ belongs to $\pi(\mathfrak{S}_\infty)''\subset \pi(R_\infty)''$ and $E\,\pi(R_\infty)\,E\subset \mathbb{C}E$;
  \item[{\rm (c)}] factor representation $\pi$ is irreducible   iff ${\rm dim}\,H_\pi^{\rm sph}=1$.
\end{itemize}
\end{Prop}
\begin{proof}
Part ${\rm (a)}$ follows from Proposition \ref{mult}, since the function $f(r)=(\pi(r)\xi,\xi)$ belongs to $\mathfrak{F}_S$ (and even satisfies a stronger condition of two-sided $\mathfrak S_\infty$-invariance). For the readers convenience, below we present a proof for this particular case.

 Since ${\rm supp}\,q\cap{\rm supp}\,r=\emptyset$, there
 exist  elements $\left\{ s_i\right\}_{i\in\mathbb{N}}\subset\mathfrak{S}_\infty$ such that
\begin{eqnarray}\label{estimate_supp}
\min\left\{k\in{\rm supp}\,s_i q s_i^{-1}\right\}\,>i, \;\; s_i r s_i^{-1} =r.
\end{eqnarray}
Since $\xi$ is cyclic vector, (\ref{estimate_supp}) shows that $\lim\limits_{i\to\infty}\pi\left( s_i q s_i^{-1} \right)$ exists in the weak operator topology. Put $A=\lim\limits_{i\to\infty}\pi\left( s_i q s_i^{-1} \right)$. We conclude from (\ref{estimate_supp}) that
\begin{eqnarray*}
A\pi(t)=\pi(t)A \;\text{ for all } t\in R_\infty.
\end{eqnarray*}
Since $\pi$ is a factor-representation, then $A=c(q){\rm I}_{H_\pi}$, where $c(q)=\left( \pi(q)\xi,\xi \right)$, and ${\rm I}_{H_\pi}$ is the identity operator in $H_\pi$.

 Using  (\ref{estimate_supp}) again, we obtain
\begin{eqnarray*}
\left( \pi(q)\pi(r)\xi,\xi\right)\stackrel{(\ref{fixed})}=\left(\pi\left( s_i\right)\pi(q)\pi(r)\pi\left(s_i^{-1}\right)\xi,\xi\right)=\left( \pi\left(s_i q s_i^{-1}\right)\pi(r)\xi,\xi\right).
\end{eqnarray*}
Letting $i\to \infty$, we see that
\begin{eqnarray}
\left( \pi(q)\pi(r)\xi,\xi\right)=\left(A \pi(r)\xi,\xi\right)=(\pi(q)\xi,\xi)(\pi(r)\xi,\xi).
\end{eqnarray}
To prove {\rm (b)}, we consider element $\mathfrak{p}_n=\frac{1}{n!}\sum\limits_{s\in\mathfrak{S}_n}s$ from semigroup algebra $\mathbf{C}[R_n]\subset\mathbf{C}[R_\infty]$.  By definition, $\pi\left( \mathfrak{p}_n\right)\geq\pi\left( \mathfrak{p}_{n+1}\right)$. Therefore, there exists $\lim\limits_{n\to\infty}\pi\left( \mathfrak{p}_n\right)= E$ in the strong operator topology.  It follows immediately that $E$ is a projection from $\pi\left(\mathfrak{S}_\infty\right)''\subset\pi\left( R_\infty\right)^{\prime\prime}$. By Lemma  \ref{Gelfand_pair}, algebra $E\pi\left( R_\infty\right)^{\prime\prime}E$ is abelian. Since $\pi$ is a factor representation, we obtain that  $E\pi\left( R_\infty\right)^{\prime\prime}E$ is a factor. Therefore, $E\pi\left( R_\infty\right)^{\prime\prime}E=\mathbb{C}E$, which finishes the proof of {\rm (b)}.

From part {\rm (b)} we deduce that $E$ is a minimal orthogonal projection in $\pi\left( R_\infty\right)^{\prime\prime}$. Taking into account that $H_\pi^{\rm sph}=E H_\pi$, we obtain {\rm (c)}.
\end{proof}
To classify the admissible representations in
Subsection \ref{par_classification} we will use the next statement.
\begin{Prop}\label{Prop_minimal}
Let $\pi$ be a factor-representation  of $R_\infty$ with the cyclic vector $\xi$ such that $\left( \pi\left( srt\right)\xi ,\xi \right)=\left( \pi\left( r\right)\xi ,\xi \right)$ for all $s,t\in\mathfrak{S}_\infty$ and $r\in R_\infty$. Denote by $\mathfrak{A}_k$ the von Neumann algebra generated by $\pi\left( \epsilon _{\{k\}}\right)$ and $P_k$ (see lemmas \ref{as_projection} and \ref{as_relations}). Then the following hold:
\begin{itemize}
  \item [{\rm a)}] $\xi$ is a separating vector for $P_k\mathfrak{A}_kP_k$;
  \item [{\rm b)}] $P_k$ is a minimal projection in $\mathfrak{A}_k$.
\end{itemize}
\end{Prop}
\begin{proof}
To prove {\rm a)} we fix $a\in P_k\mathfrak{A}_kP_k$ such that $a\xi=0$. It suffices to show that
\begin{eqnarray}
\left( a\pi(r_1) \xi, \pi(r_2)\xi\right)=0 \text{ for all } r_1, r_2\in R_\infty.
\end{eqnarray}
For $N$ from the complement of the set ${\rm supp}\,r_1\cup\,{\rm supp}\,r_2$, using lemma \ref{as_projection},  we have
\begin{eqnarray*}
&\left( a\pi(r_1) \xi, \pi(r_2)\xi\right)=\lim\limits_{M\to\infty}\left( a\pi\left( (k\,M)\right)\pi(r_1) \xi, \pi(r_2)\xi\right)=\left( a\pi\left( (k\,N)\right)\pi(r_1) \xi, \pi(r_2)\xi\right)\\
&~ \text{and}~
\pi\left( (k\,N)\right)a\pi\left( (k\,N)\right)\pi(r_i)= \pi(r_i)\pi\left( (k\,N)\right)a\pi\left( (k\,N)\right), ~\text{ where }~ i=1,2.
\end{eqnarray*}
Hence
\begin{eqnarray*}
\left( a\,\pi(r_1)\xi, \pi(r_2)\xi\right)=\left( a\,\pi((k\,\,N))\,\pi(r_1)\xi, \pi(r_2)\xi\right)\\
=\left(\pi((k\,\,N))\, a\,\pi((k\,\,N))\,\pi(r_1)\xi,\pi((k\,\,N))\, \pi(r_2)\xi\right)\\
=\left(\pi(r_1)\,\pi((k\,\,N))\, a\,\pi((k\,\,N))\xi,\pi((k\,\,N))\, \pi(r_2)\xi\right)\\
\left(\pi(r_1)\,\pi((k\,\,N))\, a\xi,\pi((k\,\,N))\, \pi(r_2)\xi\right)=0.
\end{eqnarray*}
 {\rm a)} is proved.

The property  {\rm b)} follows from   {\rm a)}  and Lemmas  \ref{as_projection}, \ref{as_relations}. Indeed, if $a\in\mathfrak{A}_k$ then
\begin{eqnarray*}
\left(\pi\left((k\;l) \right) P_ka\xi,\pi(r)\xi\right)\stackrel{\text{Lemma \ref{as_projection}}}{=}
\lim\limits_{n\to\infty}\left(\pi\left(\left(k\;l \right) \right)\pi((k\;n))a\xi,\pi(r)\xi \right)\\=
\lim\limits_{n\to\infty}\left(\pi\left(\left(k\;n \right) \right)\pi((l\;n))a\xi,\pi(r)\xi \right)\stackrel{\text{Lemma \ref{as_relations}}}{=}\lim\limits_{n\to\infty}\left(\pi\left(\left(k\;n \right) \right)a\xi,\pi(r)\xi \right)\\=\left(P_k a\xi,\pi(r)\xi \right)\;\text{ for all }\;r\in R_\infty.
\end{eqnarray*}
Thus $\pi\left((k\;l) \right) P_ka\xi=P_k a\xi$ for all $l$. Together with Lemma \ref{as_relations} the latter implies that $\pi\left(s \right) P_ka\xi=P_ka\xi$ for all $s\in\mathfrak{S}_\infty$. Hence, using Proposition \ref{Prop_spherical}, we obtain that $P_kaP_k\xi=P_ka\xi=EP_kaE\xi =\zeta E \xi=\zeta\xi$, where $\zeta\in\mathbb{C}$. Since $\xi$ is separable for $P_k\mathfrak A_k P_k$ by $a)$, it follows that $P_kaP_k=\zeta P_k$, which finishes the proof.
\end{proof}
\subsection{Examples of $\mathfrak{S}_\infty$-admissible representations.}
\noindent Let us now construct examples of irreducible $\mathfrak{S}_\infty$-admissible representations of the semigroup $R_\infty$.

As a simple example, given a tame representation $\pi$ of $\mathfrak{S}_\infty$, we define a representation $\widetilde{\pi}$ of $R_\infty$ as follows
\begin{eqnarray*}
\widetilde{\pi}(r)=\left\{
\begin{array}{rl}
\pi(s), &\text{ if } r=s\in\mathfrak{S}_\infty\\
0, &\text{ if } r=\epsilon_{\{k\}}.
\end{array}\right.
\end{eqnarray*} It is clear from definitions that $\widetilde{\pi}$ is admissible and it is irreducible whenever $\pi$ is irreducible.

To construct more complicated examples, consider the vector space $V=\mathbb{C}^m$ with the natural inner product.  Fix the unit vector $u\in V$  For any natural $n$ define the isometrical embedding $V^{\otimes n}\stackrel{\mathfrak{i}_{n\,n+1}}{\mapsto}V^{\otimes (n+1)}$ by $\mathfrak{i}_{n\,n+1}(v)=v\otimes u$. Now we identify $V^{\otimes n}$ with $\mathfrak{i}_{n\,n+1}\left( V^{\otimes n}\right)$, and denote by $V^{\otimes\infty}$ the closure of $\cup_{n=1}^\infty V^{\otimes n}$ with respect to the topology, which defined by the standard inner product on each $V^{\otimes n}$. Fix the unit vector $w\in V$ and denote by $p_w$ the orthogonal projection onto the line $\mathbb{C}w$.

Let the operators $\mathfrak{R}\left( \epsilon_{\{n\}} \right)$ and $\mathfrak{R}\left( s\right)$ $\left( n\in\mathbb{N}, s\in\mathfrak{S}_\infty\right)$ act in $V^{\otimes\infty}$ as follows
\begin{eqnarray}\label{representation_mathfrak_R}
\begin{split}
&\mathfrak{R}\left( \epsilon_{\{n\}} \right)\left( v_1\otimes v_2\otimes\cdots\right)=w_1\otimes w_2\otimes\cdots, \text{ where } w_l=\left\{
\begin{array}{rl}
v_l, &\text{ if } l\neq n\\
p_w v_n, &\text{ if } l=n
\end{array}
\right.;\\
&\mathfrak{R}(s)\left( v_1\otimes v_2\otimes\cdots\right)=v_{s^{-1}(1)}\otimes v_{s^{-1}(2)}\otimes\cdots.
\end{split}
\end{eqnarray}
\begin{Prop}
The assignment $\mathfrak S_\infty\cup\{\epsilon_{\{n\}}:n\in\mathbb N\}\ni r\mapsto \mathfrak{R}(r)$ extends by multiplicativity to a $\mathfrak{S}_\infty$-admissible representation of $R_\infty$.
\end{Prop}
\noindent Let $\xi=u\otimes u\otimes u\otimes\cdots$.
Denote by $\left[ \mathfrak{R}\left( R_\infty\right)\xi\right]$ the closure of the span of  set $\mathfrak{R}\left( R_\infty\right)\xi\subset V^{\otimes\infty}$. The following statement is straightforward.
\begin{Prop}\label{prop_non_zero_lambda}
Let $\pi_{u,w}$ be the restriction of representation $\mathfrak{R}$  to the subspace $\left[ \mathfrak{R}\left( R_\infty\right)\xi\right]$. Then  $\pi_{u,w}$ is irreducible. The representations $\pi_{u_1,w_1}$ and $\pi_{u_2,w_2}$ are unitary equivalent if and only if $\left| \left( u_1,w_1\right) \right|=\left| \left( u_2,w_2\right) \right|$.
\end{Prop}
\subsubsection{Approximation of $\pi_{u,w}$ by irreducible representations of the finite semigroups $R_n$.} Consider the case $m=2$. In the space $V=\mathbb{C}^2$ fix basis $\left\{w,w^\perp \right\}$, where $w^\perp$ is unit vector orthogonal to $w$. Using proposition \ref{prop_non_zero_lambda}, we suppose that
\begin{eqnarray}
u=\kappa w+\sqrt{1-\kappa^2}\cdot w^\perp, \text{ where } \kappa\geq 0.
\end{eqnarray}
Denote by $\pi_{u,w}^{(n)}$ the restriction of $\pi_{u,w}$ on the semigroup $R_n$ $(n<\infty)$ in the subspace $\left[ \pi_{u,w}\left( R_n \right)\xi \right]$.

Now we will find the decomposition of $\pi_{u,w}^{(n)}$ into irreducible components.

For any subset $\mathbb{A}\subset\left\{ 1,2,\ldots, n\right\}$ we will use the symbol $\mathbf{e}_{n\mathbb{A}}$ to denote the vector $v_1\otimes v_2\otimes\cdots\otimes v_n\otimes u\otimes u\otimes\cdots\in\left[ \pi_{u,w}\left( R_n \right)\xi\right]$, where $v_j=\left\{\begin{array}{rl}
w, &\text{ if } j\in\mathbb{A}\\
w^\perp, &\text{ if } j\notin\mathbb{A}
\end{array}\right.$.
Denote by $\pi_{u,w}^{(n,\mathbb{A})} $ the restriction of $\pi_{u,w}^{(n)}$ to the subspace $\left[\pi_{u,w}^{(n)} (R_n) \mathbf{e}_{n\mathbb{A}}\right]$. Since the equality $\# \mathbb{A}=\#\mathbb{B}$, where $\mathbb{B}\subset\left\{ 1,2,\ldots, n\right\}$, implies that $\left[\pi_{u,w}^{(n)} (R_n) \mathbf{e}_{n\mathbb{A}}\right]=\left[\pi_{u,w}^{(n)} (R_n) \mathbf{e}_{n\mathbb{B}}\right]$, we will use the notation
$\pi_{u,w}^{(n,l)} $, where $l=\#\mathbb{A}$, instead $\pi_{u,w}^{(n,\mathbb{A})}$.
It is easily seen that next statement holds.
\begin{Prop}\label{finite_irr}
 $\pi_{u,w}^{(n,l)} $  is an irreducible representation of the finite semigroup $R_n$.
\end{Prop}

Set $\mathbf{e}_{nl}^{sph}= {{n}\choose{l}}^{-1/2}\sum\limits_{\mathbb{A}:\#\mathbb{A}=l}\mathbf{e}_{n\mathbb{A}}$. By definition, $\left\|\mathbf{e}_{nl}^{sph} \right\|=1$ and
\begin{eqnarray*}
\pi_{u,w}(s)\mathbf{e}_{nl}^{sph}=\mathbf{e}_{nl}^{sph} \text{ for all } s\in\mathfrak{S}_n.
\end{eqnarray*}
Hence, using lemma \ref{Gelfand_pair} and proposition \ref{finite_irr}, we have
\begin{eqnarray*}
\left\{\eta\in \left[\pi_{u,w}^{(n)} (R_n) \mathbf{e}_{n\mathbb{A}}\right]:\pi_{u,w}(s)\eta=\eta ~\text{for all }~ s\in\mathfrak{S}_n\right\}=\left\{c\cdot \mathbf{e}_{nl}^{sph} \right\}_{c\in\mathbb{C}},
\end{eqnarray*}
where $\# A=l$.
An easy computation shows that
\begin{eqnarray*}
\left(\pi_{u,w}^{(n,l)}\left(\epsilon_\mathbb{B} \right)\mathbf{e}_{nl}^{sph}, \mathbf{e}_{nl}^{sph} \right)=\left\{\begin{array}{rl}
\frac{l(l-1)\cdots(l-\#\mathbb{B}+1)}{n(n-1)\cdots(n-\#\mathbb{B}+1)}, &\text{ if } \#\mathbb{B}\leq l\\
0, &\text{ if } \#\mathbb{B}> l.
\end{array}\right.
\end{eqnarray*}
Hence we obtain the next result.
\begin{Prop}
Fix any sequence $l_n$ such that $\lim\limits_{n\to\infty}l_n/n=\kappa=|(u,w)|$.
Then the equality $\left(\pi_{u,w}(r)\xi,\xi \right)=\lim\limits_{n\to\infty}\left(\pi_{u,w}^{(n,l_n)}\left(r \right)\mathbf{e}_{nl_n}^{sph}, \mathbf{e}_{nl_n}^{sph} \right)$ holds for each $r\in R_\infty$.
\end{Prop}
\subsubsection{ $\mathfrak{S}_\infty$-admissible irreducible components of the representation $\mathfrak{R}$
.}
First, we will consider the case $|(u,w)|=\alpha >0$. Fix a partition $\lambda $ of $k\leq m-2$.
 Let $\left\{ e_1,e_2,\ldots, e_m\right\}$ be the  orthonormal basis for $V$, and let $W_k$ be the subspace in $V$ generated by $\left\{ e_1,e_2,\ldots, e_k\right\}$. Denote by $W_k^\perp$ the orthogonal complement  of $W_k$. Suppose that the unit vectors $u$ and $w$ lie in  $W_k^\perp$. Let $\xi_k=e_1\otimes e_2\otimes\cdots\otimes e_k\otimes u\otimes u \otimes\cdots$. Now we notice that  the restriction of $\mathfrak{R}$ to the group $\mathfrak{S}_k$ in $\left[\mathfrak{R}\left( \mathfrak{S}_k\right)\xi_k\right]$ is unitary equivalent to the regular representation of $\mathfrak{S}_k$. Therefore, there exists an $\mathfrak{R}\left( \mathfrak{S}_k\right)$-invariant subspace $H_\lambda\subset \left[\mathfrak{R}\left( \mathfrak{S}_k\right)\xi_k\right]$ such that the restriction of $\mathfrak{R}$ to $\mathfrak{S}_k$ in $H_\lambda$ is the irreducible representation corresponding to $\lambda$. Denote by $\pi_{u, w}^{\alpha ,\lambda}$ the restriction of $\mathfrak{R}$ to the subspace $\left[ \mathfrak{R}\left( R_\infty\right)H_\lambda\right]$.
 \begin{Prop}\label{prop_zero_lambda}
 The representation $\pi_{u,w}^{\alpha ,\lambda}$ is irreducible.
 \end{Prop}
 Now we suppose that $(u,w)=0$. Let $\xi _k= \underset{j}{\underbrace{ w\otimes\cdots\otimes w}}\otimes e_1\otimes e_2\otimes\cdots\otimes e_k\otimes u\otimes u \otimes\cdots$. Define $H_\lambda$ and $\pi_{u,w}^{0 ,\lambda}$ same as above.
 \begin{Prop}\label{Prop_zero}
 The representation $\pi_{u,w}^{0 ,\lambda}$ is irreducible and $\pi_{u,w}^{0 ,\lambda}(\epsilon_{\{k\}})=0$ for all $k$.
 \end{Prop}
 \subsection{The classification of $\mathfrak{S}_\infty$-admissible representations of $R_\infty$.}\label{par_classification}
  Let ${\rm id}_m$ be the unit representation of $R_\infty$ in $m$-dimensional space; i .e.  ${\rm id}_m(r)={\rm I}$ for all $r\in R_\infty$, where ${\rm I}$ is the identity operator.
 Here we prove the following statement.
 \begin{Th}\label{main_admissible}
 Let $\pi$ be an $\mathfrak{S}_\infty$-admissible factor-representation of the semigroup $R_\infty$ in the Hilbert space $H_\pi$. Then there exists a representation $\pi_{u,w}^{\alpha ,\lambda }$ (see propositions \ref{prop_non_zero_lambda}, \ref{prop_zero_lambda}) such that $\pi$ is unitary equivalent to $\pi_{u,w}^{\alpha ,\lambda }\otimes {\rm id}_m$.
 \end{Th}
 As a preliminary step we prove three auxiliary lemmas.
 \begin{Lm}\label{relations_projections}
 Define the projections $P_i$ as in lemma \ref{as_projection}. There exists $\alpha \in[0,1]$ such that $\pi\left( \epsilon_{\{i\}}\right)P_i\pi\left( \epsilon_{\{i\}}\right)=\alpha \pi\left( \epsilon_{\{i\}}\right)$.
 \end{Lm}
 \begin{proof} First we notice that exists $\lim\limits_{n\to\infty}\pi(\epsilon_{\{n\}})$ in the weak operator topology. Since $\pi$ is factor-representation, $\lim\limits_{n\to\infty}\pi(\epsilon_{\{n\}})=\alpha {\rm I}$, where ${\rm I}$ is the identity operator in $H_\pi$ and $\alpha \in[0,1]$.
 Now our statement follows from the relation $\epsilon_{\{i\}} (i\,n)\epsilon_{\{i\}}=\epsilon_{\{i\}}\epsilon_{\{n\}}$.
 \end{proof}
 Recall that $H_\pi(n)=\left\{ \eta\in H_\pi:\pi(s)\eta=\eta \text{ for all } s\in\mathfrak{S}_{n\infty}\right\}$. \label{E_n} Denote by $E(n)$ the orthogonal projection onto $H_\pi(n)$.
  The next statement follows from Proposition \ref{prop_tame_fix} and Lemma \ref{as_relations}.
 \begin{Lm}\label{lemma_inv_product}
 $H_\pi(n)=\prod\limits_{i=n+1}^\infty P_i\,H_\pi$ for all $n\in\mathbb{N}$.
 \end{Lm}
\noindent Notice that for each $n$ the subspace $H_\pi(n)$ is $\pi(R_n)$-invariant. Using proposition \ref{prop_tame_fix}, define the depth $d(\pi)$\label{depth_of_pi} of $\pi$ as $\min\left\{ n:H_\pi(n)\neq 0\right\}\;\footnote{see page \pageref{def_yP_n}}$.
 Let $B\left( H_\pi\right)$ be the set of all bounded linear operators on $H_\pi$. Consider subset $\mathbb{S}\subset B\left( H_\pi\right)$, and put $\mathbb{S}^\prime=\left\{ A\in B\left( H_\pi\right):AS=SA \text{ for all } S\in\mathbb{S}\right\}$. Algebra $\mathbb{S}^\prime$ is called the commutant of $\mathbb{S}$. In the sequel,   $\mathbb{S}^{\prime\prime}$ stands for the commutant of the commutant $\mathbb{S}^\prime$ of $\mathbb{S}$, i. e. $\mathbb{S}^{\prime\prime}=\left(\mathbb{S}^\prime \right)^\prime$.
 \begin{Rem}\label{remark_cyclic}
 Denote by $\widetilde{H}_\pi(n)$ \label{widetilde_H_pi_n} the closure of  the span of $\pi\left( R_{n\infty}\right)H_\pi(n)$.
 By lemma \ref{lemma_inv_product}, $A^\prime H_\pi(n)\subset H_\pi(n)$ for all $A^\prime\in \pi\left( R_\infty\right)^\prime$. Therefore, $\pi\left( R_\infty\right)^\prime\widetilde{H}_\pi(n)\subset \widetilde{H}_\pi(n)$. Thus the orthogonal projection $\widetilde{E}(n)$ onto  $\widetilde{H}_\pi(n)$ lies in
$\left\{ \pi\left( R_\infty\right)^\prime\right\}^\prime=\pi\left( R_\infty\right)^{\prime\prime}$.
 \end{Rem}
 \begin{Lm}\label{setting-up}
 Let $s\in\mathfrak{S}_\infty$, and let $n=d(\pi)$. If $s\notin \mathfrak{S}_n\cdot\mathfrak{S}_{n\infty}$, then $\widetilde{E}(n)\pi(s)\widetilde{E}(n)=0$.
 \end{Lm}
 \begin{proof}
 It suffices to  prove that
 \begin{eqnarray}\label{zero_equality}
 \left( \pi(s)\pi\left( \epsilon_\mathbb{A}\right)\xi,\eta \right)=0 \text{ for all  subsets } \mathbb{A}\subset \mathbb{N} \text{ and any } \xi, \eta\in H_\pi(n).
 \end{eqnarray}
 Let $s=c_1\cdot c_2\cdot \cdots \cdot c_k$ be the cycle decomposition of $s$. Here, the cycle $c_i=\left( a_1^{(i)}\;a_2^{(i)}\;\ldots a_{p_i}^{(i)}\;\right)$ is a permutation sending $a_j^{(i)}\;$
to $a_{j+1}^{(i)}\;$ for $1\leq j<p_i$ and  $a_{p_i}^{(i)}\;$ to $ a_1^{(i)}$. Define the maps $$f^{(i)}:\left\{ 1,2,\ldots,p_i\right\}\mapsto\left\{ e,\epsilon_{\left\{a_1^{(i)}\right\}} \right\}\subset R_\infty, i=1,2,\ldots,k,$$ where $e$ is the unit of $R_\infty$, as follows
 $$f^{(i)}(j)=\left\{
 \begin{array}{rl}
 \epsilon_{\left\{a_1^{(i)}\right\}},&\text{ if } a_j^{(i)}\in\mathbb{A}\\
 e,&\text{ if } a_j^{(i)}\notin\mathbb{A}.
 \end{array}\right.$$
  Using the relation $c_i=\left(a_1^{(i)}\;a_{p_i}^{(i)} \right)\left(a_1^{(i)}\; a_{p_i-1}^{(i)}\right)\cdots\left(a_1^{(i)}\; a_{2}^{(i)}\right)$, we obtain
 \begin{eqnarray}\label{cycle_decomposition}
 \begin{split}
& s\cdot\epsilon_\mathbb{A}= \left[f^{(1)}\left( p_1\right)\cdot\left(a_1^{(1)}\;a_{p_1}^{(1)} \right)\cdot f^{(1)}\left( p_1-1\right)\cdot\left(a_1^{(1)}\; a_{p_1-1}^{(1)}\right)\cdot f^{(1)}\left( p_1-2\right)\right.\\
&\left.\cdots f^{(1)}\left( 2\right)\cdot\left(a_1^{(1)}\; a_{2}^{(1)}\right)\cdot f^{(1)}\left( 1\right)\right]
 \cdots \left[f^{(k)}\left( p_k\right)\cdot\left(a_1^{(k)}\;a_{p_k}^{(k)} \right)\cdot f^{(k)}\left( p_k-1\right)\right.\\
 &\left.\cdot\left(a_1^{(k)}\; a_{p_k-1}^{(k)}\right)
 \cdots f^{(k)}\left( 2\right)\cdot\left(a_1^{(k)}\; a_{2}^{(k)}\right)f^{(k)} \left( 1\right)\right]
\cdot \epsilon _\mathbb{B},
 \end{split}
 \end{eqnarray}
 where $\mathbb{B}\subset \mathbb{N}\setminus\,{\rm supp}\,s$. Since $s\notin\mathfrak{S}_n\cdot\mathfrak{S}_{n\infty}$, we can suppose, without loss of generality, that
 $a_1^{(1)}\leq n$ and $a_{p_1}^{(1)}>n$. Using the relations $\pi(t)\xi=\xi$, $\pi(t)\eta=\eta$ $\left( t\in\mathfrak{S}_{n\infty}\right)$, we conclude  that
 \begin{eqnarray}
 \left( \pi\left(s\cdot\epsilon_\mathbb{A}\right)\xi,\eta \right) = \left( \pi\left(\left(a_{p_1}^{(1)}\;N\right) \cdot s\cdot\epsilon_\mathbb{A}\cdot\left(a_{p_1}^{(1)}\;N\right)\right)\xi,\eta \right)
 \end{eqnarray}
 for all $N>\max\left(\mathbb{B}\cup\,{\rm supp}\,s\right)$ (see (\ref{cycle_decomposition})).
 After the passage  to the limit $N\to\infty$, applying (\ref{cycle_decomposition}), we obtain
  \begin{eqnarray}
  \begin{split}
 \left( \pi\left(s\cdot\epsilon_\mathbb{A}\right)\xi,\eta \right) = \left( \pi\left(f^{(1)}\left( p_1\right) \right)\cdot P_{a_1^{(1)}}\cdot\pi\left( f^{(1)}\left( p_1-1\right)\right)\cdots \xi,\eta\right)\\
 =\left( \pi\left( f^{(1)}\left( p_1-1\right)\right)\cdots \xi,P_{a_1^{(1)}}\cdot\pi\left(f^{(1)}\left( p_1\right) \right)\eta\right).
 \end{split}
 \end{eqnarray}
 Since $a_1^{(1)}\leq n$, then $\pi\left(f^{(1)}\left( p_1\right) \right)\eta\in H_\pi(n)$. It follows from lemma \ref{as_relations} that $\pi(t)\cdot P_{a_1^{(1)}}\cdot\pi\left(f^{(1)}\left( p_1\right) \right)\eta=P_{a_1^{(1)}}\cdot\pi\left(f^{(1)}\left( p_1\right) \right)\eta$ for all $t\in\mathfrak{S}_\mathbb{K}$, where
 $\mathbb{K}=a_1^{(1)}\cup [n+1,\infty)$. Therefore, if $P_{a_1^{(1)}}\cdot\pi\left(f^{(1)}\left( p_1\right) \right)\eta\neq 0$, then the depth  of $\pi$ is less than $n$. This contradicts to the fact that $d(\pi)=n$.
 \end{proof}
 Recall that $\pi$ is $\mathfrak{S}_\infty$-{\it admissible} factor-representation of $R_\infty$.
 \begin{Lm}\label{R_n_infty_scalar}
We have the following:
   $E(n)\,\pi(R_{n\infty})\,E(n)\subset\mathbb{C}E(n)$.
 \end{Lm}
 \begin{proof}
 For each  $r\in R_{n\infty}$, we have
 \begin{eqnarray}\label{E_pi_E}
 E(n)\,\pi(s)\,\pi(r)\,\pi(t)\,E(n)=E(n)\,\pi(r)\,E(n)~\text{ for all }~s,t\in\mathfrak{S}_{n\infty}.
 \end{eqnarray}
 There exist the elements $s_{_N}\in \mathfrak{S}_{N\infty}$ such that
 \begin{eqnarray}\label{element_r_N}
r_{_N}= s_{_N}\,r\,s_{_N}^{-1}\in R_{N\infty} ~\text{and}~ s_{_{N+M}}s_{_N}^{-1}\in \mathfrak{S}_{N\infty}~\text{for all natural }~ M\geq 1
 \end{eqnarray}
 Hence,
 \begin{eqnarray*}
 \begin{split}
 &\lim\limits_{N\to\infty} \left( \pi(r_N) \pi(q_1)\xi, \pi(q_2)\xi\right)=\left( \pi(r_L) \pi(q_1)\xi, \pi(q_2)\xi\right), ~\text{where}~\\
  &L=1+{\rm max}\left\{ k\in({\rm supp}\,q_1)\cup({\rm supp}\,q_2)\right\}, ~\text{for all}~\;q_1,q_2\in R_{\infty} ~\text{and}~ \xi\in H_\pi(n).
  \end{split}
 \end{eqnarray*}
  Therefore, there exists $A(r)=\lim\limits_{N\to\infty}\pi(r_N)$ in the weak operator topology. Since $\pi$ is factor-representation, it follows from \eqref{element_r_N} that $A(r)$ is scalar operator; i.e. $A(r)=c(r){\rm I}$, where $c(r)\in \mathbb{C}$. Applying \eqref{E_pi_E}, we obtain $E(n)\,\pi(r)\,E(n)=c(r)E(n)$.
 \end{proof}
 \begin{proof}[{\bf The proof of theorem \ref{main_admissible}.} ]
  First, consider the case $\alpha >0$ (see lemma \ref{relations_projections}). Let $d(\pi)=n$.
 Since $\widetilde{E}(n)\in \pi\left(  R_n\cdot R_{n\infty}\right)^\prime$ (see remark \ref{remark_cyclic} on page \pageref{remark_cyclic}), the operators $\pi_{\widetilde{E}(n)}(r)=\widetilde{E}(n)\cdot\pi(r)\cdot\widetilde{E}(n)$, $r\in R_n\cdot R_{n\infty}$ define a representation of the subsemigroup $ R_n\cdot R_{n\infty}$ on the space $\widetilde{H}_\pi(n)$. By lemma \ref{setting-up},
 \begin{eqnarray}\label{reduction}
 \pi_{\widetilde{E}(n)}\left( R_n\cdot R_{n\infty}\right)^{\prime\prime}=\widetilde{E}(n)\cdot \pi\left( R_\infty \right)^{\prime\prime}\cdot\widetilde{E}(n).
 \end{eqnarray}
 Let us prove that
 \begin{eqnarray}\label{equality_zero}
\pi\left( \epsilon _{\{k\}}\right) \cdot\widetilde{E}(n)=0 \text{ for all } k\leq n.
 \end{eqnarray}
 Assume that $\pi\left( \epsilon _{\{k\}}\right) \cdot\widetilde{E}(n)\neq0$ for some $k\leq n$. Then, applying lemma \ref{relations_projections}, we have
 $ P_k\cdot\pi\left(\epsilon _{\{k\}}\right) \cdot\widetilde{E}(n)\neq0$. Thus
 $P_k\cdot\pi\left(\epsilon _{\{k\}}\right)\eta\neq0$  for some $\eta\in H_\pi(n)$. It follows from lemma \ref{as_relations} that $\pi(s)P_k\cdot\pi\left(\epsilon _{\{k\}}\right)\eta$ $=P_k\cdot\pi\left(\epsilon _{\{k\}}\right)\eta$ for all $s\in\mathfrak{S}_\mathbb{B}$, where $\mathbb{B}=k\cup[n+1,\infty)$. This contradicts  the fact that $d(\pi)=n$. The equality (\ref{equality_zero}) is proved.

 Further, applying Lemmas \ref{setting-up} and \ref{R_n_infty_scalar}  and formula \eqref{equality_zero}, we obtain that
   \begin{align}\label{EnREn}
   E(n)\,\pi(R_\infty)''\,E(n)= E(n)\,\pi_{\widetilde{E}(n)}(R_nR_{n\infty})''\,E(n)\\ =E(n)\,\pi(R_nR_{n\infty})''\,E(n)=
   E(n)\,\pi(\mathfrak S_n)''\,E(n).
   \end{align}
  Since $\mathfrak S_n$ is a finite group, the latter implies that the $W^*$-algebra $$E(n)\,\pi(R_{\infty})''\,E(n)$$ is a factor of type $I_m$ for some $m<\infty$. It follows that $\pi(R_\infty)''$ is a of type $I$. Therefore  $\pi(R_{\infty})'$, is a factor of type ${\rm I}$ as well.

   Let $\left\{ \mathbf{e}_{kj}'\right\}_{k,j=1}^l$  be a matrix unit of $\pi(R_{\infty})' $ (see \cite{TAKES1}, page 51) such that $\mathbf{e}_{11}^\prime$ is a minimal projection in $\pi(R_\infty)^\prime$, and let $\xi_1$ be a unit vector from $\mathbf{e}_{11}'\,E(n)$.  Then the vectors  $\xi_k=\mathbf{e}_{k1}'\xi_1,\,k\in\mathbb N$, are pairwise orthogonal vectors in $H_\pi(n)$ and the following hold:
 \begin{itemize}
  \item [{\rm i)}]  the subspaces $\left[ \pi\left(R_n\cdot R_{n\infty}\right)\xi_i\right]$ are pairwise orthogonal;
   \item [{\rm ii)}] $\oplus_{k=1}^l\left[ \pi\left(R_n\cdot R_{n\infty}\right)\xi_i\right]=\widetilde{H}_\pi(n)$;
   \item [{\rm iii)}] $\left( \pi(r)\xi_1,\xi_1\right)=\left( \pi(r)\xi_k,\xi_k\right)$ for all $r\in R_n\cdot R_{n\infty}$ and $k=1,2,\ldots,l$.
 \end{itemize}
 Since $ \pi_{\widetilde{E}(n)}$ is factor-representation of $R_n\cdot R_{n\infty}$, the relation
 $\left( \pi\left( r_1r_2\right)\xi_i,\xi _i \right)=\left( \pi\left( r_1\right)\xi_i,\xi _i \right)\left( \pi\left( r_2\right)\xi_i,\xi _i \right)$
  holds for all $r_1\in R_n$ and $r_2\in R_{n\infty}$. Therefore, the restriction of $\pi$ to $R_{n\infty}$ on the subspace $\left[ \pi\left( R_{n\infty}\right)\xi_i\right]$  is a cyclic factor- representation. It follows from propositions \ref{Prop_spherical}, \ref{Prop_minimal} and lemmas \ref{as_relations}, \ref{relations_projections} that
  \begin{eqnarray}\label{value_spherical_function}
  \left( \pi(r)\xi _k,\xi _k\right)=\alpha ^{\#\mathbb{A}}, \text{ where } r=s\cdot\epsilon _\mathbb{A}\in R_{n\infty}, s\in\mathfrak{S}_{n\infty}.
  \end{eqnarray}
  The restriction of $\pi$ to $R_n$ on the subspace $\left[\pi\left( R_n\right)\xi _i \right]$ is a factor-representation for the same reason. It follows from (\ref{equality_zero}) and the properties {\rm i)}-{\rm iii)} that the restriction of $ \pi_{\widetilde{E}(n)}$  to $R_n$ has the form
  \begin{eqnarray*}
  \pi_{\widetilde{E}(n)}(s)=T_\lambda (s)\otimes {\rm id}_m(s) \text{ for all } s\in\mathfrak{S}_n;
  \pi_{\widetilde{E}(n)}\left( s\cdot\epsilon _\mathbb{A}\right)=0, \text{ if } \mathbb{A}\neq\emptyset,
  \end{eqnarray*}
  where $T_\lambda$ is the irreducible representation of $\mathfrak{S}_n$. Now, applying  the pro\-per\-ties {\rm i)}-{\rm iii)}, (\ref{value_spherical_function}) and lemma \ref{setting-up}, we obtain the statement of theorem \ref{main_admissible} in the case $\alpha >0$.

  If $\alpha =0$, then, by lemma \ref{relations_projections}, the projections  $\pi\left( \epsilon_i \right)$ and $P_i$ are orthogonal for any $i$. It follows from lemma    \ref{lemma_inv_product} that
  \begin{eqnarray}\label{kernel}
  \pi\left( \epsilon _{\{l\}}\right)\eta =0 \text{ for all  } \eta \in H_\pi(n) \text{ and } l=n+1, n+2,\ldots.
  \end{eqnarray}
Put $Q_l=\prod\limits_{i=l+1}^\infty\left( {\rm I}-\pi\left( \epsilon_i \right)\right)$. It follows from (\ref{kernel}) that \begin{eqnarray}
Q_n\eta=\eta  \text{ for all  } \eta \in H_\pi(n), \text{  and  }  Q_l  H_\pi(n)\subset  H_\pi(n) \text{ for   } l\in0\cup\mathbb{N}.
\end{eqnarray}
Let $j=\min \left\{ i: Q_{i} H_\pi(n)\neq 0 \right\}$, and let
 \begin{eqnarray}\label{projection_tame}
 F_j= Q_{j}\cdot\prod\limits_{i=n+1}^\infty P_i.
 \end{eqnarray}
  Since $P_k\pi\left(\epsilon _{\{l\}} \right)=\pi\left( \epsilon _{\{l\}}\right)P_k$ for all $k\neq l$, the operator $F_j$ is an orthogonal projection.

 Notice that, by definition of $j$ and (\ref{projection_tame}),
 \begin{eqnarray}\label{projection_tame_symmetrscal}
\pi\left( \epsilon_{\{i\}} \right)F_j= \left\{
 \begin{array}{lr}
 0,& \text{ if } i> j;\\
 F_j,&\text{ if } i\leq j.
 \end{array}
 \right.
 \end{eqnarray}
Let $\widetilde{F}_j$ be the projection onto the subspace $\left[ \pi\left( R_j\cdot R_{j\infty}\right)F_{j}\,H_\pi\right]$.
 Then, using (\ref{projection_tame_symmetrscal}),  we obtain that
 \begin{eqnarray}\label{tame_symmetrscal}
\pi\left( \epsilon_{\{i\}} \right)\widetilde{F}_j= \left\{
 \begin{array}{lr}
 0,& \text{ if } i> j;\\
 \widetilde{F}_j,&\text{ if } i\leq j.
 \end{array}
 \right..
 \end{eqnarray}

By (\ref{projection_tame}), $\widetilde{F}_j$   lies in $\pi\left( R_\infty\right)^{\prime\prime}\cap\pi\left(R_j\cdot R_{j\infty} \right)^\prime$.
 Using the relation  $\epsilon _{\{i\}}(i\,l)\epsilon _{\{i\}}=\epsilon _{\{i\}}\epsilon_{\{l\}} $, we have from (\ref{projection_tame_symmetrscal})
 \begin{eqnarray}\label{trivial_widetildeF_j}
 \pi(r) \widetilde{F}_j= \widetilde{F}_j\text{ for all } r\in R_j, \text{ and }\widetilde{F}_jH_\pi=\left[ \pi\left(\mathfrak S_j\cdot\mathfrak{S}_{j\infty} \right)F_jH_\pi\right].
 \end{eqnarray}
 On account of the lemmas \ref{as_relations}, \ref{lemma_inv_product}, we have
 \begin{eqnarray*}
 \pi(s)F_j\pi\left(s^{-1} \right)F_j=\left(\prod\limits_{i=j+1}^\infty\left( {\rm I}-\pi\left( \epsilon_{s(i)} \right)\right)\cdot \prod\limits_{k=n+1}^\infty P_{s(k)}\right)\cdot F_j.
 \end{eqnarray*}
 Therefore, we conclude from (\ref{projection_tame_symmetrscal}) that $\pi(s)F_j\pi\left(s^{-1} \right)F_j=0$ for all $s\notin \mathfrak{S}_j\cdot\mathfrak{S}_{j\infty}$. It follows from (\ref{trivial_widetildeF_j}) that
 \begin{eqnarray}\label{imprim}
 \pi(s)\cdot\widetilde{F}_j\cdot\pi\left(s^{-1} \right)\cdot\widetilde{F}_j=0 \text{ for all } s\notin \mathfrak{S}_j\cdot\mathfrak{S}_{j\infty}.
 \end{eqnarray}
 Hence, using (\ref{tame_symmetrscal}), (\ref{trivial_widetildeF_j}), we have $\widetilde{F}_j\left\{ \pi\left(R_\infty \right) \right\}^{\prime\prime}\widetilde{F}_j=\left\{\widetilde{F}_j\pi\left( \mathfrak{S}_{j}\cdot \mathfrak{S}_{j\infty}\right)\widetilde{F}_j\right\}^{\prime\prime}$.
Therefore, the restriction  of $\pi$ to the subgroup $\mathfrak{S}_{j\infty}$ on the space $\widetilde{F}_jH_\pi$  is the tame factor-representation (see definition \ref{Def_tame}).

Set  $X=\mathfrak{S}_\infty/\left(\mathfrak{S}_j\cdot \mathfrak{S}_{j\infty}\right)$. Choose a representative $g_x\in\mathfrak{S}_\infty$ for each  coset $x\in X$.
 By (\ref{imprim}), the map $H_\pi\ni\eta \stackrel{\mathfrak{I}}{\mapsto} \mathfrak{I}\eta\in l^2\left(X, \widetilde{F}_jH_\pi\right)$, where $\left(\mathfrak{I}\eta \right)(x)= \widetilde{F}_j\pi\left(g_x^{-1} \right)\eta$, is an isometry from $H_\pi$ into $l^2\left(X, \widetilde{F}_jH_\pi\right)$. An easy computation shows that for $\eta\in l^2\left(X, \widetilde{F}_jH_\pi\right)$
 \begin{eqnarray}
 \begin{split}
 \left(\mathfrak{I}\pi(s)\mathfrak{I}^{-1}\eta \right)(x)=\pi\left(k^{-1}(s,x) \right)\eta\left(s^{-1}x \right),\\ \text{where } k(s,x)=s^{-1}g_xg_{s^{-1}x}^{-1}\in \mathfrak{S}_j\cdot \mathfrak{S}_{j\infty};\\
 \left(\mathfrak{I}\pi\left(\epsilon_{\{i\}} \right)\mathfrak{I}^{-1}\eta \right)(x)=\pi\left(\epsilon_{\left\{ g_x^{-1}i \right\}} \right)\eta(x)\;\;\;\;\text{ (see (\ref{tame_symmetrscal}))}.
 \end{split}
 \end{eqnarray}
Hence, using (\ref{trivial_widetildeF_j}) and the description of the tame representations of $\mathfrak{S}_{j\infty}$ (see \cite{LIEB} or \cite{Olshsurvey} ), we obtain that $\pi$ is a multiple to one of the representations from Proposition \ref{Prop_zero}.
 \end{proof}

  \section{Properties of $\mathfrak{S}_\infty$-invariant indecomposable positive definite functions on $R_\infty$.}\label{section:properties}
  Let $f$ be a positive definite function on $R_\infty$ such that $f(e)=1$. Recall that $f$ is called $\mathfrak{S}_\infty$-invariant if $f(rs)=f(sr)$ for all $r\in R_\infty$, $s\in\mathfrak{S}_\infty\subset R_\infty$. The next statement can be proven using the same method as analogous statement in \cite{Thoma}. It is also a straightforward modification of Proposition 7 from \cite{Dudko_Nes_2009}.
  \begin{Prop}\label{factor_condition}
Let $\pi_f$ be the GNS-representation of $R_\infty$, corresponding to $f$.  Then $\pi_f$ is a factor-representation if and only if $f\left( r_1r_2\right)=f\left( r_1\right)f\left( r_2\right)$ for all $r_1,r_2\in R_\infty$ such that ${\rm supp}\,r_1\,\cap\,{\rm supp}\,r_2=\emptyset$ (see definition \ref{support_of_el_semigroup}).
\end{Prop}

\subsection{Asymptotic operators $\mathcal{O}_k$.}\label{asumpthotic_operators} Let $f$ be a $\mathfrak{S}_\infty$-invariant  positive definite function on $R_\infty$, and let $ (\pi_f,H_f,\xi_f) $ be the corresponding GNS-representation acting on a Hilbert space $H_f$ with a cyclic vector $\xi_f$, where $f(r)=\left(\pi_f(r)\xi_f,\xi_f \right)$. Recall that $(k\;n)\in\mathfrak S_\infty$ stand for the transposition of $k,n\in\mathbb N$, $k\neq n$. The next assertion is easy to derive from the $\mathfrak{S}_\infty$-invariance of $f$ for all $r\in R_\infty$ and can be proven same way as analogous statement from \cite{Ok2} (see page 3559).
\begin{Lm}\label{Okounkov_operator}
The limits  $\mathcal{O}_k=\lim\limits_{n\to\infty}\pi_f\left((k\;n) \right)$ exist in the weak operator topology. The operators $\left\{ \mathcal{O}_k \right\}_{k\in\mathbb{N}}$ satisfy the following relations:\newline $\mathcal{O}_k\mathcal{O}_l=\mathcal{O}_l\mathcal{O}_k$ and $\pi_f\left(s \right)\mathcal{O}_k\pi_f\left(s^{-1} \right)=\mathcal{O}_{s(k)}$ for all $s\in\mathfrak{S}_\infty$.
\end{Lm}
\begin{Rem}
 By definition, $\mathcal{O}_k$ is a self-adjoint operator and $\left\|\mathcal{O}_k\right\|\leq 1$.
 It is clear that $\mathcal{O}_l$ and $\mathcal{O}_k$ are unitary equivalent for all $l$, $k$.
 \end{Rem}
 Let $\mathfrak{A}_k$  be $W^*$-algebra generated by $\mathcal{O}_k$  and $\pi\left(\epsilon_{\{k\}} \right)$. Namely, $\mathfrak{A}_k=\left\{\mathcal{O}_k,\pi_f\left( \epsilon_k \right) \right\}^{\prime\prime}$. Denote by $E_k(0)$ the orthogonal projection onto subspace $\left\{ \eta\in H_f: \mathcal{O}_k\eta=0 \right\}$.\label{def_of-Algebra A_k}
 \begin{Lm}\label{separating}
 If $A\in \left(1-E_k(0) \right)\mathfrak{A}_k\left(1-E_k(0) \right)$ and $A\xi_f=0$ then $A\equiv 0$. In other words vector $\xi_f$ is separating for the $W^\ast$-algebra $\left(1-E_k(0) \right)\mathfrak{A}_k\left(1-E_k(0) \right)$.
 \end{Lm}
 \begin{proof}
 We can assume that $A=A^*$.
 To obtain a contradiction, suppose that $A\neq 0$. Then  $\mathcal{O}_kA\neq 0$. Therefore, there exist $q,r\in R_n$ such that
 \begin{eqnarray*}
 \left( \mathcal{O}_kA\pi_f(q)\xi_f,\pi_f(r)\xi_f \right)=\lim\limits_{N\to\infty}\left(\pi_f((k\;N))A\pi_f(q)\xi_f,\pi_f(r)\xi_f \right)\neq0.
 \end{eqnarray*}
 Now we fix $N$ with the properties
 \begin{eqnarray*}
 N>n \; \text{ and }\; \left(\pi_f((k\;N))A\pi_f(q)\xi_f,\pi_f(r)\xi_f \right)\neq0.
 \end{eqnarray*}
 Hence, since $\pi_f((k\;N))A\pi_f((k\;N))$ lies in $\pi_f\left(R_n \right)^\prime$, we have
 \begin{eqnarray*}
 0\neq\left(\pi_f((k\;N))A\pi_f(q)\xi_f,\pi_f(r)\xi_f \right)\\=\left(\pi_f((k\;N))\pi_f(q)\xi_f, \pi_f(r) \pi_f((k\;N))A\pi_f((k\;N))\xi_f\right).
 \end{eqnarray*}
 Hence, by $\mathfrak{S}_\infty$-invariance of $f$, we have\newline
 $
0\neq \left(\pi_f((k\;N))\pi_f(q)\xi_f, \pi_f(r) \pi_f((k\;N))A\pi_f((k\;N))\xi_f\right)$\newline $=\left(\pi_f((k\;N))\pi_f(q)\pi_f((k\;N))\xi_f, \pi_f(r) \pi_f((k\;N))A\xi_f\right)=0$. This contradiction finishes the proof
 \end{proof}
 Recall that $[S]$ stands for the norm closure of the linear span of a set of vectors $S$ in a Hilbert space. 
 Let $E_\mathfrak{S}$ be the orthogonal projection $H_f$ onto  $\left[\pi_f\left(\mathfrak{S}_\infty \right)\xi_f \right]$. By Proposition \ref{factor_condition}, the restriction $\pi_f$ to $\mathfrak{S}_\infty$ in Hilbert space $E_\mathfrak{S}H_f$ is ${\rm II}_1$-factor-representation.
 For the reader's convenience we give the important results from \cite{Ok2} in the next proposition.
 \begin{Prop}\label{sgroup_spectral}
    Operator $\mathcal{O}_kE_\mathfrak{S}$ has  the spectral decomposition\\
    $
    \mathcal{O}_kE_\mathfrak{S}= \sum\limits_{n} \lambda_n\, E_k^\mathfrak{S}(\lambda_n)$, where
    \begin{itemize}
      \item $\lambda_n\neq0$, $E_k^\mathfrak{S}(\lambda_n)\neq 0$,
      \item $E_k^\mathfrak{S}(\alpha)$  is the spectral projection, corres\-ponding to    $\alpha\in\mathbb{R}$,
      \item $\left(E_k^\mathfrak{S}(\alpha)\xi_f,\xi_f \right)=
          m_{\alpha}\alpha, \;m_\alpha\in\mathbb{N}\cup0$,
      \item  $\sum\limits_{n} m_{\lambda_n}|\lambda_n|\leq1$.
    \end{itemize} \end{Prop}
 \begin{Lm}
 In the notation of Proposition \ref{sgroup_spectral} operator $\mathcal{O}_k$ has  the spectral decomposition of the form $ \mathcal{O}_k= \sum\limits_{n} \lambda_n\;E_k(\lambda_n)$, where $E_k(\alpha)E_\mathfrak{S}=E_k^\mathfrak{S}(\alpha)$.
 \end{Lm}
 \begin{proof}
 Lemma \ref{separating} implies that the map
 \begin{eqnarray*}
 \left(1-E_k(0) \right)\mathcal{O}_k\left(1-E_k(0) \right)\stackrel{\theta}{\mapsto}E_\mathfrak{S}\left(1-E_k(0) \right)\mathcal{O}_k\left(1-E_k(0) \right)E_\mathfrak{S}
 \end{eqnarray*}
extends to an isomorphism of abelian  $W^\ast$-algebra $\left(1-E_k(0) \right)\{\mathcal{O}_k\}^{\prime\prime}\left(1-E_k(0)\right) $ onto $E_\mathfrak{S}\left(1-E_k(0) \right)\{\mathcal{O}_k\}^{\prime\prime}\left(1-E_k(0) \right)E_\mathfrak{S}$. Therefore, if $\alpha\neq0$ then $\theta^{-1}\left(E_k^\mathfrak{S}(\alpha) \right)=E_k(\alpha)$.
 \end{proof}
 \begin{Rem}\label{Zero_on_cyclic}
  Since $f(e)=1$, then $\left(E_k(0)\xi_f,\xi_f\right)+\sum\limits_{n:\lambda_n\neq0}\left(E_k(\lambda_n)\xi_f,\xi_f\right)=1$.
   Therefore, the equalities  $\sum\limits_{n:\lambda_n\neq0}\left(E_k(\lambda_n)\xi_f,\xi_f\right)\stackrel{\text{Prop.} \ref{sgroup_spectral}}{=}\sum\limits_{n} m_{\lambda_n}|\lambda_n|=1$ and  $E_k(0)\xi_f=0$  are equivalent for all $k\in\mathbb{N}$.
 \end{Rem}
 \subsection{The relations between  $\mathcal{O}_k$ and $\pi_{f}\left(\epsilon_{\{k\}} \right)$.}
 Let us denote by $\mathcal{C}_\mathfrak{S}(r)$ the set $\left\{ srs^{-1} \right\}_{s\in\mathfrak{S}_\infty}$, where $r\in R_\infty$. The proof of the following lemma  is straightforward and we leave it to the reader as an exercise.
 \begin{Lm}
 Given $r_1,r_2\in R_\infty$ suppose that $\mathcal{C}_\mathfrak{S}(r_1)=\mathcal{C}_\mathfrak{S}(r_2)$. Then there exists $s\in\mathfrak{S}_\infty$ such that $r_2=sr_1s^{-1}$ and ${\rm supp}\,s\subseteq{\rm supp}\,r_1\cup{\rm supp}\,r_2$.
 \end{Lm}
 \noindent
 The next statement follows immediately  from the previous lemma and $\mathfrak{S}_\infty$-invariance of $f$.
 \begin{Prop}\label{infty_scalar}
 \noindent
 Let $r_n$  be a sequence of elements from $R_\infty$ such that  ${\rm supp}\,r_n\subseteq [n+1,\infty)$. Suppose that $\{r_n\}_{n\in\mathbb{N}}\subset \mathcal{C}_\mathfrak{S}(r)$ for some $r\in R_\infty$. Then there exists $\lim\limits_{n\to\infty}\pi_f\left(r_n \right)=f(r)I$ with respect to the weak operator topology.
 \end{Prop}
 \noindent Further, let $m>1$ and consider a cycle $c=\left(k_1\;k_2\;\ldots\;k_m \right)$. One has:
 \begin{eqnarray*}
c=\left(k_m\;k_1 \right)\cdot\left(k_{m-1}\;k_1 \right)\cdot\cdots\left(k_3\;k_1 \right)\cdot\left(k_2\;k_1 \right)\\=\left(k_{m-1}\;k_m \right)\cdot\cdots\left(k_2\;k_m \right)\cdot\left(k_1\;k_m \right).
\end{eqnarray*}
It follows that
\begin{eqnarray*}
\begin{split}
\epsilon_{\{k_1\}}\cdot c\cdot\epsilon_{\{k_1\}}=\left(k_{m-1}\;k_m \right)\cdot\cdots\left(k_2\;k_m \right)\cdot\epsilon_{\{k_1\}}\cdot\left(k_1\;k_m \right)\cdot\epsilon_{\{k_1\}}\\
=\left(k_{m-1}\;k_m \right)\cdot\cdots\left(k_2\;k_m \right)\cdot\epsilon_{\{k_m\}}\cdot\epsilon_{\{k_1\}}.
\end{split}
\end{eqnarray*}
Hence, using lemma \ref{Okounkov_operator}, we have for $m>1$
\begin{eqnarray*}
\begin{split}
&\lim\limits_{k_2\to\infty}\lim\limits_{k_3\to\infty}\ldots\lim\limits_{k_{m-1}\to\infty}
\pi_f\left(\epsilon_{\{k_1\}}\cdot c \cdot \epsilon_{\{k_1\}}\right)\\
&=\pi_f\left(\epsilon_{\{k_1\}} \right)\pi_f\left( \left(k_1\;k_m \right) \right) \mathcal{O}_{k_1}^{m-2}\pi_f\left(\epsilon_{\{k_1\}} \right)
=\mathcal{O}_{k_m}^{m-2}\pi_f\left(\epsilon_{\{k_m\}} \right)\pi_f\left(\epsilon_{\{k_1\}} \right).
\end{split}
\end{eqnarray*}
Therefore, letting $k_m\to\infty$, applying Lemma \ref{Okounkov_operator}  and  Proposition \ref{infty_scalar}, we obtain
\begin{eqnarray}\label{O_p_relation}
\pi_f\left(\epsilon_{\{k_1\}} \right) \cdot\mathcal{O}_{k_1}^{m-1}\cdot\pi_f\left(\epsilon_{\{k_1\}} \right)=
\left(\mathcal{O}_{k_1}^{m-2}\cdot\pi_f\left(\epsilon_{\{k_1\}} \right)\xi_f,\xi_f \right)\cdot\pi_f\left(\epsilon_{\{k_1\}} \right).
\end{eqnarray}
Since for every $j$ the spectral orthogonal projection $E_k(\lambda_j)$ can be approximated weakly arbitrarily well by linear combinations of powers of $\mathcal O_k$ we deduce that
\begin{eqnarray}\label{value_state-on_projection}
\lambda_j\cdot\pi_f\left(\epsilon_{\{k\}} \right) \cdot E_{k}\left(\lambda_j \right)\cdot\pi_f\left(\epsilon_{\{k\}} \right)=
\left( E_{k}\left(\lambda_j \right)\cdot\pi_f\left(\epsilon_{\{k\}} \right)\xi_f,\xi_f \right)\cdot\pi_f\left(\epsilon_{\{k\}} \right).
\end{eqnarray}
for each $k\in\mathbb N$ and for each spectral projection $E_{k}\left(\lambda_j \right)$ of $\mathcal O_k$.
Observe that by centrality of $f$ one has:
\begin{eqnarray}\label{E_k_epsilon_0}\begin{split}
\left( E_{k}\left(\lambda_j \right)\pi_f\left(\epsilon_{\{k\}} \right)\xi_f,\xi_f \right)=\left( E_{k}\left(\lambda_j \right)\pi_f\left(\epsilon_{\{k\}} \right) E_{k}\left(\lambda_j \right)\xi_f,\xi_f \right)\\
=\|\pi_f\left(\epsilon_{\{k\}} \right)E_{k}\left(\lambda_j \right)\xi_f\|^2
\geqslant 0.\end{split}
\end{eqnarray}
 Since $\pi_f\left(\epsilon_{\{k\}} \right) \cdot E_{k}\left(\lambda_j \right)\cdot\pi_f\left(\epsilon_{\{k\}} \right)$ is a positive semi-definite operator and $\pi_f(\epsilon_{k})$ is an orthogonal projection, from \eqref{value_state-on_projection} and \eqref{E_k_epsilon_0} we obtain that for $\lambda_j\leqslant 0$ one has \begin{equation}\label{E_k_epsilon_1}\pi_f\left(\epsilon_{\{k\}} \right)E_{k}\left(\lambda_j \right)\xi_f=0.
  \end{equation}
  In particular, using Lemma \ref{separating}, we arrive at
\begin{Co}\label{zero_negative}
If $\lambda<0$ then $ E_{k}\left(\lambda \right)\cdot\pi_f\left(\epsilon_{\{k\}} \right)=0$.
\end{Co}

\subsection{The properties  of the projection $E_k(0)$.}\label{P_k_subsection}
Denote by $P_k$ the orthogonal projection onto the subspace $\left[\mathfrak{A}_k\left(1-E_k(0)\right)H_f \right]$, where as before $[S]$ stands for the closure of the linear span of a collection $S$ of vectors in a Hilbert space.
It is clear that
\begin{eqnarray}\label{inequlity}
I-P_k\leq E_k(0)
\end{eqnarray}
Using Lemma \ref{Okounkov_operator} and the definition of algebra  $\mathfrak{A}_k$, we have
\begin{eqnarray}\label{P_k_permutation}
\pi_f\left(s\right)\cdot P_k\cdot\pi_f\left(s^{-1} \right)=P_{s(k)} \;\;\text{ for all }\;k\in\mathbb{N}\;\text{ and  } s\in\mathfrak{S}_\infty.
\end{eqnarray}
\begin{Rem}
Consider the spherical representation $\pi_{u,v}$ from Proposition \ref{prop_non_zero_lambda}. If $(u,v)\neq0$ then projection $I-E_k(0)$ acts as follows
\begin{eqnarray*}
(I-E_k(0))(\cdots\otimes\eta_{k-1}\otimes\eta_{k}\otimes\eta_{k+1}\otimes\cdots)=
\cdots\otimes\eta_{k-1}\otimes p_u\eta_{k}\otimes\eta_{k+1}\otimes\cdots,
\end{eqnarray*}
where $p_u$ is the orthogonal projection onto line $\mathbb{C}u$. Since
\begin{eqnarray*}
\pi_{u,v}(\epsilon_k)(\cdots\otimes\eta_{k-1}\otimes\eta_{k}\otimes\eta_{k+1}\otimes\cdots)=\cdots\otimes\eta_{k-1}\otimes p_v\eta_{k}\otimes\eta_{k+1}\otimes\cdots,
\end{eqnarray*}
projection $P_k$ acts by
\begin{eqnarray*}
P_k\left( \cdots\otimes\eta_{k-1}\otimes\eta_{k}\otimes\eta_{k+1}\otimes\cdots\right)=\cdots\otimes\eta_{k-1}\otimes(p_u\vee p_v)\eta_{k}\otimes\eta_{k+1}\otimes\cdots,
\end{eqnarray*}
where $(p_u\vee p_v)$ is an orthogonal projection onto subspace $[u,v]\subset \mathbb{C}^m$, gene\-rated by $u$ and $v$.
\end{Rem}
Next statement follows from the bicommutant theorem.
\begin{Lm}\label{P_k_lies_in_center}
  The projection $P_k$ lies in the center of  algebra $\mathfrak{A}_k$, i. e. $P_k\in \mathfrak{A}_k\cap \mathfrak{A}_k^\prime$.
\end{Lm}
\begin{Lm}\label{1-p_lemma}
  $(I-P_k)\cdot\pi_f\left(\epsilon_{\{k\}}\right)=0.$
\end{Lm}
\begin{proof}
  It suffices to show that
  \begin{eqnarray}\label{ac_relation}
  \left((I-P_k)\cdot\pi_f\left(\epsilon_{\{k\}}\right)\pi_f(r_1)\xi_f, \pi_f(r_2)\xi_f\right)=0\;\;\text{ for all }\; r_1,r_2\in R_\infty.
  \end{eqnarray}
   Using (\ref{decomposition_into_product}), we can find  subsets $\mathbb{A}_1\subset\mathbb{N}$, $\mathbb{A}_2\subset\mathbb{N}$ and elements $s_1$, $s_2$ $\in$ $\mathfrak{S}_\infty$ wich  satisfy the relations
   \begin{eqnarray*}
   r_1=\epsilon_{\mathbb{A}_1}\cdot s_1,\; r_2=\epsilon_{\mathbb{A}_2}\cdot s_2.
   \end{eqnarray*}
   Hence, using lemma \ref{P_k_lies_in_center} and (\ref{P_k_permutation}), we have
   \begin{eqnarray*}
   \left((I-P_k)\cdot\pi_f\left(\epsilon_{\{k\}}\right)\pi_f(r_1)\xi_f, \pi_f(r_2)\xi_f\right)=\left((I-P_l)\cdot\pi_f\left(\epsilon_{\mathbb{B}}\right)\cdot\pi_f\left(t \right)\xi_f,\xi_f \right),
   \end{eqnarray*}
   where $l=s_2^{-1}(k)$, $\mathbb{B}=s_2^{-1}(\mathbb{A}_1\cup\mathbb{A}_2\cup k)$, $t=s_2^{-1}\cdot s_1$.
   Therefore, to prove (\ref{ac_relation}) sufficient to show that
   \begin{eqnarray*}
  \pi_f\left(\epsilon_{\mathbb{B}}\right)\cdot (I-P_l)\xi_f=0\;\;\text{ for any }\; \mathbb{B} \;\text{ containing } l.
   \end{eqnarray*}
   But the last relation follows from (\ref{E_k_epsilon_1}), (\ref{inequlity}) and lemma \ref{P_k_lies_in_center}.
   \end{proof}

   \begin{Rem}\label{remark_c(0)} For $\lambda_j>0$ set $c(\lambda_j)=\lambda_j^{-1}\left( E_{k}\left(\lambda_j \right)\cdot\pi_f\left(\epsilon_{\{k\}} \right)\xi_f,\xi_f \right)$.
   Since the equality $\sum\limits_{j}\pi(\epsilon_k)\,   E_k(\lambda_j)\,  \pi_f(\epsilon_k)=\pi_f(\epsilon_k)$ holds, using (\ref{value_state-on_projection}) and Corollary \ref{zero_negative}, we obtain:
   \begin{eqnarray}\label{projection_relations_zero}
   \pi_f(\epsilon_k)\,   E_k(0)\,  \pi_f(\epsilon_k)=c(0)\cdot \,  \pi_f(\epsilon_k), \text{ where } c(0)=1-\sum\limits_{j:\lambda_j>0}c(\lambda_j)\geq 0.
   \end{eqnarray}
   The following cases are possible.
 \vskip 0.2cm\noindent
 ${\bf 1)}$ If $c(0)=1$ then $\pi_f\left(\epsilon_{\{k\}} \right)\leq E_k(0)$. In this case, (\ref{E_k_epsilon_0}) implies that $\pi_f\left(\epsilon_{\{k\}} \right)\xi_f=0$. Therefore, $\pi_f\left(\epsilon_{\{k\}} \right)=0$ and $P_k=I-E_k(0)$.
\vskip 0.2cm\noindent
${\bf 2)}$  If $c(0)=0$ then $\pi_f\left(\epsilon_{\{k\}}\right)\leq I-E_k(0)$ and $P_k=I-E_k(0)$.
\vskip 0.2cm\noindent
 ${\bf 3)}$ Assume that $c(0)\in (0,1)$. A simple algebraic computation shows that $c(0)^{-1}E_k(0)\pi\left(\epsilon_{\{k\}}\right)E_k(0)$ is an orthogonal projection. Let us now show the following: \begin{equation}\label{Pk_formula}P_k=I-E_k(0)+{\displaystyle c(0)^{-1}E_k(0)\pi_f\left(\epsilon_{\{k\}}\right) E_k(0)}.\end{equation}
 Indeed, let $Q_k=I-E_k(0)+{\displaystyle c(0)^{-1}E_k(0)\pi_f\left(\epsilon_{\{k\}}\right) E_k(0)}$. Then algebraic computations show that $$Q_k\mathcal O_k=\mathcal O_k,\;\;Q_k(I-E_k(0))=I-E_k(0),\;\;Q_k\pi_f(\epsilon_k)=\pi_f(\epsilon_{\{k\}})$$ (use \eqref{value_state-on_projection} for the last one). Therefore, $Q_k$ acts as identity of the subspace $[\mathfrak A_k(I-E_k(0))H_f]$ and $Q_k\geq P_k$. On the other hand, using Lemmas \ref{P_k_lies_in_center} and \ref{1-p_lemma}, we obtain: \begin{align*}P_kQ_k=P_k(I-E_k(0))+c(0)^{-1}E_k(0)P_k\pi_f(\epsilon_{\{k\}})E_k(0)
 \\= I-E_k(0)+{\displaystyle c(0)^{-1}E_k(0)\pi_f\left(\epsilon_{\{k\}}\right) E_k(0)}=Q_k,\end{align*} and so $Q_k\leq P_k$.

  As a consequence we obtain:
  \begin{eqnarray}\label{supp_xi_f}
  \left(P_k-(I-E_k(0))\right)\xi_f={\displaystyle c(0)^{-1} E_k(0)\pi_f\left(\epsilon_{\{k\}}\right) E_k(0)}\xi_f\stackrel{(\ref{E_k_epsilon_1})}{=}0.
  \end{eqnarray}
\end{Rem}
 This Remark and (\ref{value_state-on_projection}) immediately yield the following result:
\begin{Lm} \label{minimal_projection}
Let $\mathbb{F}$ be a finite subset from $\mathbb{N}$ and $R_\mathbb{F}=\left\{r\in R_\infty: {\rm supp}\, r\subset \mathbb{F}  \right\}$
Operator $\prod\limits_{k\in\mathbb{F}}\pi_f\left(\epsilon_{\{k\}} \right)$ is a minimal orthogonal projection in the algebra gene\-rated by $\left\{\mathfrak{A}_k\right\}_{k\in\mathbb{F}}$ and $R_\mathbb{F}$.
\end{Lm}

   \begin{Lm}
     The following conditions are equivalent:
     \begin{itemize}
       \item[\bf 1)] $\sum\limits_{n:\lambda_n\neq0} m_{\lambda_n}|\lambda_n|=1$ (see Proposition \ref{sgroup_spectral});
       \item[\bf 2)] $P_k=I$.
     \end{itemize}
   \end{Lm}
   \begin{proof} Assume that $1)$ holds.  First, we notice that
   \begin{eqnarray}
   \begin{split}
  \left\{\ldots\mathfrak{A}_{k-2}\mathfrak{A}_{k-1}\mathfrak{A}_{k+1}\mathfrak{A}_{k+2}\ldots \right\}\subset \mathfrak{A}_k^\prime \text{ and }\\
  H_f=\left[\pi_f\left(R_\infty \right)\xi_f \right]=\left[\mathfrak{A}_1\mathfrak{A}_2\mathfrak{A}_3\cdots\pi_f\left(\mathfrak{S}_\infty \right) \xi_f\right].\label{sequenceAAAA}
   \end{split}
   \end{eqnarray}
   Hence,
   \begin{eqnarray}\label{I-E H}
   \left(I-E_k(0) \right)H_f\supseteq \left[\cdots\mathfrak{A}_{k-2}\mathfrak{A}_{k-1}\mathfrak{A}_{k+1}\mathfrak{A}_{k+2}\cdots   \left(I-E_k(0) \right)\pi_f\left(\mathfrak{S}_\infty \right)\xi_f \right].
   \end{eqnarray}
Since $\sum\limits_{n:\lambda_n\neq0} m_{\lambda_n}|\lambda_n|=1$ then, by Remark \ref{Zero_on_cyclic} (page \pageref{Zero_on_cyclic}), we have
\begin{eqnarray*}
E_k(0)\xi_f=0\;\;\text{ for all } k.
\end{eqnarray*}
It follows that
\begin{eqnarray*}
\left[\left(I-E_k(0) \right) \pi_f\left(\mathfrak{S}_\infty \right) \xi_f\right]=\left[ \pi_f\left(\mathfrak{S}_\infty \right) \xi_f\right] \;\;\text{ for all } k.
\end{eqnarray*}
Therefore, using (\ref{I-E H}), we obtain
\begin{eqnarray*}
 \left(I-E_k(0) \right)H_f\supseteq \left[\cdots\mathfrak{A}_{k-2}\mathfrak{A}_{k-1}\mathfrak{A}_{k+1}\mathfrak{A}_{k+2}\cdots \pi_f\left(\mathfrak{S}_\infty \right)\xi_f \right]
\end{eqnarray*}
Hence, applying  (\ref{sequenceAAAA}), we have
\begin{eqnarray*}
\left[ \mathfrak{A}_k\left(I-E_k(0) \right)H_f\right]= H_f.
\end{eqnarray*}
This proves property {\bf 2}).
\vskip 0.1cm
 Now we assume the condition  {\bf 2}) holds. Then
$$\left[\mathfrak{A}_k\left(I-E_k(0)) \right)H_f \right]\stackrel{\text{condition  {\bf 2})}}{=} H_f.$$ Consider the subset $\mathcal U=\{E_k(\lambda_n)\pi_f(\epsilon_{\{k\}})E_k(\lambda_m)\}_{m,n}\cup\{E_k(\lambda_n)\}_n\subset \mathfrak{A}_k$. Notice that it's linear hull ${\rm Lin}(\mathcal U)$ is invariant under conjugation and multiplication by $\pi_f(\epsilon_{\{k\}})$ (by \eqref{value_state-on_projection}) and $\mathcal O_k$ from the left and from the right. It follows that ${\rm Lin}(\mathcal U)$ is weakly dense in $\mathfrak A_k$.
We arrive at:
$$\left[\cup_{m,n}\left(E_k(\lambda_n)\pi_f(\epsilon_{\{k\}})E_k(\lambda_m)(I-E_k(0) H_f\right)\cup (I-E_k(0))H_f \right]=H_f.$$ From the above and \eqref{E_k_epsilon_1} it follows that \begin{eqnarray*}
\left(\eta,E_k(0)\xi_f \right)=0 \;\;\text{ for all }\;\eta \in H_f,\end{eqnarray*} which shows that $E_k(0)\xi_f=0$. Applying Remark \ref{Zero_on_cyclic} (page \pageref{Zero_on_cyclic}) again, we obtain $1)$.
\end{proof}
\subsection{The support of the state $\omega_f(x)=\left(x\xi_f,\xi_f \right)$, $x\in\left\{\pi_f\left(R_\infty \right) \right\}^{\prime\prime}$. \label{supp_comp}} First we notice that operator $\mathfrak{P}_k=I-P_k+I-E_k(0)=2I-P_k-E_k(0)$ is an orthogonal projection from $\mathfrak{A}_k$ (see Lemma \ref{P_k_lies_in_center}). Since $\mathfrak{P}_k$ commute pairwise, the projections $Q^{(n)}=\prod\limits_{k=1}^n\mathfrak{P}_k$ satisfy the inequality $Q^{(n)}\geq Q^{(n+1)}$. Therefore, $Q^{(n)}$ are convergent in the strong operator topology to an orthogonal projection $Q^{(\infty)}$.
Now we notice that, by definition of $Q^{(\infty)}$ and (\ref{supp_xi_f}), the next relations hold
\begin{eqnarray}\label{property_Q_infty}
Q^{(\infty)}\xi_f=\xi_f \text{ and } Q^{(\infty)}\cdot\pi_f(s)=\pi_f(s)\cdot Q^{(\infty)} \text{ for all } s\in\mathfrak{S}_\infty.
\end{eqnarray}
Recall the notion of the support of a weakly continuous nonnegative functional $\varphi$ on a $W^*$-algebra $M$ (\cite{TAKES1}, page 134).
\begin{Def}\label{Def-Support}
  There exists a unique orthogonal projection $e\in M$ satisfying the following properties:
  \begin{itemize}
    \item[\bf a)] $\varphi(x)=\varphi(exe)$ for all $x\in M$;
    \item[\bf b)] if a projection $p\in M$ is such that $p\leq e$ and $\varphi(p)=0$, then $p=0$.
  \end{itemize}
  The projection $e$ is called the {\it support} projection of $\varphi$ and denoted by ${\rm supp}\,\varphi$.
\end{Def}
\begin{Th}\label{supp_th}
${\rm supp}\,\omega_f=Q^{(\infty)}$.
\end{Th}
Before proving Theorem \ref{supp_th}, let us prove the following auxiliary statement.
\begin{Lm}\label{operator_from_commutant}
Let $k\in\mathbb N$. The map $$H_f\ni \pi_f(r)\xi_f\stackrel{\mathfrak{R}\left(\mathcal{O}_k\cdot\pi_f\left(\epsilon_{\{k\}} \right) \right)}{\mapsto}\pi_f(r)\cdot \mathcal{O}_k\cdot\pi_f\left(\epsilon_{\{k\}} \right)\xi_f\in  H_f,\;\; r\in R_\infty,$$ extends to a bounded linear operator belonging to $\pi_f(R_\infty)^\prime$ and $$\left\| \mathfrak{R}\left(\mathcal{O}_k\cdot\pi_f\left(\epsilon_{\{k\}} \right) \right)\right\|\leq 1.$$
\end{Lm}
\begin{proof}
Consider arbitrary elements $u=\sum\limits_{i=1}^{N^{(u)}}c_i^{(u)}\cdot\pi_f\left( r_i^{(u)} \right)\xi_f$, $v=\sum\limits_{i=1}^{N^{(v)}}c_i^{(v)}\cdot\pi\left( r_i^{(v)} \right)\xi_f$ from $H_f$, where  $c_i^{(u)},c_i^{(v)}\in\mathbb{C}$, $r_i^{(u)},r_i^{(v)}\in R_n$ $(n<\infty)$ and $N^{(u)},N^{(v)}<\infty$. For convenience suppose that $\|u\|_{H_f}=\|v\|_{H_f}=1$. Under these conditions we have
\begin{eqnarray*}
&\left(\mathfrak{R}\left(\mathcal{O}_k\cdot\pi_f\left(\epsilon_{\{k\}} \right) \right)u,v  \right)_{H_f}\\
&=\left(\sum\limits_{i=1}^{N^{(u)}}c_i^{(u)}\cdot\pi\left( r_i^{(u)} \right)\cdot\mathcal{O}_k\cdot\pi_f\left(\epsilon_{\{k\}} \right)\xi_f,\sum\limits_{i=1}^{N^{(v)}}c_i^{(v)}\cdot\pi\left( r_i^{(v)} \right)\xi_f\right)\\
&\stackrel{\text{Lemma \ref{Okounkov_operator}}}{=} \lim\limits_{N\to\infty}\left(\sum\limits_{i=1}^{N^{(u)}}c_i^{(u)}\cdot\pi\left( r_i^{(u)} \right)\cdot\pi_f((k\;N))\cdot\pi_f\left(\epsilon_{\{k\}} \right)\xi_f,\sum\limits_{i=1}^{N^{(v)}}c_i^{(v)}\cdot\pi\left( r_i^{(v)} \right)\xi_f\right)\\
&\stackrel{N>n}{=}\left(\sum\limits_{i=1}^{N^{(u)}}c_i^{(u)}\cdot\pi\left( r_i^{(u)} \right)\cdot\pi_f((k\;N))\cdot\pi_f\left(\epsilon_{\{k\}} \right)\xi_f,\sum\limits_{i=1}^{N^{(v)}}c_i^{(v)}\cdot\pi\left( r_i^{(v)} \right)\xi_f\right)\\
&=\left(\sum\limits_{i=1}^{N^{(u)}}c_i^{(u)}\cdot\pi\left( r_i^{(u)} \right)\cdot\pi_f\left(\epsilon_{\{N\}} \right)\cdot\pi_f((k\;N))\xi_f,\sum\limits_{i=1}^{N^{(v)}}c_i^{(v)}\cdot\pi\left( r_i^{(v)} \right)\xi_f\right)
\end{eqnarray*}
\begin{eqnarray*}
&\stackrel{N>n}{=}\left(\pi_f\left(\epsilon_{\{N\}} \right)\sum\limits_{i=1}^{N^{(u)}}c_i^{(u)}\cdot\pi\left( r_i^{(u)} \right)\cdot\pi_f((k\;N))\xi_f,\sum\limits_{i=1}^{N^{(v)}}c_i^{(v)}\cdot\pi\left( r_i^{(v)} \right)\xi_f\right)\\
&=\left(\pi_f\left(\epsilon_{\{N\}} \right)\sum\limits_{i=1}^{N^{(u)}}c_i^{(u)}\cdot\pi\left( r_i^{(u)} \right)\xi_f,\sum\limits_{i=1}^{N^{(v)}}c_i^{(v)}\cdot\pi\left( r_i^{(v)} \right)\cdot\pi_f((k\;N))\xi_f\right).
\end{eqnarray*}
Hence, using the equality $\left\| \sum\limits_{i=1}^{N^{(v)}}c_i^{(v)}\cdot\pi\left( r_i^{(v)} \right)\cdot\pi_f((k\;N))\xi_f\right\|_{H_f}=\|v\|_{H_f}=1$, we obtain $\left| \left(\mathfrak{R}\left(\mathcal{O}_k\cdot\pi_f\left(\epsilon_{\{k\}} \right) \right)u,v  \right)_{H_f}\right|\leq 1$.
\end{proof}
\begin{proof}[{\bf Proof of Theorem \ref{supp_th}}]  It suffices to show that
\begin{eqnarray}\label{cuclic_Q}
Q^{(\infty)}H_f=\left[ \pi_f\left( R_\infty \right)^\prime\xi_f \right].
\end{eqnarray}
Indeed, assume that \eqref{cuclic_Q} holds. If $\omega_f(q)=\left(q\xi_f,\xi_f \right)=0$, where $q\leq Q^{(\infty)}$ is  an orthogonal projection from  $\pi_f\left(R_\infty\right)''$, then $q\xi_f=0$. Therefore, $qA'\xi_f=0$ for all $A'\in \pi_f\left(R_\infty\right)'$. Since $q\leq Q^{(\infty)}$, using \eqref{cuclic_Q}, we conclude that $q\eta=0$ for all $\eta\in H_f$. 
Since $Q^{(\infty)}$ belongs to $\pi_f(R_\infty)''$ and $Q^{(\infty)}\xi_f=\xi_f$ we obtain that it satisfies the conditions of Definition \ref{Def-Support}.

Let us prove \eqref{cuclic_Q}. Since $f$ is $\mathfrak{S}_\infty$-invariant, for any $s\in\mathfrak{S}_\infty$ the map $\pi_f(r)\xi_f\stackrel{\mathfrak{R}\left(\pi_f(s)\right)}{\mapsto}\pi_f(r)\cdot\pi_f(s)\xi_f$ defines a unitary operator from $\pi_f\left( R_\infty \right)^\prime$. Therefore,
\begin{eqnarray}\label{cuclic_symmetrsc_group}
\left[ \pi_f\left( \mathfrak{S}_\infty \right)\xi_f \right]\stackrel{(\ref{property_Q_infty})}{=}Q^{(\infty)}\left[ \pi_f\left( \mathfrak{S}_\infty \right)\xi_f \right]\subset\left[ \pi_f\left( R_\infty \right)^\prime\xi_f \right].
\end{eqnarray}
Set $$A_k^{(\delta)}=\sum\limits_{j:\lambda_j>\delta}\lambda_j^{-1}\cdot E_k(\lambda_j).$$ Then the mapping
\begin{eqnarray}
H_f\ni\pi_f(r)\xi_f\stackrel{\mathfrak{R}(A_k^{(\delta)})}{\mapsto}\pi_f(r)\cdot A_k^{(\delta)}\xi_f \in H_f
\end{eqnarray}
can be extended by continuity to a bounded operator $\mathfrak{R}(A_k^{(\delta)})\in \pi_f\left( R_\infty \right)^\prime $. Since $f$ is $\mathfrak S_\infty$-central and $A_k^{(\delta)}\in\pi(\mathfrak S_\infty)''$, we have $$\omega_f\left(\pi_f(r)A_k^{(\delta)} \right)=\omega_f \left(A_k^{(\delta)}\pi_f(r) \right) ~\text{ for all}~ r\in R_\infty. $$
It easy to check that
\begin{eqnarray}
\left\| \mathfrak{R}(A_k^{(\delta)})\right\| <\delta^{-1}.
\end{eqnarray}
Further, using Lemma \ref{operator_from_commutant}, let us define the operators
\begin{eqnarray}
\mathfrak{R}\left( E_k((\delta,1])\cdot \pi_f(\epsilon_{\{k\}})\right)= \mathfrak{R}(A_k^{(\delta)})\cdot\mathfrak{R}\left(\mathcal{O}_k\cdot\pi_f\left(\epsilon_{\{k\}}\right)\right).
\end{eqnarray}
Then
$\left\|\mathfrak{R}\left( E_k((\delta,1])\cdot \pi_f(\epsilon_{\{k\}})\right) \right\|\leq \delta^{-1}.$
Therefore, for any $j\in\mathbb N$ and any positive integers $k_1<k_2<\ldots<k_j$ one has
\begin{eqnarray}
\left\| \prod\limits_{j=1}^l \mathfrak{R}\left( E_{k_j}((\delta,1])\cdot \pi_f(\epsilon_{\{k_j\}}) \right)\right\|\leq\delta^{-l}
\end{eqnarray}
and
\begin{eqnarray}
\prod\limits_{j=1}^l \left(I-E_{k_j}(0)\right)\cdot\pi_f(\epsilon_{\{k_j\}})\,\xi_f=\lim_{\delta\to 0}\prod\limits_{j=1}^l \mathfrak{R}\left( E_{k_j}((\delta,1])\cdot \pi_f(\epsilon_{\{k_j\}}) \right)\xi_f.
\end{eqnarray}
Hence, combining Lemma  \ref{1-p_lemma} and Corollary \ref{zero_negative} with \eqref{property_Q_infty}, we have
\begin{eqnarray}
\begin{split}
&\prod\limits_{j=1}^l Q^{(k_j)}\cdot\pi_f(\epsilon_{\{k_j\}})\,\xi_f=Q^{(\infty)}\prod\limits_{j=1}^l \pi_f(\epsilon_{\{k_j\}})\,\xi_f\\
&=\lim_{\delta\to 0}\prod\limits_{j=1}^l \mathfrak{R}\left( E_{k_j}((\delta,1])\cdot \pi_f(\epsilon_{\{k_j\}}) \right)\xi_f.
\end{split}
\end{eqnarray}
Therefore, $Q^{(\infty)}\prod\limits_{j=1}^l \pi_f(\epsilon_{\{k_j\}})\,\xi_f\in\left[\pi_f(R_\infty)'\xi_f \right]$. Hence, applying \eqref{cuclic_symmetrsc_group}, we obtain that $Q^{(\infty)}\pi_f(r)\xi_f\in\left[\pi_f(R_\infty)'\xi_f \right]$ for all $r\in R_\infty$, and \eqref{cuclic_Q} is proved.
\end{proof}
\subsection{The structure of the algebra $\mathfrak{A}_k$.} By definition (see paragraph \ref{asumpthotic_operators}), algebra $\mathfrak{A}_k$ is generated by  the spectral projections $E_k(\lambda_j)$ and $\pi_f(\epsilon_{\{k\}})$. Let $S(\mathcal{O}_k)$ be the spectrum of $\mathcal{O}_k$ and let $$\widehat{S}(\mathcal{O}_k)=\left\{\lambda\in S(\mathcal{O}_k): E_k(\lambda)\;\pi_f(\epsilon_{\{k\}})\neq 0\right\}.$$
Applying \eqref{value_state-on_projection} and Remark \ref{remark_c(0)}, we obtain that for each $\lambda \in \widehat{S}(\mathcal{O}_k)$ operator $$u_k(\lambda)=c(\lambda)^{-1/2}\,\,E_k(\lambda)\;\pi_f(\epsilon_{\{k\}})$$ is a partial isometry. In particular,
\begin{eqnarray}\label{partial_iso}
u_k(\lambda)\;u_k(\lambda)^*=\mathbf{e}_k(\lambda) \;\;\text{ and } \;\;\;u_k(\lambda)^*\;u_k(\lambda)=\pi_f(\epsilon_{\{k\}}),
\end{eqnarray} where $\mathbf{e}_k(\lambda)=c(\lambda)^{-1}\,E_k(\lambda)\;\pi_f(\epsilon_{\{k\}})\;E_k(\lambda)$.
Since the projections $E_k(\lambda)-\mathbf{e}_k(\lambda)$ and $\pi_f(\epsilon_{\{k\}})$ are orthogonal for each $\lambda\in \widehat{S}(\mathcal{O}_k)$, we obtain
\begin{eqnarray*}
\pi_f(\epsilon_{\{k\}})\;\sum\limits_{\lambda\in  \widehat{S}(\mathcal{O}_k)}\mathbf{e}_k(\lambda)=\pi_f(\epsilon_{\{k\}}).
\end{eqnarray*}
Therefore, the orthogonal projection $\mathbf{e}_k=\sum\limits_{\lambda\in  \widehat{S}(\mathcal{O}_k)}\mathbf{e}_k(\lambda)$ lies in centrum of $\mathfrak{A}_k$. It follows from (\ref{partial_iso}) that algebra $\mathbf{e}_k\,\mathfrak{A}_k\,\mathbf{e}_k$ is a type I factor. Furthermore, algebra $\left( {\rm I}-\mathbf{e}_k \right)\mathfrak{A}_k\left( {\rm I}-\mathbf{e}_k \right)$ is generated by the operator $\left( {\rm I}-\mathbf{e}_k \right)\mathcal{O}_k\left( {\rm I}-\mathbf{e}_k \right)$.
\vskip 0.2cm
 \subsubsection{The proof of theorem \ref{semifinite_repr}. } Recall that $\mathfrak{P}_k$ is the orthogonal projection $I-P_k+I-E_k(0)$, where $P_k$ was introduced in subsection \ref{P_k_subsection}. One has
 \begin{eqnarray}
 Q^{(\infty)}=\prod\limits_{k=1}^{\infty}\mathfrak{P}_k.
 \end{eqnarray}
 Let us introduce the operators $\widetilde{\pi}(r)=Q^{(\infty)}\,\pi_f(r)\,Q^{(\infty)}$, $r\in R_\infty$. By \eqref{property_Q_infty}, we have
 \begin{eqnarray}\label{repr_of_S_infinity}
 \widetilde{\pi}(st)=\widetilde{\pi}(s)\,\widetilde{\pi}(t) ~\text{ and }~\widetilde{\pi}(s)\,\widetilde{\pi}(s^{-1})=Q^{(\infty)} ~\text{ for all }~
s,t\in\mathfrak{S}_\infty.
 \end{eqnarray}
 Therefore, the restriction of  $\widetilde{\pi}$ to $\mathfrak{S}_\infty$ is an unitary representation of $\mathfrak{S}_\infty$  in Hilbert space $Q^{(\infty)}\,H_f$.
By definition of $\mathfrak{P}_k$, \eqref{projection_relations_zero}  and Lemma \ref{1-p_lemma}, we have
\begin{eqnarray}\label{minimality_rel}
\widetilde{\pi}(\epsilon_{\{k\}}) \,\widetilde{\pi}(\epsilon_{\{k\}})=(1-c(0))\,\widetilde{\pi}(\epsilon_{\{k\}}).
\end{eqnarray}
Therefore, if $c(0)<1$ then operator  $\widehat{\pi}(\epsilon_{\{k\}})=(1-c(0))^{-1}\,\widetilde{\pi}(\epsilon_{\{k\}})$ is an orthogonal projection.

In the case $1-c(0)=0$, applying \eqref{projection_relations_zero},  we obtain $$\pi_f(\epsilon_{\{k\}})\,E_k(0)\,\pi_f(\epsilon_{\{k\}})=\pi_f(\epsilon_{\{k\}}).$$
Hence, applying Lemma  \ref{1-p_lemma} again, we have
\begin{eqnarray}
\tilde\pi_f(\epsilon_{\{k\}})=Q^{(\infty)}\,\pi_f(\epsilon_{\{k\}})\, Q^{(\infty)}=Q^{(\infty)}\,\mathfrak{P}_k\,\pi_f(\epsilon_{\{k\}})\,\mathfrak{P}_k\, Q^{(\infty)}=0.
\end{eqnarray}
Let us prove that the mappings $$\epsilon_k\mapsto \widehat{\pi}(\epsilon_k)=\left\{
 \begin{array}{rl}
 (1-c(0))^{-1}\,\widetilde{\pi}(\epsilon_{\{k\}}),&\text{ if } c(0)<1,\\
 0,&\text{ if } c(0)=1,
 \end{array}\right.\mathfrak{S}_\infty\ni s\mapsto \widehat{\pi}(s)=\widetilde{\pi}(s)$$ extends to a  $\ast$-factor-representation of  $R_\infty$ in Hilbert space $Q^{(\infty)}\,H_f$.

 First, we recall Popova’s presentation \cite{Popova}, \cite{East} for $R_\infty$, which also extends Moore’s presentation \cite{Moore}
for $\mathfrak{S}_\infty$. This presentation is defined by the set of generators  $\mathfrak{Al}=\left\{s_i \right\}_{i=1}^\infty \cup \epsilon_{\{1\}}$, where $s_i$ is transposition $(i\;\;i+1)$,  and the following system of  the relations:
\begin{eqnarray*}
&s_i^2=e ~\text{ where }~ e ~\text{ is identity of }~R_\infty,\\
&s_i\,s_j=s_j\,s_i ~\text{ if }~ |i-j|>1,\\
&s_i\,s_{i+1}\,s_i=s_{i+1}\,s_i\,s_{i+1},\\
 &\epsilon_{\{1\}}^2= \epsilon_{\{1\}},\\
 &\epsilon_{\{1\}}\,s_1 \epsilon_{\{1\}}\,s_1=s_1\, \epsilon_{\{1\}}\,s_1\, \epsilon_{\{1\}} =\epsilon_{\{1\}}\, s_1\, \epsilon_{\{1\}}.
\end{eqnarray*}
It follows from \eqref{repr_of_S_infinity} that it is sufficient  to verify the last two relations. Namely,
\begin{eqnarray}\label{first_rel}
 \widehat{\pi}(\epsilon_{\{1\}})\, \widehat{\pi}(\epsilon_{\{1\}}))= \widehat{\pi}(\epsilon_{\{1\}}))
 \end{eqnarray}
 and
 \begin{eqnarray}\label{second_rel}
 \begin{split}
 &\widehat{\pi}(\epsilon_{\{1\}})\, \widehat{\pi}(s_1)\,  \widehat{\pi}(\epsilon_{\{1\}}) \,\widehat{\pi}(s_1) =\widehat{\pi}(s_1)\,\widehat{\pi}(\epsilon_{\{1\}})\,\widehat{\pi}(s_1)\,\,\widehat{\pi}(\epsilon_{\{1\}})\\
& =\widehat{\pi}(\epsilon_{\{1\}})\,\widehat{\pi}(s_1)\,\,\widehat{\pi}(\epsilon_{\{1\}}).
\end{split}
\end{eqnarray}
To prove equalities \eqref{first_rel} and \eqref{second_rel}  we recall  that $\mathfrak{P}_k=I-P_k+I-E_k(0)\in\mathfrak{A}_k$, where $P_k$ is defined in subsection \ref{P_k_subsection}, $\mathfrak{A}_k\subset \mathfrak{A}_j'$ for all $k\neq j$ and
\begin{eqnarray*}
\pi_f(s)\,\mathfrak{P}_{i_1}\,\mathfrak{P}_{i_2}\,\cdots\,,\mathfrak{P}_{i_m}\,\pi_f(s^{-1})=\mathfrak{P}_{s(i_1)}\,\mathfrak{P}_{s(i_2)}\,
\cdots\,,\mathfrak{P}_{s(i_m)} ~\text{ for all }~ s\in\mathfrak{S}_\infty.
\end{eqnarray*}

It is clear that \eqref{first_rel}  follows from \eqref{minimality_rel}.

To prove first equality from \eqref{second_rel} we note that $\widehat{\pi}(\epsilon_{\{1\}}))=0$ in the case $c(0)=1$, and $\widehat{\pi}(\epsilon_{\{1\}})=(1-c(0))^{-1}\,Q^{(\infty)}\,\pi_f(\epsilon_{\{1\}}))\, Q^{(\infty)}$, if $c(0)<1$. Therefore,
$\widehat{\pi}(s_1)\,  \widehat{\pi}(\epsilon_{\{1\}}) \,
\widehat{\pi}(s_1)= \left\{\begin{array}{rl}
 (1-c(0))^{-1}\,Q^{(\infty)}\pi_f(\epsilon_{\{2\}})\, Q^{(\infty)},&\text{ if } c(0)<1\\
 0,&\text{ if } c(0)=1.
 \end{array}\right.$
 It follows from this that it is sufficient to consider the case, when $c(0)<1$.

 Using definition of $\widehat{\pi}$ we have
 \begin{eqnarray}\label{first_equality}
 \begin{split}
 \widehat{\pi}(\epsilon_{\{1\}})\, \widehat{\pi}(s_1)\,  \widehat{\pi}(\epsilon_{\{1\}}) \,\widehat{\pi}(s_1) =(1-c(0))^{-2}\,Q^{(\infty)}\pi_f(\epsilon_{\{2\}})\, Q^{(\infty)}\,\pi_f(\epsilon_{\{1\}})\, Q^{(\infty)}\\
 =(1-c(0))^{-2} \left(\prod\limits_{j=3}^\infty\mathfrak{P}_j\right)\;\mathfrak{P}_2\,\pi_f(\epsilon_{\{2\}})\,\mathfrak{P}_2\;\mathfrak{P}_1\,
 \pi_f(\epsilon_{\{1\}})\,\mathfrak{P}_1\\
= (1-c(0))^{-2} \left(\prod\limits_{j=3}^\infty\mathfrak{P}_j\right)\;\mathfrak{P}_1\,\pi_f(\epsilon_{\{1\}})\,\mathfrak{P}_1\;\mathfrak{P}_2\,
 \pi_f(\epsilon_{\{2\}})\,\mathfrak{P}_2\\
 =(1-c(0))^{-2}\,Q^{(\infty)}\pi_f(\epsilon_{\{1\}})\, Q^{(\infty)}\,\pi_f(\epsilon_{\{2\}})\, Q^{(\infty)}= \widehat{\pi}(s_1)\,  \widehat{\pi}(\epsilon_{\{1\}}) \,\widehat{\pi}(s_1)\,\widehat{\pi}(\epsilon_{\{1\}}).
 \end{split}
 \end{eqnarray}
 This completes the proof of first equality from \eqref{second_rel}.

The proof of the second equality in \eqref{second_rel} is sufficient to carry out only in the case $c(0)<1$. By definition of $\widehat\pi$, taking into account that $\pi_f(s_1)\mathfrak P_1\pi_f(s_1)=\mathfrak P_2$ and that $\mathfrak P_2$ commutes with $\pi_f(\epsilon_1)$, 
 we obtain
\begin{eqnarray*}\label{second_first_rel}
\begin{split}
\widehat{\pi}(\epsilon_{\{1\}})\,\widehat{\pi}(s_1)\,\,\widehat{\pi}(\epsilon_{\{1\}})=(1-c(0))^{-2}\,Q^{(\infty)}\pi_f(\epsilon_{\{1\}})\, Q^{(\infty)}\,\pi_f(s_1)\,\pi_f(\epsilon_{\{1\}})\, Q^{(\infty)}\\
=(1-c(0))^{-2} \left(\prod\limits_{j=3}^\infty\mathfrak{P}_j\right)\,\mathfrak{P}_1\mathfrak{P}_2\,\pi_f(\epsilon_{\{1\}})\,\mathfrak{P}_1\mathfrak{P}_2\,\pi_f(s_1)\,                 \pi_f(\epsilon_{\{1\}})\,\mathfrak{P}_1\mathfrak{P}_2\\
=(1-c(0))^{-2} \left(\prod\limits_{j=3}^\infty\mathfrak{P}_j\right)\,\mathfrak{P}_1\mathfrak{P}_2\,\pi_f(\epsilon_{\{1\}})\,\pi_f(s_1)\,                 \pi_f(\epsilon_{\{1\}})\,\mathfrak{P}_1\mathfrak{P}_2\\
=(1-c(0))^{-2} \left(\prod\limits_{j=3}^\infty\mathfrak{P}_j\right)\,\mathfrak{P}_1\,\pi_f(\epsilon_{\{1\}})\,
\pi_f(\epsilon_{\{2\}})\,\mathfrak{P}_2.
\end{split}
\end{eqnarray*}
Hence, applying \eqref{first_equality}, we obtain second equality from \eqref{second_rel}. Therefore, $\widehat{\pi}$
extends to a  $\ast$-representation in Hilbert space $Q^{(\infty)}H_f$. Namely, if $r=a_1a_2\ldots a_j$, where $a_i\in\mathfrak{Al}$ for all $i=1,2,\ldots j$, then
\begin{eqnarray*}
\widehat{\pi}(r)=\widehat{\pi}(a_1)\,\widehat{\pi}(a_2)\,\cdots\,\widehat{\pi}(a_j).
\end{eqnarray*}
In particular,
\begin{eqnarray*}
\widehat{\pi}(\epsilon{\{k\}})=Q^{(\infty)}\,\pi_f((1\;k))\,Q^{(\infty)}\;\widehat{\pi}(\epsilon_{\{1\}})\;Q^{(\infty)}\,
\pi_f((1\;k))\,Q^{(\infty)}.
\end{eqnarray*}
It is clear that each element $r\in R_\infty$  is written in the form $r=s\,\epsilon_{\mathbb{A}}$, where $s\in\mathfrak{S}_\infty$, $\mathbb{A}\subset \mathbb{N}$ and $\epsilon_{\mathbb{A}}=\prod\limits_{i\in\mathbb{A}}\epsilon_{\{i\}}$.
 Using definition of $\widehat{\pi}$, it is easy to verify that
  \begin{eqnarray*}
  Q^{(\infty)}\,\pi_f(s\,\epsilon_{\mathbb{A}})\, Q^{(\infty)}=(1-c(0))^{\#\mathbb{A}} \,\widehat{\pi}(s\,\epsilon_{\mathbb{A}}).
  \end{eqnarray*}
   It follows that $\widehat{\pi}(R_\infty)''=Q^{(\infty)}\,\pi_f(R_\infty)''\,Q^{(\infty)}$.  Since factor $ \pi_f(R_\infty)^{\prime\prime}$ has type ${\rm II}$ or  ${\rm I}$ and $Q^{(\infty)}\in \pi_f(R_\infty)^{\prime\prime} $, it follows from Proposition \ref{1 or 2} that there exists projection $F\in \widehat{\pi}(R_\infty)''$ with the following property
\begin{eqnarray}
\omega_f(A) =\left( A\xi_f,\xi_f\right)=\kappa\,{\rm Tr}(FA) ~\text{ for all }~ A\in \widehat\pi_f(R_\infty)^{\prime\prime},
\end{eqnarray}
 where ${\rm Tr}$ is normal semi-finite trace on $\pi_f(R_\infty)^{\prime\prime}$ and $\kappa\in\mathbb{R}_{>0}$. By Theorem \ref{supp_th} and Proposition \ref{quasi_unique}, we obtain that $F=Q^{(\infty)}$. Therefore, $\widehat{\pi}(R_\infty)''$ is a finite type factor and $\omega_f$ is the unique normalized normal trace on   $\widehat{\pi}(R_\infty)''$.
Then, by Proposition \ref{factor_condition}, $\chi(s)=\omega_f(\widehat{\pi}(s))=(\widehat{\pi}(s)\,\xi_f,\xi_f)$, where $s\in\mathfrak{S}_\infty$, is an indecomposable character on $\mathfrak{S}_\infty\subset R_\infty$. Therefore, there exist two collections $\alpha=(\alpha_1\geq\alpha_2\geq\ldots>0)$ and $\beta=(\beta_1\geq\beta\geq\ldots>0)$ of positive numbers (Thoma parameters \cite{Thoma}) with the properties:
\begin{eqnarray*}
\sum\alpha_i+\sum\beta_i\leq 1 ~\text{and }~\chi(c_n)=\sum\alpha_i^n+(-1)^{n-1}\beta_i^n,
\end{eqnarray*}
where $c_n=(k_1\;\,k_2\;\,\ldots\;\,k_n)$ is any cycle of length $n$. By Theorem \ref{finite_factor_repr}, two cases are possible:
\begin{itemize}
  \item[a)] $\widehat{\pi}(\epsilon_1)=0$ $\Rightarrow$ $(\widehat{\pi}(c_n\,\epsilon_1)\xi_f,\xi_f)=0$;
  \item[b)] there exists $\alpha_i$ such that $(\widehat{\pi}(c_n\,\epsilon_1)\xi_f,\xi_f)=\alpha_i^n$.
\end{itemize}
Hence, using the equality $\widehat{\pi}(\epsilon_{\{k\}})=(1-c(0))^{-1}\,Q^{(\infty)}\pi_f(\epsilon_{\{k\}})\, Q^{(\infty)}$ and \eqref{property_Q_infty}, we obtain in the case $c(0)<1$
\begin{eqnarray*}
\left(\pi_f(c_n\,\epsilon_{\{k_1\}})\xi_f,\xi_f \right)=(1-c(0))\left( \widehat{\pi}(c_n)\widehat{\pi}(\epsilon_{\{k_1\}})\xi_f,\xi_f\right)=(1-c(0))\alpha_i^n.
\end{eqnarray*}
\section{Realizations of the factor-representations, corresponding to the  $\mathfrak{S}_\infty$-invariant positive definite functions.}\label{section:realizations}
Our main goal in this section is the construction of the factor-representations for $R_\infty$ which are defined by
the indecomposable $\mathfrak{S}_\infty$-invariant positive definite function.
\subsection{Parameters of the factor-representations.}\label{parameters_of_repr}
Let $B(\mathbb{\mathbf{H}})$ denote
the set of all (bounded linear) operators acting on a complex separable Hilbert space $\mathbf{H}$.
Let ${\rm Tr}$ be the ordinary\footnote{ ${\rm Tr}(\mathfrak{p})$=1 for any  nonzero minimal projection  $\mathfrak{p}\in B(\mathbb{\mathbf{H}})$. } trace on $B(\mathbb{\mathbf{H}})$. Fix the self-adjoint operator $A\in B(\mathbf{H})$ and the {\it minimal} orthogonal projection $\mathbf{q}\in B(\mathbf{H})$. Let ${\rm Ker}\, A=\left\{ u\in B(\mathbf{H}):Au=0\right\}$, and let $\mathfrak{A}=\left\{ A,\mathbf{q}\right\}^{\prime\prime}$. Denote by $\widetilde{\mathbf{H}}$ the subspace generated by $\left\{ \mathfrak{A}v,   v\in ({\rm Ker} A)^\perp \right\}$, where $({\rm
      Ker} A)^\perp=\mathbf{H}\ominus {\rm Ker}\, A $. Denote by $\widetilde{\mathbf{P}}$ the orthogonal projection onto $\widetilde{\mathbf{H}}$.   It is clear that $\widetilde{\mathbf{P}}$ lies in $\mathfrak{A}$  and ${\rm dim}\left( \widetilde{\mathbf{H}}\ominus ({\rm Ker} A)^\perp\right)\leq 1$. Let $E(\Delta)$ be the spectral projection of operator $A$ corresponding to $\Delta\subset \mathbb{R}$. The closure of $ \mathfrak{A}E(\Delta) \mathbf{H}$ will be denoted by $\mathbf{H}_\Delta$.  Set $\mathbf{H}_{reg}=\mathbf{H}\ominus\widetilde{\mathbf{H}}$. By above, $\mathbf{H}_{reg}\subset  {\rm Ker}\, A$. We will assume that the following conditions hold:
\begin{itemize}
  \item[\rm(a)] ${\rm Tr} (|A|)\leq1$,
   \item[\rm(b)] if $\mathbf{q}\neq0$, then   ${\rm Tr}(\mathbf{q})$=1;
  \item[\rm(c)] if ${\rm Tr} (|A|)=1$, then $\widetilde{\mathbf{H}}$  coincides with $\mathbf{H}$;
    \item[\rm(d)] if ${\rm Tr} (|A|)<1$, then
      $\dim \mathbf{H}_{reg}=\infty$;
       \item[\rm(e)] $\mathbf{q}\cdot E([-1,0))=0$;
       \item[\rm(f)] $\mathbf{q}v=0$ for all $v\in \mathbf{H}_{reg}$.\label{property_f}
\end{itemize}
\subsection{The structure of the algebra $\mathfrak{A}$. }\label{structure_of_A} Since ${\rm Tr}\leq 1$, the spectrum of $A$ is at most countable. Let $S(A)=\left\{\lambda_1,\lambda_2,\ldots,  \right\}$ be the spectrum of $A$. Denote by $E(\lambda)$ a spectral projection corresponding to $\lambda\in S(A)$. Since $\mathbf{q}$ is a minimal projection, we have $\mathbf{q}\cdot E(\lambda)\cdot\mathbf{q}=c(\lambda)\cdot\mathbf{q}$, where $c(\lambda)\in \mathbb{R}_{\geq 0}$. Thus, if $c(\lambda)>0$ then the operator $\mathbf{q}_\lambda=c(\lambda)^{-1}\cdot E(\lambda)\cdot\mathbf{q}\cdot E(\lambda)$ is a minimal orthogonal projection. It is easy to check that the projections $\mathbf{q}$ and $\widetilde{E}(\lambda)=E(\lambda)-\mathbf{q}_\lambda$ are orthogonal. If $c(\lambda)=0$ then we set $\widetilde{E}(\lambda)=E(\lambda)$.

Set $\widehat{S}(A)=\left\{ \lambda\in S(A): c(\lambda)\neq 0  \right\}$. If $\lambda_k,  \lambda_l \in \widehat{S}(A)$ then the operator $\mathbf{e}_{\lambda_k\lambda_l}=\left( c(\lambda_k)\cdot c(\lambda_l) \right)^{-1/2} \cdot E\left( \lambda_k \right)\cdot \mathbf{q}\cdot E\left( \lambda_l \right)$ is a one-dimensional partial isometry and the following relation hold
\begin{eqnarray*}
\mathbf{e}_{\lambda_k\lambda_l}\cdot \mathbf{e}_{\lambda_i\lambda_j}=\delta_{\lambda_l\lambda_i}\cdot\mathbf{e}_{\lambda_k\lambda_j}.
\end{eqnarray*}
Observe that $\mathbf{e}_{\lambda_i\lambda_i}=\mathbf{q}_{\lambda_i}$
It is easily seen that the projection $\mathbf{q}$ belongs to the factor $\mathfrak{F}$ generated by $\left\{  \mathbf{e}_{\lambda_k\lambda_l} \right\}_{\lambda_k,\lambda_l\in \widehat{S}(A)}$.
Hence we obtain the following statement:
\begin{Prop}
Projection $\mathbf{f}=\sum\limits_{\lambda_i\in\widehat{S}(A)}\mathbf{e}_{\lambda_i\lambda_i}$ lies in the centre of  $\mathfrak{A}$. The algebra $\mathfrak{A}$ is decomposed into the direct sum
\begin{eqnarray*}
\mathfrak{A}=({\rm I}-\mathbf{f})\mathfrak{A}+\mathbf{f}\mathfrak{A},
\end{eqnarray*}
where $({\rm I}-\mathbf{f})\mathfrak{A}$ is the abelian algebra generated by projections $\widetilde{E}(\lambda_i), i\in S(A)$, and $\mathbf{f}\mathfrak{A}=\mathfrak{F}$.
\end{Prop}
\subsubsection{Hilbert space $\mathcal{H}_A^\mathbf{q}$.}\label{hp}
Let $\mathbb{S}=\left\{ 1,2,\ldots, {\rm dim}\,\mathbf{H} \right\}$. In particular, we assume that $\mathbb{S}=\mathbb N$, if ${\rm dim}\,\mathbf{H}=\infty$.
Fix the matrix unit $\left\{ e_{kl}\right\}_{k,l\in\mathbb{S}}\subset B(\mathbf{H})$
such that
\begin{eqnarray}
Ae_{ll}=e_{ll}A \text{ and } e_{ll}\widetilde{\mathbf{H}}\subset \widetilde{\mathbf{H}}\text{ for all } l.
\end{eqnarray}
Let $\mathbb{S}_{reg}=\left\{ n_1,n_2,\ldots\right\}=\left\{ l:e_{ll}\mathbf{H}\subset \mathbf{H}_{reg}\right\}$ \label{s_regular}(see ({\rm d})), where $n_k<n_{k+1}$.
 Define
a state $\psi_k$ on $B\left(\mathbf{H} \right)$ as
follows
\begin{eqnarray}\label{psik}
\psi_k\left(b \right)={\rm Tr}\left(b|A| \right)+\left(1-{\rm
Tr}\left(|A| \right)\right) {\rm Tr}\left(b e_{n_kn_k}\right),\;\;b\in B\left(\mathbf{H} \right).
\end{eqnarray}
Let $ _1\psi_k$ denote the product-state on $B\left(\mathbf{H}\right)^{\otimes k}$:
\begin{eqnarray}\label{product_psik}
 _1\psi_k\left(b_1\otimes b_2\otimes \ldots\otimes b_k
 \right)=\prod\limits_{j=1}^k\psi_j\left(b_j \right).
\end{eqnarray}
Now define inner product on $B \left( \mathbf{H}\right)^{\otimes k}$ by
\begin{eqnarray}\label{inner_product_psik}
\left( v,u\right)_k=\,_1\psi_k\left( u^*v \right).
\end{eqnarray}
Let $\mathcal{H}_k$ denote the Hilbert space obtained by completion of
$B \left( \mathbf{H}\right)^{\otimes k}$  in above inner product
norm. Now we consider the natural isometrical embedding
\begin{eqnarray}
v\ni\mathcal{H}_k\mapsto v\otimes {\rm I}\in\mathcal{H}_{k+1},
\end{eqnarray}
and define Hilbert space $\mathcal{H}^\mathbf{q}_A$ as the completion of
$\bigcup\limits_{k=1}^\infty \mathcal{H}_k$.

\paragraph{The action of $R_\infty$ on $\mathcal{H}^\mathbf{q}_A$. }\label{action}
  First, using the
embedding\\ $a\ni B\left(\mathbf{H} \right)^{\otimes
k}\mapsto a\otimes{\rm I}\in B\left(\mathbf{H}
\right)^{\otimes (k+1)}$, we identify $B\left(\mathbf{H}
\right)^{\otimes k}$ with subalgebra $B\left(\mathbf{H}
\right)^{\otimes
k}\otimes\mathbb{C}\subset B\left(\mathbf{H}
\right)^{\otimes (k+1)}$. Therefore, algebra
$B\left(\mathbf{H}\right)^{\otimes\infty}=\bigcup\limits_{n=1}^\infty
B\left(\mathbf{H} \right)^{\otimes n}$ is well defined.

Now we construct the explicit embedding of $\mathfrak{S}_\infty$ into the unitary subgroup of  $B\left(\mathbf{H}\right)^{\otimes\infty}$. For $a\in B\left(\mathbf{H}\right)$ put $a^{(k)}={\rm I}\otimes\cdots\otimes{\rm I}\otimes\underbrace{a}_{k}\otimes{\rm I}\otimes{\rm I}\cdots$. Let $E_{-}=E([-1,0))$ and let
\begin{eqnarray*}
U_{E_{-}}^{(k,\,k+1)}=({\rm I}-E_{-})^{(k)}({\rm I}-E_{-})^{(k+1)}+E_{-}^{(k)} ({\rm
I}-E_{-})^{(k+1)}\\
+ ({\rm I}-E_{-})^{(k)}E_{-}^{(k+1)}-E_{-}^{(k)}E_{-}^{(k+1)}.
\end{eqnarray*}
Define the unitary operator $T\left((k\;k+1)\right)\in B\left(\mathbf{H}
\right)^{\otimes\infty}$ as follows
\begin{eqnarray}\label{embeding_T}
T\left((k\;k+1)\right)=U_{E_{-}}^{(k,\, k+1)}\sum\limits_{ij}e_{ij}^{(k)}e_{ji}^{(k+1)}.
\end{eqnarray}
Put $T(\epsilon _{\{1\}})=\mathbf{q}^{(1)}$.

A easy verification of the standard relations between $\left\{ T((k\;k+1))\right\}_{k\in\mathbb{N}}$ and
$\mathbf{q}^{(1)}$ shows that
$T$ extends by multiplicativity to the $\star$-homomorphism   of $R_\infty$ to $B\left(\mathbf{H}
\right)^{\otimes\infty}$.

Left multiplication in $B\left(\mathbf{H}\right)^{\otimes\infty}$ defines  $\star$-representation $\mathfrak{L}_A$ of   $B\left(\mathbf{H}\right)^{\otimes\infty}$ by bounded operators on  $\mathcal{H}^\mathbf{q}_A$. Put $\Pi_A(r)=\mathfrak{L}_A\left( T(r)\right)$, $r\in R_\infty$. Denote by $\pi_A^{(0)}$ and $\mathfrak{L}_A^{(0)}$\label{notation for the L^0} the restrictions of $\Pi_A$ and $\mathfrak{L}_A$, respectively,  to $\left[ \Pi_A\left( R_\infty\right)\xi _0\right]$, where $\xi _0$ is the vector from $\mathcal{H}^\mathbf{q}_A$ corresponding to the unit element of $B\left(\mathbf{H}\right)^{\otimes\infty}$.
\begin{Rem}
If $T\left( \epsilon_{\{1\}} \right)=0$ and $T(s)$ $(s\in\mathfrak{S}_\infty)$ is defined by (\ref{embeding_T}), then $\left\{  \pi_A^{(0)}\left( R_\infty \right) \right\}^{\prime\prime}=\left\{  \pi_A^{(0)}\left( \mathfrak{S}_\infty \right) \right\}^{\prime\prime}$ and the corresponding representation $\pi_A^{(0)}$ is ${\rm II}_1$-factor-representation of $R_\infty$.

\end{Rem}

\begin{Rem}
If $A=\mathbf{p}$, where $\mathbf{p}$ is one-dimensional projection, then $\left(\Pi_\mathbf{p}(s)\xi _0,\xi _0 \right)_{\mathcal{H}^\mathbf{q}_A}=1$ for all  $s\in\mathfrak{S}_\infty$; i. e. $\pi_\mathbf{p}^{(0)}$ is $\mathfrak{S}_\infty$-spherical representation of $R_\infty$ (see p. \ref{sph_repr}).
\end{Rem}
\subsubsection{The invariance  and multiplicativity  of  $\varphi_A$.}
Let $\varphi_A(r)$ $= \left(\pi_A^{(0)}(r) \xi _0,\xi _0\right) $, $r\in R_\infty$.
\begin{Lm}\label{invariant_state}
The positive definite function $\varphi_A$ is $\mathfrak{S}_\infty$-invariant. Namely, the  relation $\varphi_A(sr)=\varphi_A(rs)$ holds for all $s\in\mathfrak{S}_\infty$ and $r\in R_\infty$.
\end{Lm}
\begin{proof}
Recall that $\mathbb{S}_{reg}=\left\{ n_1,n_2,\ldots\right\}=\left\{ l:e_{ll}\mathbf{H}\subset \mathbf{H}_{reg}\right\}$ (see paragraph \ref{hp}).
For any $s\in\mathfrak{S}_\infty$ set $U_s=\sum\limits_k e_{n_{s(k)}n_k}+\sum\limits_{i\in\mathbb{S}\setminus\mathbb{S}_{reg}}e_{ii}$.
Consider the sequence of the operators $U_s^{\otimes n}=\prod\limits_{j=1}^n U_s^{(j)}$ $\in B(\mathbf{H})^{\otimes\infty}$, where $U_s^{(j)}$ is defined in paragraph \ref{action}. It follows from definition of Hilbert space $\mathcal{H}^\mathbf{q}_A$ that there exists $\lim\limits_{n\to\infty} \mathfrak{L}_A\left( U_s^{\otimes n}\right)$  in the strong operator topology. Set $\mathfrak{L}_A\left(U_s^{\otimes \infty}\right)=\lim\limits_{n\to\infty} \mathfrak{L}_A\left( U_s^{\otimes n}\right)$. By definition of the representation $\Pi_A$ (paragraph \ref{action}), we have
\begin{eqnarray}\label{commutativity}
\Pi_A(t)\cdot \mathfrak{L}_A\left(U_s^{\otimes \infty}\right)=\mathfrak{L}_A\left(U_s^{\otimes \infty}\right)\cdot \Pi_A(t) \text{ for all } s,t\in\mathfrak{S}_\infty.
\end{eqnarray}
It follows from property ({\rm f})(paragraph \ref{parameters_of_repr}) that
\begin{eqnarray}\label{commutativity_idempotent}
\Pi_A\left( \epsilon_{\{k\}} \right)\cdot \mathfrak{L}_A\left(U_s^{\otimes \infty}\right)=\mathfrak{L}_A\left(U_s^{\otimes \infty}\right)\cdot \Pi_A\left( \epsilon_{\{k\}} \right)
\text{ for all } s\in\mathfrak{S}_\infty, k\in\mathbb{N}.
\end{eqnarray}
Denote by $A_k$ the operator $A+\left( 1-{\rm Tr}\,\left( A\right)\right)e_{n_kn_k}$ $\in B(\mathbf{H})$. The simple calculation shows that for any $s\in\mathfrak{S}_\infty$ and $n>\max\left\{l:l\in{\rm supp}\;s \right\}$ next equality holds
\begin{eqnarray}\label{condition_of_invariance}
T(s)\cdot U_s^{\otimes n}\cdot A_k^{(k)}=A_{s(k)}^{(s(k))} \cdot T(s)\cdot U_s^{\otimes n}.
\end{eqnarray}
Hence, using  (\ref{psik}), (\ref{product_psik}), (\ref{inner_product_psik}), (\ref{commutativity}) and (\ref{commutativity_idempotent}), we obtain the statement of  lemma \ref{invariant_state}.
\end{proof}
The next statement is obvious.
\begin{Prop}\label{multiplicativity}
Let $r=q_1\cdot q_2\cdots q_k$ be the decomposition of $r\in R_\infty$ in the product of the independent quasicycles (see page \pageref{quasicycle}). Then $\varphi_A(r)=\prod\limits_{j=1}^k\varphi_A(q_j)$. It follows from Proposition \ref{factor_condition} that $\pi_A^{(0)}$ is a factor-representation.
\end{Prop}
\subsubsection{The formula for $\varphi_A$.}
It follows from (\ref{decomposition_into_product}) and Proposition \ref{multiplicativity}  that it is sufficient to obtain the expression for  $\varphi_A(r)$ in the case, when either $r=c$ (a cycle) or $r=c\cdot\epsilon _{\{a\}}$ (a quasi-cycle), where $c=\left( n_1\;n_2\;\ldots,n_k\right)$, $a\in {\rm supp}\,c$. To simplify notation we take $c=(1\;2\;\ldots\;k)$ and $a=k$. It is clear that
\begin{eqnarray}
\begin{split}
r=c\cdot\epsilon_{\{a\}}=(1\;k)\cdot(1\;k-1)\cdots(1\;2)\cdot\epsilon_{\{a\}}.
\end{split}
\end{eqnarray}
Using (\ref{psik}), (\ref{product_psik}), (\ref{inner_product_psik}) and the definition of representation $\pi_A^{(0)}$ (see paragraph \ref{action}), we obtain
\begin{eqnarray*}
\varphi _A(r)={\rm Tr}\left(  A^{k-1}\mathbf{q}|A|\right).
\end{eqnarray*}
Similarly, for $r=c$ we obtain $\varphi _A(r)=\varphi_A(c)={\rm Tr}\left( A^{k-1}|A|\right)$.
\subsubsection{The conditions of the cyclicity of vector $\xi_0$ for  $\pi_A^{(0)}\left(R_\infty \right)^\prime$.}
By definition, vector $\xi_0$ is cyclic for the representation $\pi_A^{(0)}$. Here we find the conditions of the cyclicity of $\xi_0$ for  the commutant  of $\pi_A^{(0)}$. In this case vector $\xi_0$ is {\it called cyclic and separating} or {\it bicyclic}.
\begin{Th}\label{condition_of_bicyclic}
  Vector $\xi_0$ is cyclic for  $\pi_A^{(0)}\left(R_\infty \right)^\prime$ if and only if the equality $({\rm
      Ker} A)^\perp =\widetilde{\mathbf{H}}$ holds.
\end{Th}
\begin{proof}
  Let  $E(0)$ be the orthogonal projection on ${\rm Ker}\, A$, and  let $\widetilde{\mathbf{P}}$ be the orthogonal projection on $\widetilde{\mathbf{H}}$.

  Suppose that $\xi_0$ is cyclic vector for $\pi_A^{(0)}\left(R_\infty \right)^\prime$; i.e.
  \begin{eqnarray}\label{bicyclicity}
  \left[ \pi_A^{(0)}\left(R_\infty \right)^\prime\xi_0\right]=\left[ \pi_A^{(0)}\left(R_\infty \right)\xi_0\right].
  \end{eqnarray}
  By Lemma \ref{Okounkov_operator}, the limit $\mathcal{O}_k=\lim\limits_{n\to\infty}\pi_A^{(0)}\left((k\;n) \right)$ exists in the weak operator topology. An easy computation shows that
  \begin{eqnarray}\label{explicity}
  \mathcal{O}_k=\mathfrak{L}_A\left(A^{(k)} \right), \text{ where }  A^{(k)}= {\rm I}\otimes\cdots\otimes{\rm I}\otimes\underbrace{A}_{k}\otimes{\rm I}\otimes{\rm I}\cdots.
  \end{eqnarray}
  By definition, $({\rm
      Ker} A)^\perp \subset\widetilde{\mathbf{H}}$ and $({\rm I}-\widetilde{\mathbf{P}})\mathbf{H} =\mathbf{H}_{reg}\subset{\rm
      Ker} A$, where $\widetilde{\mathbf{P}}$ is an orthogonal projection on $\widetilde{\mathbf{H}}$.

  We now suppose that $({\rm
      Ker} A)^\perp \neq\widetilde{\mathbf{H}}$; i. e. $\widetilde{\mathbf{P}}-(I-E(0))=\Xi\neq0$. By the above, operator $\Xi$ is an orthogonal projection and $\Xi\mathbf{H}\subset( {\rm
      Ker} A)\ominus\mathbf{H}_{reg} $. Therefore,
      \begin{eqnarray}\label{not_faithful}
\left(  \mathfrak{L}_A\left(\Xi^{(k)}\right)\xi_0,\xi_0 \right)={\rm Tr}\left(|A|\Xi \right)=0.
    \end{eqnarray}
Using (\ref{explicity}) and definition of $\Pi_A$, we can assert that
\begin{eqnarray*}
   \mathfrak{L}_A\left(B^{(k)} \right)\in \left\{\Pi_A(R_\infty) \right\}^{\prime\prime}\;\text{ for all } \; B\in\mathfrak{A}.
\end{eqnarray*}
It follows that $ \mathfrak{L}_A\left(\Xi^{(k)}\right)\in\left\{\Pi_A(R_\infty) \right\}^{\prime\prime}$.
Hence, applying  (\ref{bicyclicity}) and (\ref{not_faithful}), we have $\widetilde{\mathbf{P}}-(I-E(0))=0$; i. e.  $({\rm
      Ker} A)^\perp =\widetilde{\mathbf{H}}$.

      Now we will prove that  (\ref{bicyclicity}) follows from the last equality.
      It follows from  (\ref{condition_of_invariance}), that  multiplication by $T(s)\cdot U_s^{\otimes n}$ on the right define under the condition $n>{\rm max}\;\left\{ l:l\in{\rm supp}\;s \right\}$ unitary operator $\mathfrak{R}_A\left(T(s)\cdot U_s^{\otimes n} \right)$:
      \begin{eqnarray*}
      \mathcal{H}^\mathbf{q}_A\ni x\stackrel{\mathfrak{R}_A\left(T(s)\cdot U_s^{\otimes n} \right)}{\mapsto} x\cdot \left(T(s)\cdot U_s^{\otimes n} \right).
      \end{eqnarray*}
      Using definition of $ \mathcal{H}^\mathbf{q}_A$, it is easy to check that
      \begin{eqnarray*}
     x\cdot T(s)\cdot U_s^{\otimes n}=x\cdot T(s)\cdot U_s^{\otimes N}, , \text{ where } x\in\mathcal{H}^\mathbf{q}_A,\\
     \text{ for all } n,N>{\rm max}\;\left\{ l:l\in{\rm supp}\;s \right\}.
      \end{eqnarray*}
      Therefore, we can  define unitary operator
      \begin{eqnarray*}
      \mathfrak{R}_A\left(T(s)\cdot U_s^{\otimes \infty} \right)=\lim\limits_{n\to\infty}\mathfrak{R}_A\left(T(s)\cdot U_s^{\otimes n} \right).
      \end{eqnarray*}
      Since $\Pi_A(r)=\mathfrak{L}_A\left( T(r)\right)$, $r\in R_\infty$, then $ \mathfrak{R}_A\left(T(s)\cdot U_s^{\otimes \infty} \right)\in\left\{\Pi_A(R_\infty) \right\}^\prime$.
      It follows immediately that
\begin{eqnarray}\label{cycle_sym_gr}
\Pi_A(s)\xi_0= \mathfrak{L}_A\left( U_s^{\otimes \infty} \right)\cdot\mathfrak{R}_A\left(T(s)\cdot U_s^{\otimes \infty} \right)\xi_0\text{ for all } s\in\mathfrak{S}_\infty.
\end{eqnarray}
We recall that $\mathfrak{L}_A\left( U_s^{\otimes \infty} \right)\in \left\{\Pi_A(R_\infty) \right\}^\prime$ (see (\ref{commutativity}) and (\ref{commutativity_idempotent})). As we can see above, so far we did not use the equality $({\rm Ker} A)^\perp =\widetilde{\mathbf{H}}$, which means that
\begin{eqnarray}\label{regularity}
\mathbf{q}\leq {\rm I}-E(0).
\end{eqnarray}
Now we will prove that
\begin{eqnarray}\label{idempotent_cycl}
\Pi_A\left(\epsilon_{\left\{k \right\}} \right)\xi_0\in\left[ \left\{\Pi_A(R_\infty) \right\}^\prime\xi_0\right].
\end{eqnarray}
By the definition of representation $\Pi_A$ (see p. \ref{action}), we have
\begin{eqnarray*}
\Pi_A\left(\epsilon_{\left\{k \right\}} \right)\xi_0=\mathbf{q}^{(k)} \text{ for all } k.
\end{eqnarray*}
By (\ref{regularity}), for any $\epsilon>0$ there exists spectral projection $E\left(\Delta_\epsilon \right)$ of operator $A$ such that
\begin{eqnarray*}
\begin{split}
\left\|E\left(\Delta_\epsilon \right)^{(k)}\cdot\mathbf{q}^{(k)} \cdot E\left(\Delta_\epsilon \right)^{(k)}-\mathbf{q}^{(k)} \right\|_{\mathcal{H}^\mathbf{q}_A}<\epsilon,\\ \text{ where }
\Delta_\epsilon=\left\{\lambda\in[-1,1]:|\lambda|>\sigma(\epsilon)>0 \right\}.
\end{split}
\end{eqnarray*}
To prove (\ref{idempotent_cycl}), it is sufficient to show that operator of the right multiplication by $\mathbf{q}^{(k)}_{\sigma(\epsilon)}=E\left(\Delta_\epsilon \right)^{(k)}\cdot\mathbf{q}^{(k)} \cdot E\left(\Delta_\epsilon \right)^{(k)}$ is bounded in $\mathcal{H}^\mathbf{q}_A$.

First we denote by $\,^{\scriptscriptstyle(-1)}\!\!A_\vartheta$ operator $\int\limits_{\left\{\lambda:|\lambda|>\vartheta \right\}} |\lambda|^{-1}\,{\rm d}\, E(\lambda)$, where $\vartheta>0$ and $\int\limits_{-1}^1\lambda\,{\rm d}\, E(\lambda)$ is the spectral decomposition of $A$.
According to the definition of  $\mathcal{H}^\mathbf{q}_A$ (see p. \ref{hp}), we have
\begin{eqnarray*}
\left(x,x \right)_{ \mathcal{H}^\mathbf{q}_A}\geq \left(x\cdot E\left(\Delta_\epsilon \right)^{(k)},x\cdot E\left(\Delta_\epsilon \right)^{(k)} \right)_{ \mathcal{H}^\mathbf{q}_A}\geq \sigma(\epsilon)\cdot\left(x\cdot\left(\,^{\scriptscriptstyle(-1)}\!\!A_{\sigma(\epsilon)}\right)^{(k)},x \right)_{ \mathcal{H}^\mathbf{q}_A}.
\end{eqnarray*}
But on the other side
\begin{eqnarray*}
\left(x\cdot \mathbf{q}^{(k)}_{\sigma(\epsilon)},x \cdot \mathbf{q}^{(k)}_{\sigma(\epsilon)}\right)_{ \mathcal{H}^\mathbf{q}_A}\leq\left(x\cdot\left(\,^{\scriptscriptstyle(-1)}\!\!A_{\sigma(\epsilon)}\right)^{(k)},x \right)_{ \mathcal{H}^\mathbf{q}_A}.
\end{eqnarray*}
Therefore, $\left(x,x \right)_{ \mathcal{H}^\mathbf{q}_A}\geq\sigma(\epsilon)\cdot\left(x\cdot \mathbf{q}^{(k)}_{\sigma(\epsilon)},x \cdot \mathbf{q}^{(k)}_{\sigma(\epsilon)}\right)_{ \mathcal{H}^\mathbf{q}_A}$. Thus
\begin{eqnarray*}
\left\|\mathfrak{R}_A\left(\mathbf{q}^{(k)}_{\sigma(\epsilon)} \right) \right\|\leq{\sigma(\epsilon)}^{-1/2}.
\end{eqnarray*}
Since operator $\mathfrak{R}_A\left(\mathbf{q}^{(k)}_{\sigma(\epsilon)} \right)$ belongs to algebra $\left\{\Pi_A(R_\infty) \right\}^\prime$, this gives (\ref{idempotent_cycl}).
Now (\ref{bicyclicity}) follows from (\ref{cycle_sym_gr}) and (\ref{idempotent_cycl}).
\end{proof}
Under the conditions stated in p. \ref{parameters_of_repr}, the inequality  ${\rm dim}\,(\widetilde{\mathbf{P}}\ominus ({\rm Ker}\,A)^\perp)\leq 1$ holds. We considered in Theorem \ref{condition_of_bicyclic} the case ${\rm dim}\,(\widetilde{\mathbf{P}}\ominus ({\rm Ker}\,A)^\perp)=0$. If ${\rm dim}\,(\widetilde{\mathbf{P}}\ominus ({\rm Ker}\,A)^\perp)=1$, then $\left[ \pi_A^{(0)}\left(R_\infty \right)^\prime\xi_0\right]\subsetneq\left[ \pi_A^{(0)}\left(R_\infty \right)\xi_0\right]$. Denote by $\mathfrak{Q}$ the orthogonal projection of the space $\left[ \pi_A^{(0)}\left(R_\infty \right)\xi_0\right]$ onto $\left[ \pi_A^{(0)}\left(R_\infty \right)^\prime\xi_0\right]$. As a consequence of lemma \ref{invariant_state}, we have $\pi_A^{(0)}(s)\,\left[ \pi_A^{(0)}\left(R_\infty \right)^\prime\xi_0\right]=\left[ \pi_A^{(0)}\left(R_\infty \right)^\prime\xi_0\right]$ for all $s\in \mathfrak{S}_\infty$. Therefore,
\begin{eqnarray*}
\mathfrak{Q}\in \left(  \pi_A^{(0)}(\mathfrak{S}_\infty) \right)^{\prime}\cap\left( \pi_A^{(0)}\left(R_\infty \right) \right)^{\prime\prime}.
\end{eqnarray*}

Let  $\mathbf{Q}$ be the orthogonal projection of the space $\mathbf{H}$ onto $\left( {\rm I}-E(0) \right)\mathbf{H}\oplus \mathbf{H}_{reg}$ (see Section \ref{parameters_of_repr}): i.e. $\mathbf{Q}=2\,{\rm I}-E(0)- \widetilde{\mathbf{P}}$.
From now on, $\mathfrak{Q}^{(n)}$ denotes the projection $\mathfrak{L}_A\left( \underbrace{\mathbf{Q}\otimes\mathbf{Q}\otimes\cdots\mathbf{Q}}_{n}\otimes {\rm I}\otimes{\rm I}\otimes\cdots \right)$. Since $\mathfrak{Q}^{(n)}\geq\mathfrak{Q}^{(n+1)}$ and $\left( \mathfrak{Q}^{(n)}\xi_0,\xi_0 \right)_{\mathcal{H}_A^\mathbf{q}}=1$ for all $k$, there exists $\lim\limits_{n\to\infty}\mathfrak{Q}^{(n)}$   in strong operator topology.  We  denote this limit by $\mathfrak{Q}^{(\infty)}$. By the above,
\begin{eqnarray}\label{equalities}
\mathfrak{Q}^{(\infty)}\xi_0=\xi_0\;\text{ and } \pi_A^{(0)}(s)\cdot\mathfrak{Q}^{(\infty)}=\mathfrak{Q}^{(\infty)}\cdot\pi_A^{(0)}(s)\;\text{ for all } s\in\mathfrak{S}_\infty.
\end{eqnarray}

\begin{Prop}\label{Q^infty=Q}
$\mathfrak{Q}=\mathfrak{Q}^{(\infty)}$.
\end{Prop}
\begin{proof}

Since $\pi_A^{(0)}\left( \epsilon_{k} \right)=\mathfrak{L}_A^{(0)}\left( \mathbf{q}^{(k)} \right)$ (see Paragraph \ref{action} for definitions), using (\ref{explicity}), we obtain that $\mathfrak{Q}^{(\infty)}$ $\in \left\{\pi_A^{(0)}(R_\infty)\right\}^{\prime\prime}$.
Hence, applying  (\ref{equalities}), we have 
\begin{eqnarray}\label{geq}
\mathfrak{Q}^{(\infty)}\geq\mathfrak{Q}.
\end{eqnarray}
Let us prove that
\begin{eqnarray}\label{QqQ in commutant}
(\mathbf{Q}\mathbf{q}\mathbf{Q})^{(k)}\in \left[ \pi_A^{(0)}\left(R_\infty \right)^\prime\xi_0\right].
\end{eqnarray}
Applying property {\rm(f)} on page \pageref{property_f}, we have
\begin{eqnarray}
(\mathbf{Q}\mathbf{q}\mathbf{Q})^{(k)}=\left(\left( I-E(0)\right)\mathbf{q} \left( I-E(0)\right)\right)^{(k)}.
\end{eqnarray}
Thus, for any $\epsilon>0$ there exists spectral projection $E\left(\Delta_\epsilon \right)$ of operator $A$, where $\Delta_\epsilon=\left\{\lambda\in[-1,1]:|\lambda|>\sigma(\epsilon)>0 \right\}$, such that
\begin{eqnarray}\label{approximation of QqQ}
\left\|E\left(\Delta_\epsilon \right)^{(k)}\cdot\mathbf{q}^{(k)} \cdot E\left(\Delta_\epsilon \right)^{(k)}-\left(\mathbf{Q}\mathbf{q} \mathbf{Q}\right)^{(k)} \right\|_{\mathcal{H}^\mathbf{q}_A}<\epsilon.
\end{eqnarray}
Set $\mathbf{q}_{\sigma(\epsilon)}=E\left(\Delta_\epsilon \right)^{(k)}\cdot\mathbf{q}^{(k)} \cdot E\left(\Delta_\epsilon \right)^{(k)}$. Let us introduce right multiplication operator $\mathfrak{R}_A(\mathbf{q}_{\sigma(\epsilon)})$ by 
\begin{eqnarray}
\mathcal{H}^\mathbf{q}_A\ni x\stackrel{\mathfrak{R}_A(\mathbf{q}_{\sigma(\epsilon)})}{\mapsto}x\,\mathbf{q}_{\sigma(\epsilon)}\in\mathcal{H}^\mathbf{q}_A.
\end{eqnarray}
For the same reason as in the proof of Theorem \ref{condition_of_bicyclic}
 we have the following estimate 
\begin{eqnarray}
\left\| \mathfrak{R}_A(\mathbf{q}_{\sigma(\epsilon)})\right\|\leq \sigma(\epsilon)^{-1/2}.
\end{eqnarray}
Therefore,  operator $\mathfrak{R}_A(\mathbf{q}_{\sigma(\epsilon)})$ 
  belongs to algebra $\left\{\Pi_A(R_\infty) \right\}^\prime$ and $\mathfrak{R}_A(\mathbf{q}_{\sigma(\epsilon)})\xi_0=\mathbf{q}_{\sigma(\epsilon)}\in\left[ \pi_A^{(0)}\left(R_\infty \right)^\prime\xi_0\right]$. Since $\epsilon$ is arbitrary, \eqref{approximation of QqQ} shows that \eqref{QqQ in commutant} is proven.
  
  Now we will prove by induction that
  \begin{eqnarray}\label{2 step of ind}
  \prod\limits_{m\in\mathbb{S}}\left( \mathbf{Q}\mathbf{q}\mathbf{Q}\right)^{(m)}\in\left[ \pi_A^{(0)}\left(R_\infty \right)^\prime\xi_0\right]~\text{for any subset }~ \mathbb{S} ~\text{such that}~ \#\mathbb{S}=n.
  \end{eqnarray}
  If $l\in\mathbb{S}$ then there exists a sequence $F_j'\in \pi_A^{(0)}\left(R_\infty \right)^\prime$ such that
  \begin{eqnarray}\label{limit of commutant}
  \lim\limits_{j\to\infty}\left\|\prod\limits_{m\in\mathbb{S}\setminus\{k\}}\left( \mathbf{Q}\mathbf{q}\mathbf{Q}\right)^{(m)}- F_j'\xi_0\right\|_{\mathcal{H}^\mathbf{q}_A}=0.
  \end{eqnarray} 
  Since $\mathfrak{L}_A^{(0)}\left(\left(  \mathbf{Q}\mathbf{q}\mathbf{Q}\right)^{(k)} \right)$ belongs to $\pi_A^{(0)}\left(R_\infty \right)''$,  \eqref{limit of commutant} shows that
  \begin{eqnarray*}
  \begin{split}
  \prod\limits_{m\in\mathbb{S}}\left(  \mathbf{Q}\mathbf{q}\mathbf{Q}\right)^{(m)}=\lim\limits_{j\to\infty}\mathfrak{L}_A^{(0)}\left(\left(  \mathbf{Q}\mathbf{q}\mathbf{Q}\right)^{(k)} \right)F_j'\xi_0\\=\lim\limits_{j\to\infty} F_j'\,\left( \mathbf{Q}\mathbf{q}\mathbf{Q}\right)^{(k)}\stackrel{\eqref{QqQ in commutant}}{\in}\left[ \pi_A^{(0)}\left(R_\infty \right)^\prime\xi_0\right],
  \end{split}
  \end{eqnarray*}
  Thus, \eqref{2 step of ind} is proven.
  
  Now we conclude from \eqref{equalities} and \eqref{2 step of ind} that
  \begin{eqnarray*}
    \pi_A^{(0)}(s)\,\prod\limits_{m\in\mathbb{S}}\left(  \mathbf{Q}\mathbf{q}\mathbf{Q}\right)^{(m)}\in\pi^{(0)}_A(s)\,\left[ \pi_A^{(0)}\left(R_\infty \right)^\prime\xi_0\right]\\
    =\left[ \pi_A^{(0)}\left(R_\infty \right)^\prime\,\pi^{(0)}_A(s)\xi_0\right]\stackrel{\text{Lemma \ref{invariant_state}}}{=}\left[ \pi_A^{(0)}\left(R_\infty \right)^\prime\,\xi_0\right]~\text{ for all }~s\in\mathfrak{S}_\infty.
  \end{eqnarray*}
  Hence, applying \eqref{equalities} again, we have $\mathfrak{Q}^{(\infty)}\pi^{(0)}_A(R_\infty)\xi_0\subset \left[ \pi_A^{(0)}\left(R_\infty \right)^\prime\,\xi_0\right]$.
Now Proposition \ref{Q^infty=Q} follows from \eqref{geq}. 
\end{proof}
\subsubsection{Modular automorphism group.}\label{mag} Define state $\widetilde{\varphi}$ on $\left\{\pi_A^{(0)}(R_\infty) \right\}^{\prime\prime}$ by the formula $\widetilde{\varphi}\left(B \right)=\left(B\xi_0,\xi_0 \right)$. Assume that vector $\xi_0$  is bicyclic. Using Theorem \ref{condition_of_bicyclic}, we obtain that $\mathbf{q}\cdot E(0)=0$ (see Section \ref{parameters_of_repr}). Under this condition $\widetilde{\varphi}$ is faithful normal state on $\left\{\pi_A^{(0)}(R_\infty) \right\}^{\prime\prime}$.  It is clear that $\widetilde{\varphi}\left(\pi_A^{(0)}(r) \right)$ $=\varphi_A(r)$, $r\in R_\infty$. Let $\Delta$ be the  modular operator and let $\sigma_t$, $t\in \mathbb{R}$ be the modular automorphism group, corresponding to $\widetilde{\varphi}$ (see \cite{TAKES}). Since $\varphi_A$ is $\mathfrak{S}_\infty$-invariant, we have
\begin{eqnarray*}
\widetilde{\varphi}\left(BD \right)=\widetilde{\varphi}\left(DB \right) \text{ for all } B\in \left\{\pi_A^{(0)}(R_\infty) \right\}^{\prime\prime} \text{ and } D\in \left\{\pi_A^{(0)}(\mathfrak{S}_\infty) \right\}^{\prime\prime}.
\end{eqnarray*}
It follows from this that
\begin{eqnarray}\label{mod_oper_on_sym}
\Delta\left(D\xi_0 \right)=D\xi_0 \text{ and } \sigma_t(D)= D  \text{ for all } D\in \left\{\pi_A^{(0)}(\mathfrak{S}_\infty) \right\}^{\prime\prime}  \text{ and } t\in\mathbb{R}.
\end{eqnarray}
Now we will calculate the action of $\Delta$ on the elements $\mathbf{e}_{\lambda_k\lambda_l}^{(j)}\in\mathcal{H}_A^\mathbf{q}$, where $\mathbf{e}_{\lambda_k\lambda_l}$ are defined in Section \ref{structure_of_A} (page \pageref{structure_of_A}).

For convenience of the reader, we recall the definition of the modular operator $\Delta$. First denote by $S$ the closure of the antilinear map
\begin{eqnarray}
\mathcal{H}_A^\mathbf{q}\ni X\xi_0\stackrel{S}{\mapsto}X^*\xi_0, \text{ where } X\in \left\{\pi_A^{(0)}(R_\infty) \right\}^{\prime\prime}.
\end{eqnarray}
The adjoint operator $F$ of $S$ is defined by equality
\begin{eqnarray*}
\left(X^*\xi_0,Y^\prime\xi_0 \right)=\left(F\left(Y^\prime\xi_0 \right),X\xi_0 \right), \text{ where } X\in \left\{\pi_A^{(0)}(R_\infty) \right\}^{\prime\prime}, Y^\prime\in\left\{\pi_A^{(0)}(R_\infty) \right\}^{\prime}.
\end{eqnarray*}
Therefore, $F\left(Y^\prime\xi_0\right)=\left(Y^\prime \right)^*\xi_0$. Finally, $\Delta=FS$.

Now we notice that, by (\ref{explicity}) and Lemma \ref{Okounkov_operator}, the operators $\mathfrak{L}_A\left(A^{(j)} \right)$ and $\mathfrak{L}_A\left(\mathbf{q}^{(j)} \right)$ lie in $\left\{\pi_A^{(0)}(R_\infty) \right\}^{\prime\prime}$. It follows that  $\mathfrak{L}_A\left(\mathbf{e}_{\lambda_k\lambda_l}^{(j)} \right)\in \left\{\pi_A^{(0)}(R_\infty) \right\}^{\prime\prime}$, where $\mathbf{e}_{\lambda_k\lambda_l}^{(j)}$  are defined in p. \ref{structure_of_A}. We thus get
\begin{eqnarray}\label{involution}
S\mathbf{e}_{\lambda_k\lambda_l}^{(j)}=\mathbf{e}_{\lambda_l\lambda_k}^{(j)}, \text{ where } \lambda_k, \lambda_l>0\;\;\; (\text{see p. \ref{parameters_of_repr}}).
\end{eqnarray}
By the above, $F\mathbf{e}_{\lambda_l\lambda_k}^{(j)}= \left(\mathfrak{R}_A \left(\mathbf{e}_{\lambda_l\lambda_k}^{(j)} \right) \right)^*\xi_0$, where $\mathfrak{R}_A \left(\mathbf{e}_{\lambda_l\lambda_k}^{(j)} \right)$ is an operator of the right multiplication by $\mathbf{e}_{\lambda_l\lambda_k}^{(j)}$. To calculate $\left(\mathfrak{R}_A \left(\mathbf{e}_{\lambda_l\lambda_k}^{(j)} \right) \right)^*$, we take the elements $V,W\in \mathcal{H}_A^\mathbf{q}$ of the view $V=v_1\otimes v_2\otimes\cdots\otimes v_n\otimes{\rm I}\otimes{\rm I}\otimes\cdots$, $W=w_1\otimes w_2\otimes\cdots\otimes w_n\otimes{\rm I}\otimes{\rm I}\otimes\cdots$. Assuming $j\leq n$, we have
\begin{eqnarray*}
&\left(\mathfrak{R}_A \left(\mathbf{e}_{\lambda_l\lambda_k}^{(j)} \right) V,W \right)_{\mathcal{H}_A^\mathbf{q}}={\rm Tr}\left(|A|w_j^*v_j\cdot\mathbf{e}_{\lambda_l\lambda_k}^{(j)} \right)\cdot \prod\limits_{\{i=1\}\&\{i\neq j\}}^n\psi_i(w_i^*v_i)\\
&=\lambda_l^{-1}\lambda_k{\rm Tr}\left(|A|\cdot \left(w_j\cdot \mathbf{e}_{\lambda_k\lambda_l}^{(j)}\right)^*v_j \right)\cdot \prod\limits_{\{i=1\}\&\{i\neq j\}}^n\psi_i(w_i^*v_i)\\
&=\lambda_l^{-1}\lambda_k\left( V,\mathfrak{R}_A \left(\mathbf{e}_{\lambda_k\lambda_l}^{(j)} \right)W \right)_{\mathcal{H}_A^\mathbf{q}}.
\end{eqnarray*}
Hence, using (\ref{involution}), we obtain
\begin{eqnarray*}
\Delta \mathbf{e}_{\lambda_k\lambda_l}^{(j)}=FS\mathbf{e}_{\lambda_k\lambda_l}^{(j)}=\lambda_l^{-1}\lambda_k\mathbf{e}_{\lambda_k\lambda_l}^{(j)}.
\end{eqnarray*}
Therefore, using (\ref{mod_oper_on_sym}), we have
\begin{eqnarray}\label{action_mod_aut_gr_on_generators}
\begin{split}
&\sigma_t\left(\mathfrak{L}_A (\mathbf{e}_{\lambda_k\lambda_l}^{(j)})\right)=\Delta^{it}\cdot\mathfrak{L}_A (\mathbf{e}_{\lambda_k\lambda_l}^{(j)})\cdot \Delta^{-it}=\lambda_l^{-it}\lambda_k^{it}\cdot \mathfrak{L}_A (\mathbf{e}_{\lambda_k\lambda_l}^{(j)}),\\
&\sigma_t(X)=X, \text{ when } X\in \left\{\pi_A^{(0)}(\mathfrak{S}_\infty) \right\}^{\prime\prime}.
\end{split}
\end{eqnarray}
The action of $\sigma_t$ on $\left\{\pi_A^{(0)}(R_\infty) \right\}^{\prime\prime}$ is uniquely determined by (\ref{action_mod_aut_gr_on_generators}).
\subsubsection{The types of the factor $\left\{\pi_A^{(0)}(R_\infty) \right\}^{\prime\prime}$.} The next statement follows from Subsection \ref{mag}.
\begin{Th}
The following conditions are equivalent:
\begin{itemize}
  \item {\bf 1}. algebra $(I-E(0))\cdot\mathfrak{A}\cdot(I-E(0))$ is abelian;
  \item {\bf 2}. algebra $\left\{\pi_A^{(0)}(R_\infty) \right\}^{\prime\prime}$ is a factor of the type ${\rm II}$ or ${\rm I}$.
\end{itemize}
\end{Th}

\section{Notations}\label{section:notations}

\vskip 0.1cm\noindent
$\mathfrak S_n$ is the group of permutations on the set of $n$ elements $\{1,2,\ldots,n\}$
\vskip 0.1cm\noindent
$\mathfrak S_\infty=\cup_{n\in\mathbb N}\mathfrak S_n$ is the infinite symmetric group (group of all finite permutations of $\mathbb N$)
\vskip 0.1cm\noindent
$\epsilon_\mathbb{A}$ is  element of ${\rm Diag}_n$ such that $\mathcal{D}(\epsilon_\mathbb{A})=X_n\setminus\mathbb{A}$
(see page \pageref{epsilon_mathbb_A})
\vskip 0.1cm\noindent
$\mathcal D(r)$, $\mathcal I(r)$ are the domain and the range of $r\in R_\infty$
\vskip 0.1cm\noindent
$\mathfrak{m}(r,a)=\min\left\{ k\geq 1: r^k(a)\in  \mathcal{I}(r)\setminus\mathcal{D}(r)\right\}$
\vskip 0.1cm\noindent
$\mathfrak{O}_a=\left\{a,r(a),\ldots,r^{\mathfrak{m}(r,a)}(a)\right\}$ for $a\in\mathcal D(a)\setminus \mathcal I(a)$
\vskip 0.1cm\noindent
$r=\prod\limits_{a\in\mathcal{D}(r)\setminus\mathcal{I}(r)} \,^a\!q \cdot \prod\limits_{i}c_i^{(r)} \cdot d^{(r)}$ is the generalized cyclic decomposition for $r\in R_\infty$ (see Theorem \ref{decomposition_into_product_cycles})
\vskip 0.1cm\noindent
$\mathfrak{F}_\mathfrak{S}$ is the set of all $\mathfrak{S}_\infty$-central indecomposable states on $R_\infty$
(see page \pageref{F_S})
\vskip 0.1cm\noindent
$R_{n\infty}=\left\{r\in R_\infty:r(x)=x \text{ for all } x=1,2,\ldots,n \right\}$ (see page \pageref{R_n_infty})
\vskip 0.1cm\noindent
$\omega_\eta(A)=(A\eta,\eta)$ for $\eta$ in a Hilbert space
\vskip 0.1cm\noindent
$H_\pi(n)=\left\{ \eta\in H_\pi:\pi(s)\eta=\eta \text{ for all } s\in\mathfrak{S}_{n\infty}\right\}$ (page  \pageref{def_yP_n})
\vskip 0.1cm\noindent
 $\mathfrak{S}_\infty$-{\it admissible} representation (page \pageref{admissible})
\vskip 0.1cm\noindent
 $E(n)$ is the orthogonal projection onto $H_\pi(n)$ (page \pageref{E_n})
\vskip 0.1cm\noindent
$\widetilde{H}_\pi(n)$ is the closure  of the span of $\pi\left( R_n\cdot R_{n\infty}\right)H_\pi(n)$ (page  \pageref{widetilde_H_pi_n})
\vskip 0.1cm\noindent
$\widetilde{E}(n)$ is the orthogonal projection onto $\widetilde{H}_\pi(n)$  (page \pageref{E_n})
\vskip 0.1cm\noindent
$d(\pi)=\min\left\{ n:H_\pi(n)\neq 0\right\}$  (page  \pageref{depth_of_pi})
\vskip 0.1cm\noindent
$\mathcal{C}_\mathfrak{S}(r)$ is the set $\left\{ srs^{-1} \right\}_{s\in\mathfrak{S}_\infty}$, where $r\in R_\infty$
\vskip 0.1cm\noindent
$I$ stands for the identity operator in a Hilbert space
\vskip 0.1cm\noindent
$\mathcal O_k$ is the weak limit of the sequence of operators $\pi((k,n))$ when $n\to\infty$ (page \pageref{Okounkov_operator})
\vskip 0.1cm\noindent
$\mathfrak{A}_k$  is the $W^*$-algebra generated by $\mathcal{O}_k$  and $\pi\left(\epsilon_{\{k\}} \right)$
\vskip 0.1cm\noindent
$E_k(\lambda)$ is the spectral projection of $\mathcal O_k$ corresponding to an eigenvalue $\lambda$
\vskip 0.1cm\noindent
$P_k$ is the orthogonal projection onto the subspace $\left[\mathfrak{A}_k\left(1-E_k(0)\right)H_f \right]$
\vskip 0.1cm\noindent
$\mathfrak{P}_k=2I-P_k-E_k(0)$ and $Q^{(n)}=\prod\limits_{k=1}^n\mathfrak{P}_k$ (page \pageref{supp_comp})
\vskip 0.1cm\noindent
$\mathfrak{R}(A)\in\pi_f(R_\infty)'$ is defined by $\pi_f(r)\xi_f\mapsto \pi_f(r) A\xi_f,\; r\in R_\infty,$ (see Lemma \ref{operator_from_commutant} and page \pageref{cuclic_symmetrsc_group})
\vskip 0.1cm\noindent
$\mathbf{H}$ is a complex separable Hilbert space, $A\in B(\mathbf{H})$ is a self-adjoint operator, $\mathbf{q}\in B(\mathbf{H})$ is a one-dimensional orthogonal projection (see page \pageref{parameters_of_repr})
\vskip 0.1cm\noindent $\mathfrak{A}=\left\{ A,\mathbf{q}\right\}^{\prime\prime}$, $\widetilde{\mathbf{H}}$ is the subspace generated by $\left\{ \mathfrak{A}v,   v\in ({\rm Ker} A)^\perp \right\}$, $\mathbf{H}_{reg}=\mathbf{H}\ominus\widetilde{\mathbf{H}}$  (see page \pageref{parameters_of_repr})
\vskip 0.1cm\noindent $\mathbf{e}_{\lambda_k\lambda_l}=\left( c(\lambda_k)\cdot c(\lambda_l) \right)^{-1/2} \cdot E\left( \lambda_k \right)\cdot \mathbf{q}\cdot E\left( \lambda_l \right)$ is a one-dimensional partial isometry (see page \pageref{structure_of_A})
\vskip 0.1cm\noindent
$\left\{ e_{kl}\right\}\subset B(\mathbf{H})$ is a matrix unit, $\psi_k\left(b \right)={\rm Tr}\left(b|A| \right)+\left(1-{\rm
Tr}\left(|A| \right)\right) {\rm Tr}\left(b e_{n_kn_k}\right),$ $b\in B\left(\mathbf{H} \right)$ (see page \pageref{hp})
\vskip 0.1cm\noindent
$_1\psi_k\left(b_1\otimes b_2\otimes \ldots\otimes b_k
 \right)=\prod\limits_{j=1}^k\psi_j\left(b_j \right)$, $\mathcal{H}^\mathbf{q}_A$ is the completion of $B \left( \mathbf{H}\right)^{\otimes \infty}$ with respect to the form $(u,v)= _1\psi_\infty(u^*v)$ (see Section \ref{hp})
 \vskip 0.1cm\noindent
 $T$ is a $*$-homomorphism of $R_\infty$ to $B\left(\mathbf{H}\right)^{\otimes\infty}$ defined in Paragraph \ref{action}, $\mathfrak{L}_A$ ($\mathfrak{R}_A$) is the left (right) multiplication on $B\left(\mathbf{H}\right)^{\otimes\infty}$, 
 \vskip 0.1cm\noindent
 $\Pi_A$ is the $*$-representation of $R_\infty$ given by $\Pi_A(r)=\mathfrak{L}_A(T(r))$, $\pi_A^{(0)}$ is the restriction of $\Pi_A$ onto $\left[ \Pi_A\left( R_\infty\right)I\right]$ (see Paragraph \ref{action})

{}

\end{document}